\documentclass[reqno,10pt]{amsart}
\usepackage{amsfonts}
\usepackage{a4wide}

\usepackage{amsmath,amssymb,amsthm,amsfonts}
\usepackage{mathrsfs}
\usepackage{bbm}
\usepackage[colorlinks]{hyperref}

\usepackage{color}

\newtheorem{lemma}{Lemma}[section]
\newtheorem{theorem}{Theorem}[section]

\newtheorem{proposition}{Proposition}[section]
\newtheorem{remark}{Remark}[section]

\numberwithin{equation}{section}

\newcommand{\dis}{\displaystyle}

\newcommand{\R}{\mathbb{R}}

\newcommand{\Z}{\mathbb{Z}}

\renewcommand{\S}{\mathbb{S}}
\newcommand{\T}{\mathbb{T}}

\newcommand{\FP}{\mathbf{P}}

\newcommand{\FX}{\mathbf{X}}
\newcommand{\FY}{\mathbf{Y}}

\newcommand{\Fb}{\mathbf{b}}
\newcommand{\Fk}{\mathbf{k}}

\newcommand{\SL}{\mathscr{L}}

\newcommand{\CA}{\mathcal{A}}

\newcommand{\CC}{\mathcal{C}}
\newcommand{\CE}{\mathcal{E}}
\newcommand{\CF}{\mathcal{F}}
\newcommand{\CG}{\mathcal{G}}

\newcommand{\CI}{\mathcal{I}}
\newcommand{\CJ}{\mathcal{J}}
\newcommand{\CK}{\mathcal{K}}
\newcommand{\CL}{\mathcal{L}}
\newcommand{\CN}{\mathcal{N}}
\newcommand{\CM}{\mathcal{M}}

\newcommand{\CW}{\mathcal{W}}
\newcommand{\CP}{\mathcal{P}}
\newcommand{\CQ}{\mathcal{Q}}
\newcommand{\CS}{\mathcal{S}}
\newcommand{\CV}{\mathcal{V}}

\newcommand{\RC}{\mathscr{C}}
\newcommand{\SH}{\mathscr{H}}

\newcommand{\RF}{\mathfrak{F}}
\newcommand{\RG}{\mathfrak{G}}
\newcommand{\RI}{\mathfrak{I}}

\newcommand{\RX}{\mathfrak{X}}
\newcommand{\RL}{\mathfrak{L}}

\newcommand{\al}{\alpha}
\newcommand{\bet}{\beta}
\newcommand{\ga}{\gamma}
\newcommand{\om}{\omega}

\newcommand{\la}{\lambda}
\newcommand{\de}{\delta}
\newcommand{\si}{\sigma}
\newcommand{\pa}{\partial}
\newcommand{\ka}{\kappa}
\newcommand{\eps}{\epsilon}

\newcommand{\ta}{\theta}

\newcommand{\vps}{\varepsilon}

\newcommand{\Ga}{\Gamma}

\newcommand{\lag}{\langle}
\newcommand{\rag}{\rangle}

\newcommand{\eqdef}{\overset{\mbox{\tiny{def}}}{=}}

\usepackage{cite}

\usepackage{tikz}
\usepackage{tikz-cd}
\usepackage{tikz-3dplot}
\usepackage{amsmath,amssymb,amsthm,amsfonts,bm}
\usepackage{color}
\usepackage{pdfpages}
\usepackage{textcomp}
\usetikzlibrary{shapes.geometric}
\usetikzlibrary{decorations.pathreplacing,decorations.pathmorphing}
\usetikzlibrary{hobby}
\usepackage{graphicx}
\usepackage{wrapfig}
\usepackage{framed}
\usepackage{tcolorbox}
\usetikzlibrary{shadows}

\makeatletter
\@namedef{subjclassname@2020}{%
  \textup{2020} Mathematics Subject Classification}
\makeatother

\begin{document}

\title[Boltzmann equation for planar Couette flow]{The Boltzmann equation for plane Couette flow}

\author[R.-J. Duan]{Renjun Duan}
\address[RJD]{Department of Mathematics, The Chinese University of Hong Kong,
Shatin, Hong Kong, P.R.~China}
\email{rjduan@math.cuhk.edu.hk}

\author[S.-Q. Liu]{Shuangqian Liu}
\address[SQL]{School of Mathematics and Statistics, Central China Normal University, Wuhan 430079, P.R.~China;
	\\Hong Kong Institute for Advanced Study, City University of Hong Kong, Hong Kong, P.R~China}
\email{tsqliu@jnu.edu.cn}

\author[T. Yang]{Tong Yang}
\address[TY]{Department of Mathematics, City University of Hong Kong,
Hong Kong, P.R.~China}
\email{matyang@cityu.edu.hk}

\begin{abstract}
In the paper, we study the plane Couette flow of a rarefied gas between two parallel infinite plates at $y=\pm L$ moving relative to each other with opposite velocities $(\pm \alpha L,0,0)$ along the $x$-direction. Assuming that the stationary state takes the specific form of $F(y,v_x-\alpha y,v_y,v_z)$ with the $x$-component of the molecular velocity  sheared linearly along the $y$-direction, such steady flow is governed by a boundary value problem on a steady nonlinear Boltzmann equation driven by an external shear force under the homogeneous non-moving diffuse reflection boundary condition. In case of the Maxwell molecule collisions, we establish the existence of spatially inhomogeneous non-equilibrium stationary solutions to the steady problem for any small enough shear rate $\alpha>0$ via an elaborate perturbation approach using Caflisch's decomposition together with Guo's $L^\infty\cap L^2$ theory. The result indicates the polynomial tail at large velocities for the stationary distribution. Moreover,  the large time asymptotic stability of the stationary solution with an exponential convergence is also obtained
 and as a consequence the nonnegativity of the steady profile is justified.
\end{abstract}




\maketitle

%

\tableofcontents


\section{Intoduction}

The steady state of a rarefied gas between two parallel plates with the same temperatures and opposite velocities is one of the most fundamental boundary-value problems in kinetic theory, see the books of Kogan \cite{Ko}, Cercignani \cite{CerBook}, Garz\'o-Santos \cite{GaSa}, and Sone \cite{Sone07}. In particular, numerical analysis of the plane Couette flow of  rarefied gas on the basis of the nonlinear Boltzmann equation has been extensively
conducted in the physical literatures, cf. \cite{LL,OSA,Ro,STO,SY}. On the other hand, the mathematical study on this problem, even in the case when there is a temperature gap between two plates and a constant external force parallel to the boundaries, has been carried out by  Esposito-Lebowitz-Marra \cite{ELM-94,ELM-95} for the hydrodynamic description of the steady rarefied gas flow via the approximation of the corresponding compressible Navier-Stokes equations with no-slip boundary condition. The result in \cite{ELM-95} for  hard sphere model was later extended in \cite{DE-96} to the case of hard intermolecular potentials with Grad's angular cutoff as well as the Maxwell molecule case for which only the polynomial decay of the stationary solution for large velocities is obtained compared to the exponential decay  for hard sphere model. In addition, closely related to the plane Couette flow, the stationary Boltzmann equation for  rarefied gas in a Couette flow setting between two coaxial rotating cylinders was also studied extensively by Arkeryd-Nouri \cite{AN06,AN05} in the fluid dynamic regime, see also  a recent work \cite{AEMN} for further investigation of ghost effect induced by curvature.

The current study of the plane Couette flow with boundaries is motivated by the previous work \cite{DL-2020} by
 the first two authors for uniform shear flow via the Boltzmann equation without boundaries. We refer readers to \cite{BNV-2019,CerCo,G1,JNV-ARMA,T,TM} and references therein for more details of the topic on uniform shear flow. In particular,  in a recent significant progress \cite{BNV-2019}, Bobylev-Nota-Vel\'azquez  studied the self-similar asymptotics of solutions in large time for the Boltzmann equation with a general deformation of small strength and they also showed that the self-similar profile can have the finite polynomial moments of higher order as long as the deformation strength is smaller. In this paper, we will take into account the effect of shear force induced by the relative motion of the boundaries. We hope that the current study can shed some light on the relation between the Couette flow with boundary  and the uniform shear flow without boundary. A rigorous justification of the behavior of solutions in the limit $L\to\infty$ is left for future research.


To specify the problem, we consider the rarefied gas between two parallel infinite plates with the same uniform temperature $T_0>0$, one at $y=+L$ is moving with velocity $(U_+,0,0)$ and $U_+=\al L$ and the other at $y=-L$ is moving with velocity  $(U_-,0,0)$ and $U_-=-\al L$, where $\al>0$ is a  parameter for the shear rate, see Figure \ref{pict1} below.  Moreover, we assume that the gas molecules are of the Maxwellian type and reflected diffusively on the plates $y=\pm L$.

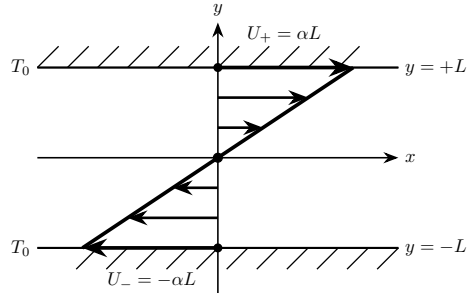
\begin{figure}[htbp]
	\centering
	\begin{tikzpicture}[scale=0.4,use Hobby shortcut,>=Stealth,interface/.style={postaction={draw,decorate,decoration={border,angle=45,amplitude=0.4cm,segment length=4mm}}},every node/.style={scale=0.7},box/.style ={draw,minimum height=0.9cm,minimum width=2.1cm,inner sep=4,align=center,fill=white,drop shadow={opacity=0.4,shadow xshift=.3ex,shadow yshift=-.3ex}}]
	\draw[line width=0.5pt,interface] (-6,3) .. (6,3);
	\draw[line width=0.5pt,interface] (6,-3)..(-6,-3);
	\filldraw (0,0) circle(0.15);
	\draw[line width=0.8pt] (-6,-3)node[left]{$T_0$}--(0,-3)--(6,-3)node[right]{$y=-L$};
	\draw[line width=0.8pt] (-6,3)node[left]{$T_0$}--(0,3)--(6,3)node[right]{${y=+L}$};
	\draw[line width=0.6pt,->] (-6,0)--(0,0)--(6,0)node[right]{$x$};
	\draw[line width=0.6pt,->] (0,-4.5)--(0,0)--(0,4.5)node[above]{$y$};
	\draw[line width=1.4pt,->] (0,3)--(4.5,3);
	\draw[line width=1.4pt,->] (0,-3)--(-4.5,-3);
	\draw[line width=1pt,->] (0,2)--(3,2);
	\draw[line width=1pt,->] (0,1)--(1.5,1);
	\draw[line width=1pt,->] (0,-1)--(-1.5,-1);
	\draw[line width=1pt,->] (0,-2)--(-3,-2);
	\draw[line width=1.4pt] (4.5,3)--(-4.5,-3);
	\draw (2.2,3.6) node[above]{$U_+=\alpha L$};
	\draw (-2.2,-3.6) node[below]{$U_-=-\alpha L$};
	\foreach \Point in {(0,3), (0,-3)}{\fill \Point circle(1.5mm);}
	\end{tikzpicture}
	\caption{Plane Couette flow}
	\label{pict1}
\end{figure}

Let the non-negative unknown function $F=F(y,v)\geq 0$ be the time-independent density distribution function of gas particles with velocity $v=(v_x,v_y,v_z)\in \R^3$ located at position $y\in (-L,L)$ along the vertical direction with the slab symmetry in the horizontal $(x,z)$-plane in space.  Then, the motion of such rarefied gas flow can be governed by the steady Boltzmann equation
\begin{equation}
\label{eqst}
v_y\pa_{y}F=\frac{1}{{\rm K\!n}}Q(F,F)
\end{equation}
subject to the diffuse reflection boundary conditions at $y=\pm L$ respectively, i.e.,
\begin{equation}
\label{eqstbdy}
F(\pm L,v)=\CM_{T_0}(v_x-U_\pm,v_y,v_z)\int_{v_y\lessgtr 0} F(\pm L,v)|v_2|\,dv\quad \text{for }v_y\gtrless 0,
\end{equation}
as well as a given total mass
\begin{equation}
\label{tm}
\frac{1}{2L}\int_{-L}^L\int_{\R^3}F(y,v)\,dvdy=M
\end{equation}
for some positive constant $M>0$. Here,
 the non-dimensional parameter ${\rm K\!n}>0$ is the Knudsen number given by the ratio of the mean free path to the typical length and $\CM_{T_0}=\CM_{T_0}(v)$ associated with the uniform wall temperature $T_0$ at $y=\pm L$ is a global Maxwellian of the form
\begin{equation}
\CM_{T_0}(v)=\frac{1}{2\pi T_0^2} e^{-\frac{|v_x|^2+|v_y|^2+|v_z|^2}{2T_0}},\quad v=(v_x,v_y,v_z)\in \R^3.\notag
\end{equation}
For the Maxwell molecule model, the collision operator operator $Q$, which is bilinear and acts only on velocity variable, takes the form of
\begin{equation}\label{def.Q}
\begin{split}
Q(F_1,F_2)(v)=&\int_{\R^3}\int_{\S^2}B_0(\cos\ta)[F_1(v'_\ast)F_2(v')-F_1(v_\ast)F_2(v)]\,d\omega dv_{\ast},
\end{split}
\end{equation}
where the velocity pairs $(v_\ast,v)$ and $(v_\ast',v')$ satisfy the relation
\begin{equation}
\label{v.re}
v'_\ast=v_\ast-[(v_\ast-v)\cdot\omega]\om,\quad v'=v+[(v_\ast-v)\cdot\omega]\om,
\end{equation}
denoting the $\omega$-representation  according to conservation of momentum and energy in the elastic collision, i.e.,
$v_\ast+v=v_\ast'+v'$ and 
$|v_\ast|^2+|v|^2=|v'_\ast|^2+|v'|^2$, respectively.
Throughout the paper, we assume that the collision kernel $B_0(\cos\theta)$ with $\cos\theta=(v-v_\ast)\cdot \om/|v-v_\ast|$, depending only on the angle $\theta$ between the relative velocity $v-v_\ast$ and $\omega$, satisfies the Grad's angular cutoff assumption
\begin{equation}\label{coa}
0\leq B_0(\cos\theta)\leq C |\cos\theta|,
\end{equation}
for a generic constant $C>0$.

In the paper, for the boundary-value problem \eqref{eqst}, \eqref{eqstbdy} and \eqref{tm} with  finite Knudsen number, we  look for stationary solutions of the following specific form
\begin{equation}
\label{def.fst}
F_{st}(y,v_x-\alpha y,v_y,v_z),
\end{equation}
where the horizontal molecular velocity $v_x-\alpha y$ is sheared linearly along the $y$-direction. After plugging  \eqref{def.fst} into \eqref{eqst}, \eqref{eqstbdy}, \eqref{tm} and normalizing $L$, $M$ and $T_0$ to be one for brevity, the stationary distribution function $F_{st}$ is determined by the following boundary-value problem
\begin{equation}\label{Fst}
\left\{\begin{split}
&v_y\pa_{y}F_{st}-\al v_y\pa_{v_x}F_{st}=Q(F_{st},F_{st}),\ y\in(-1,1),\ v=(v_x,v_y,v_z)\in\R^3,\\
&F_{st}(\pm1,v)|_{v_y\lessgtr0}=\sqrt{2\pi}\mu\dis{\int_{v_y\gtrless0}}F_{st}(\pm1,v)|v_y|dv,\ v\in\R^3,\\
&\frac{1}{2}\int_{-1}^1\int_{\R^3}F_{st}(y,v)\,dvdy=1,
\end{split}\right.
\end{equation}
with the global Maxwellian $\mu=(2\pi)^{-3/2}e^{-|v|^2/2}$. This paper aims to establish the existence of solutions to the above boundary value problem \eqref{Fst} for any small enough shear rate $\alpha>0$, as well as its large time asymptotic stability.

To solve \eqref{Fst}, we will apply  the perturbation approach by taking the shear rate as a small parameter.  If $\al=0$,  $F_{st}=\mu$ is the unique equilibrium solution to the boundary value problem \eqref{Fst}. However, for $\alpha>0$, the external shear force drives the rarefied gas far from the equilibrium. Precisely, we set
\begin{equation}\label{Fst.ex}
F_{st}=\mu+\sqrt{\mu}\{\al G_1+\al^2G_R\},
\end{equation}
with
\begin{equation}
\label{Fst.ex.mc}
\int_{-1}^1\int_{\R^3}\sqrt{\mu}G_1\,dvdy=\int_{-1}^1\int_{\R^3}\sqrt{\mu}G_R\,dvdy=0.
\end{equation}
By plugging \eqref{Fst.ex} into \eqref{Fst} and comparing coefficients of the  equation in the order of $\al$, we obtain the equation for $G_1$
\begin{align}\label{G1}
v_y\pa_yG_1+LG_1=-v_xv_y\sqrt{\mu},
\end{align}
with  boundary condition
\begin{equation}
\label{G1drbc}
G_1(\pm1,v)|_{v_y\lessgtr0}=\sqrt{2\pi\mu}\dis{\int_{v_y\gtrless0}}\sqrt{\mu}G_1(\pm1,v)|v_y|\,dv,
\end{equation}
and the equation for the remainder $G_R$
\begin{align}\label{Gr}
v_y&\pa_yG_R-\al v_y\pa_{v_x}G_R+\frac{\al}{2}v_xv_yG_{R}+LG_R\notag\\&=v_y\pa_{v_x}G_1-\frac{1}{2}v_xv_yG_{1}
+\Ga(G_1,G_1)+\al\{\Ga(G_R,G_1)+\Ga(G_1,G_R)\}+\al^2\Ga(G_R,G_R),
\end{align}
with boundary condition
\begin{align}\label{Grbd}
G_R(\pm1,v)|_{v_y\lessgtr0}=\sqrt{2\pi\mu}\dis{\int_{v_y\gtrless0}}\sqrt{\mu}G_R(\pm1,v)|v_y|\,dv.
\end{align}
Here, the linear and nonlinear collision operator $L$ and $\Ga$ are given by
$$
Lf=-\mu^{-\frac{1}{2}}\{Q(\mu,\sqrt{\mu}f)+Q(\sqrt{\mu}f,\mu)\},
$$
and
$$
\Ga(f,g)=\mu^{-\frac{1}{2}}\{Q(\sqrt{\mu}f,\sqrt{\mu}g)+Q(\sqrt{\mu}g,\sqrt{\mu}f)\},
$$
respectively. Properties of these two operators will be presented in Section \ref{sec2}.
Note that to solve $G_1$, both \eqref{G1} and \eqref{G1drbc} with the restriction $\int_{-1}^1\int_{\R^3}\sqrt{\mu}G_1\,dvdy=0$ are invariant under the transformation $G_1(y,v)\to -G_1(y,-v_x,v_y,v_z)$. Thus, if the solution is unique,
$G_1$ is odd in $v_x$, namely,
\begin{equation}
\label{G1odd}
G_1(y,v)= -G_1(y,-v_x,v_y,v_z),\quad -1\leq y\leq 1,\ v=(v_x,v_y,v_z)\in\R^3.
\end{equation}
Hence, the diffuse reflection boundary condition \eqref{G1drbc} for $G_1$ can be reduced to the homogeneous inflow boundary condition
\begin{align}\label{G1bd}
G_1(\pm1,v)|_{v_y\lessgtr0}=0.
\end{align}

The first result in the paper for the existence of the  Couette flow problem is stated as follows. To the end, we use
a velocity weight function
\begin{equation}
\label{def.vwwq}
w_q=w_q(v):=(1+|v|^2)^q
\end{equation}
with an integer $q>0$.

\begin{theorem}\label{st.sol}
Assume that the Boltzmann collision kernel is of the Maxwell molecule type \eqref{coa}. Then, the boundary value problem \eqref{Fst} admits a unique steady solution $F_{st}=F_{st}(y,v)\geq 0$ of the form \eqref{Fst.ex} satisfying \eqref{Fst.ex.mc} and the following estimates on $G_1$ and $G_R$, respectively.

\begin{itemize}
  \item[(i)] The first-order correction $G_1=G_1(y,v)$, uniquely solved by the boundary value problem \eqref{G1} and \eqref{G1bd}, satisfies \eqref{G1odd} and  for any integers $m\geq 0$ and $q\geq 0$,  \begin{align}\label{st.sole1}
\|w_{q}\pa_{v_x}^mG_1\|_{L^\infty}\leq \tilde{C}_1,
\end{align}
where $\tilde{C}_1>0$ is a constant depending only on $m$ and $q$.

  \item[(ii)] The remainder $G_R=G_R(y,v)$, uniquely solved by the boundary value problem \eqref{Gr} and \eqref{Grbd},  satisfies that there is an integer $q_0>0$ such that for any integer $q\geq q_0$, there is $\al_0=\al_0(q)>0$ depending on $q$ such that for any $\al\in (0,\al_0)$ and any integer $m\geq 0$  $\widetilde{G}_R:=\sqrt{\mu}G_R$ satisfies that
\begin{equation}
\label{st.sole2}
\|w_q\pa_{v_x}^m\widetilde{G}_R\|_{L^\infty}\leq \tilde{C}_{m,q},
\end{equation}
where $\tilde{C}_{m,q}>0$ is a constant depending only on $m$ and $q$ but independent of $\al$.
 \end{itemize}

\end{theorem}

Some remarks on Theorem \ref{st.sol} are given as follows.

\begin{remark}
The steady solution $F_{st}$ to the boundary value problem \eqref{Fst} is essentially constructed in the regime where the collision is dominated and the shearing effect is weak. By \eqref{st.sole1} and \eqref{st.sole2}, the steady solution takes the form of
\begin{equation}
\label{exp2o}
F_{st}=\mu+\alpha \sqrt{\mu}G_1+O(1)\alpha^2
\end{equation}
with the remainder of the second order decaying in large velocities only polynomially. The order of the polynomial decay can be arbitrarily large as long as the shear rate is sufficiently small. It generally holds $\alpha_0(q) \to 0$ as $q\to\infty$, and in particular,  one may take
$$
\alpha_0(q)=\frac{\nu_0}{8q}
$$
as shown in the proof.  The result is consistent with the one in \cite{DL-2020} for uniform shear flow without boundaries in the spatially homogeneous setting.
\end{remark}

\begin{remark}
Without using the odd-in-$v_x$ property as in \eqref{G1odd}, the existence of $G_1(y,v)$ to the BVP \eqref{G1} under the diffuse reflection boundary condition \eqref{G1drbc} also can be established by the same approach as for treating the  remainder $G_R$. Here, we take this formulation only for brevity of presentation because the proof for the homogeneous inflow boundary is relatively easier than that for the diffuse reflection boundary.
\end{remark}

\begin{remark}
We notice that it is necessary to deal with the $v_x$-derivative estimates due to the appearance of the shear force term $v_y\pa_{v_x} F_{st}$, in particular, the term $v_y\pa_{v_x}G_1$ becomes a source term in the equation \eqref{Gr} for $G_R$. We emphasize that although one can obtain the derivative estimates as in \eqref{st.sole1} and \eqref{st.sole2} in $v_x$, it is impossible to obtain similar estimate  on derivative in  $v_y$ because  $G_1(y,v)$ is discontinuous at $v_y=0$, see \eqref{def.G1noK} for an explicit form of $G_1$ when the non-local collision term is omitted.
\end{remark}

To establish the {nonnegativity} of the stationary profile $F_{st}(y,v)$, we  further  study the following initial boundary value problem of the Boltzmann equation with a shear force
\begin{align}\label{F}
\left\{\begin{array}{rll}
&\pa_tF+v_y\pa_{y}F-\al v_y\pa_{v_x}F=Q(F,F),\ t>0,\ y\in(-1,1),\ v=(v_x,v_y,v_z)\in\R^3,\\[2mm]
&F(0,y,v)=F_0(y,v),\ y\in(-1,1),\ v\in\R^3,\\[2mm]
&F(t,\pm1,v)|_{v_y\lessgtr0}=\sqrt{2\pi}\mu\dis{\int_{v_y\gtrless0}}F(t,\pm1,v)|v_y|dv,\ t\geq0,\ v\in\R^3.
\end{array}\right.
\end{align}
One may expect that the solution of the time-dependent  problem \eqref{F} tends in large time toward that of the steady problem \eqref{Fst}. For this, the second result is concerned with the large time asymptotic stability of the stationary solution $F_{st}$ which  gives the nonnegativity  of $F_{st}$.

\begin{theorem}\label{ust.mth}
Let $F_{st}(y,v)$ be the steady state obtained in Theorem \ref{st.sol} corresponding to a shear rate $\alpha\in (0,\alpha_0)$. There are constants $\varepsilon_0>0$, $\la_0>0$ and $C>0$, independent of $\alpha$, such that if initial data $F_0(y,v)\geq0$ satisfy
\begin{equation}
\left\|w_q\left[F_0(y,v)-F_{st}(y,v))\right]\right\|_{L^\infty}\leq\varepsilon_0\notag
\end{equation}
with
\begin{align}\label{id.cons}
\int_{-1}^1\int_{\R^3}[F_0(y,v)-F_{st}(y,v)]\,dvdy=0,
\end{align}
then the initial boundary value problem \eqref{F} admits a unique solution $F(t,y,v)\geq0$ satisfying the following decay estimate:
\begin{equation}\label{lif.decay}
\left\|w_q\left[F(t,y,v)-F_{st}(y,v)\right]\right\|_{L^\infty}
\leq C e^{-\la_0 t}\left\|w_q\left[F_0(y,v)-F_{st}(y,v)\right]\right\|_{L^\infty},
\end{equation}
for any $t\geq 0$.
\end{theorem}

\begin{remark}
Thanks to Theorem \ref{st.sol}, the expansion \eqref{exp2o} for the steady state $F_{st}(y,v)$ is uniform in all $\alpha\in (0,\alpha_0)$ when the large enough integer $q$  is chosen and hence $\alpha_0=\alpha_0(q)>0$ is fixed. Thus, the exponential time decay estimate \eqref{lif.decay} also holds uniformly for any  $\alpha\in (0,\alpha_0)$, in particular, $C$ and $\lambda_0$ are independent of $\al$. As $\al\to 0$, we are able to recover the exponential convergence of the solution $F(t,y,v)$ to the global Maxwellian $\mu$ in  $L^\infty$-norm weighted by the polynomial velocity weight $w_q(v)$.
\end{remark}

In what follows we present key points and strategy in the proof of the main results stated above. As pointed out in a recent nice survey by Esposito-Marra \cite{EM-JSP}, stationary non-equilibrium solutions to the Boltzmann equation, despite their relevance in applications, are much less studied than time-dependent solutions, and no general existence theory is available, due to technical difficulties. Readers may refer to \cite{EM-JSP} and references therein for a thorough review on this subject. As for the Boltzmann equation on the plane Couette flow, \cite{ELM-95} and \cite{DE-96}  mentioned before
 seem to be the only   mathematical works on the fluid dynamic approximation solutions in the steady case for small Knudsen number. But it remains unsolved how to justify the large time asymptotics toward the stationary solution for the time-dependent problem in the same setting of the fluid limit.  In this paper, motivated by \cite{DL-2020}, instead of constructing the fluid dynamic approximation solutions, we focus on the existence and dynamical stability of the plane Couette flow with the  finite Knudsen number  for both the steady and unsteady problems.

First of all, for the original Couette flow problem \eqref{eqst}, \eqref{eqstbdy} and \eqref{tm}, we note that a direct perturbation approach by linearization of the boundary condition in $\alpha$ in terms of the techniques in \cite{DE-96,ELM-94,ELM-95} or \cite{EGKM-13} can be applied to prove the existence of stationary solutions, because the inhomogeneous data appear only on the tangent $(x,z)$-plane. The solution thus obtained has the structure around global Maxwellians of the form
\begin{equation}
F(y,v)=\mu(v)+\sqrt{\mu(v)} (\alpha g_1+\alpha^2g_2+\cdots)\notag
\end{equation}
corresponding to the linearization of the wall Maxwellians at $y=\pm L$
\begin{equation}
\mu(v_x\pm \alpha L,v_y,v_z)=\mu(v)+(\pm \alpha L) \mu_1(v) + (\pm\alpha L)^2\mu_2 (v)+\cdots.\notag
\end{equation}
On the other hand, in the formulation used in this paper, we rather look for the solution of the specific structure {\eqref{def.fst}}, and hence the problem can be reduced to solve \eqref{Fst} for the Boltzmann equation driven by an external shear force under the homogeneous non-moving diffuse reflection boundary condition. This means that the solution to the Couette flow problem \eqref{eqst}, \eqref{eqstbdy} and \eqref{tm} is established around the local Maxwellian $\mu(v_x-\alpha y,v_y,v_z)$ instead of the global Maxwellian $\mu$ such that the kinetic diffusive reflection boundary condition \eqref{eqstbdy} is satisfied for the background solution $\mu(v_x-\alpha y,v_y,v_z)$.  In addition, as mentioned before, it seems more convenient to use the formulation with shear forces than the original one driven by the relative motion of boundaries in order to understand the asymptotic behavior of solutions in the limit $L\to \infty$, that is, how the Couette flow with boundaries converges to a shear flow without boundary that is closely related to what has been studied in the previous works \cite{DL-2020} for uniform shear flow in the spatially homogeneous setting.

We also comment on  the boundary value problem \eqref{G1} and \eqref{G1drbc} for determing the first order correction term $G_1(y,v)$. Notice that the inhomogeneous source term $-v_xv_y\sqrt{\mu}$ in \eqref{G1} does not satisfy the boundary condition \eqref{G1drbc}, so a space-dependent non-trivial solution is induced. If the boundary condition is omitted and only the spatially homogeneous equation is considered, the corresponding solution can be written as
\begin{equation}
\label{g1usf}
L^{-1}(-v_xv_y\sqrt{\mu})=-\frac{1}{2b_0} v_xv_y\sqrt{\mu}
\end{equation}
with the positive constant
$b_0:=3\pi \int_{-1}^1 B_0(z)z^2(1-z^2)\,dz$.
The form \eqref{g1usf} is then consistent with  the uniform shear flow in \cite{DL-2020}. To solve the boundary value problem \eqref{G1} and \eqref{G1drbc}, the same approach as for treating the remainder $G_R$ can be applied. However, in order to simplify the proof, we have made use of an additional property \eqref{G1odd} to reduce the diffusive reflection boundary condition \eqref{G1drbc} to the homogeneous inflow boundary condition \eqref{G1bd}. To treat \eqref{G1} and \eqref{G1bd}, we develop a direct $L^\infty$-$L^2$ method without using the stochastic cycles as in \cite{Guo-2010}. In particular, thanks to the  splitting $L=\nu_0-K$, if the non-local term $KG_1$ is omitted, the solution to the boundary value problem
\begin{equation}
v_y\pa_{y}G_1+\nu_0G_1=\RF,\quad G_1(\pm1,v)|_{v_y\lessgtr0}=0,\notag
\end{equation}
can be explicitly expressed as
\begin{equation}
G_1(y,v)={\bf 1}_{v_y>0}\int_{-1}^ye^{-\frac{\nu_0(y-y')}{v_y}}v_y^{-1}\RF(y',v)dy'+{\bf 1}_{v_y<0}\int_{y}^1e^{-\frac{\nu_0(y-y')}{v_y}}v_y^{-1}\RF(y',v)dy'.\notag
\end{equation}
Moreover, we use the bootstrap argument as in \cite{DHWZ-19} to treat the following problem with a parameter $\si\in [0,1]$:
\begin{equation}
v_y\pa_yG_1+\nu_0G_1=\si KG_1+\RF,\quad G_1(\pm1,v)|_{v_y\lessgtr0}=0.\notag
\end{equation}
With the solvability starting from $\si=0$, we are able to iteratively solve the above boundary value problem for $\si$ over the intervals $[0,\si_\ast]$, $[\si_\ast,2\si_\ast]$, and so on, where $\si_\ast>0$ is small enough such that $\si_\ast KG_1$ can be regarded as a source term in the $L^\infty$ estimation. Therefore, in the end, the original problem corresponding to $\si=1$ can be solved. In this procedure, the uniform $L^\infty$ estimate can be obtained through the interplay with the $L^2$ estimates in terms of the Guo's technique in \cite{Guo-2010}. Here, we have omitted the discussions on the mass conservation \eqref{Fst.ex.mc} for $G_1$. In fact, inspired by \cite{Guo-2010}, an extra damping term $\eps G_1$ with the vanishing parameter $\eps>0$ has to be used, cf.  Section \ref{sp.sec} for details.

We now discuss some key points about estimating the remainder $G_R$ for the existence of solutions to the boundary value problem \eqref{Gr} and \eqref{Grbd}. The direct $L^\infty$-$L^2$ approach is no longer available because the linear term $\frac{1}{2}\alpha v_xv_yG_R$ can not be controlled in the large velocity regime. Notice that this term arises from the action of the shear force on the exponential weight function $\sqrt{\mu}$ in the perturbation. To overcome it, as in \cite{DL-2020}, we apply the Caflisch's decomposition
\begin{equation}
\sqrt{\mu}G_R=G_{R,1}+\sqrt{\mu}G_{R,2}\notag
\end{equation}
and $G_{R,1}$ and $G_{R,2}$ satisfy the coupled boundary value problems
\begin{equation}
\label{bvpgr1}
\left\{\begin{aligned}
&v_y\pa_yG_{R,1}-\al v_y\pa_{v_x}G_{R,1}+\nu_0 G_{R,1}=\chi_{M}\CK G_{R,1}-\frac{1}{2}\al\sqrt{\mu}v_xv_yG_{R,2}+\CF_1,\\
&G_{R,1}(\pm1,v)|_{v_y\lessgtr0}=0,
\end{aligned}\right.
\end{equation}
and
\begin{equation}
\label{bvpgr2}
\left\{\begin{aligned}
&v_y\pa_yG_{R,2}-\al v_y\pa_{v_x}G_{R,2}+LG_{R,2}=(1-\chi_{M})\mu^{-\frac{1}{2}}\CK G_{R,1}+\CF_2,\\
&G_{R,2}(\pm1,v)|_{v_y\lessgtr0}=\sqrt{2\pi \mu}\dis{\int_{v_y\gtrless0}}\sqrt{\mu}G_{R}(\pm1,v)|v_y|dv,
\end{aligned}\right.
\end{equation}
respectively. Then, in \eqref{bvpgr1}, the term $-\frac{1}{2}\al\sqrt{\mu}v_xv_yG_{R,2}$ can be controlled due to the appearance of $\sqrt{\mu}$. Here, since the operator norm of $\CK$ may not be small, the term $\chi_{M}\CK G_{R,1}$ over the large velocity regime can be viewed as a source in  \eqref{bvpgr1} for $G_{R,1}$, while the complementary term $(1-\chi_{M})\mu^{-\frac{1}{2}}\CK G_{R,1}$ is taken as a source in  \eqref{bvpgr2} for $G_{R,2}$. A crucial observation inspired by \cite{AEP-87} in estimating $G_{R,1}$ is that the norm of the weighted operator $w_q\chi_{M}\CK$ on $L^\infty_v$ with the polynomial velocity weight $w_q=(1+|v|^2)^q$ can be arbitrarily small as long as $M$ and $q$ are chosen sufficiently large, see Lemma \ref{CK}. Notice that Lemma \ref{CK} holds  only for the Maxwell molecule potential as shown in the proof. Compared to the previous work \cite{DL-2020} for uniform shear flow, it is more complicated to solve the coupling steady boundary value problems \eqref{bvpgr1} and \eqref{bvpgr2} in a bounded domain. We now list the main steps in the proof.
\begin{itemize}
  \item {\it Step 1.} We first modify the coupled boundary value problems with two parameters $\eps>0$ being small enough and $0\leq \si\leq 1$, see \eqref{pals}, and  obtain the a priori estimates uniform in $\eps$ and $\si$ in the $L^\infty$ framework, see Lemma \ref{lifpri} {and the proof for Proposition \ref{Gr.lem}}. For the proof of Lemma \ref{lifpri}, we apply  Guo's approach in  \cite{Guo-2010} to
  the  shear flow problem in a slab. In particular, we introduce the mild formulation \eqref{H2m} to treat the diffuse  boundary condition with the help of Lemma \ref{k.cyc}, and reprove  Ukai's trace theorem in Lemma \ref{ukai} for  the $L^2$ estimates.
  \item {\it Step 2.} Similar to solving the first order correction term $G_1$, we design an explicit procedure to solve the parametrized boundary value problem \eqref{pals} iteratively for $\si\in [0,1]$ from $\si=0$ to $\si=1$ for any fixed $\eps>0$, see Lemma \ref{ex.pals}.  Notice that the problem for $\si=0$ is reduced to the one without the nonlocal collision terms under the homogeneous inflow boundary condition so that the characteristic method can be directly applied.
  \item  {\it Step 3.} We study the limit $\eps\to 0$ to obtain the desired solution, see Subsection \ref{sub5.4} for details. The key point is to obtain the macroscopic estimates in order to bound the $L^2$ norm of $G_{R,2}$ in terms of the $L^\infty$ norm of $G_{R,1}$. We apply the dual argument developed first in \cite{EGKM-13}. Note that it is delicate to deduce these estimates  to be uniform for any small parameter $\eps>0$.
\end{itemize}

With the existence of stationary solution $F_{st}$,  the asymptotic stability of  the perturbation $F=F_{st}+\sqrt{\mu}f$ as \eqref{f.usp} is considered in  the reformulated IBVP as \eqref{f}. Technically, we follow the same strategy as for treating the steady problem.
Precisely, we also use the decomposition $\sqrt{\mu}f=f_1+\sqrt{\mu}f_2$ with $f_1$, $f_2$ satisfying the coupled IBVPs \eqref{f1}, \eqref{f1bd} and \eqref{f2}, \eqref{f2bd}, respectively. In order to treat initial data with only the polynomial velocity weight, we  set $f_2(0,y,v)\equiv 0$ and  the boundary conditions on $f_1$ and $f_2$ both as diffuse reflections which are slightly different from \eqref{bvpgr1} and \eqref{bvpgr2} in the steady problem. Moreover, in contrast with the steady case, we need to construct suitable temporal energy functionals so as to close the a priori estimates. In particular, the energy functional for the second component $f_2$ in the Caflisch's decomposition is complicated, because there is a subtle interplay with $f_1$. For this, we make use of the linear combination of estimates for the two functionals, where the smallness of the shear rate $\alpha$ and finiteness of the domain play an important role. Specifically, we obtain estimates \eqref{g1.es} and \eqref{g2sum3} for the weighted $L^\infty$ norms. To treat $L^2$ estimates on the right hand side of \eqref{g2sum3}, we construct another functional $\CE_{int}(t)$ in Lemma \ref{abc.lem}, see \eqref{Edef}, to capture the macroscopic dissipation, and conclude the desired estimates \eqref{ttf2.es} and hence \eqref{f2.l2fl}.

Finally, we remark that there have been extensive studies on  the stability of shear flow in the multi-dimensional space domain in the context of fluid dynamic equations, cf.~\cite{ScHe}, in particular, we mention important  contributions to
the mathematical theories in \cite{BM-15,BMV-16,BGM-17} by Bedrossian et.~al. for either an infinite 2D channel domain $\T_{x}\times \R_{y}$ or an infinite 3D channel domain $\T_{x}\times \R_{y}\times \T_z$, and an interesting work by Ionescu-Jia \cite{IJ} for the asymptotic stability of the Couette flow for the 2D Euler equations in the 2D finite channel domain $\T_x\times [0, 1]$ with the zero normal velocities at two boundary planes $y=0,1$, see also the nice survey \cite{BGM-BAMS} and references therein. In fact, in comparison with the 1D problem \eqref{Fst} under consideration, it would be more interesting to study the existence and asymptotic stability of stationary solutions in the multi-dimensional setting corresponding to those works on fluid dynamic equations. Moreover, it is also challenging to study the fluid dynamic limit for these problems as in \cite{DE-96,ELM-94,ELM-95} when the vanishing  Knudsen number is taken into account. We expect that this paper together with  \cite{DL-2020} can shed some light on the future investigation on the above problems.

The rest of this paper is organized as follows. In Section \ref{sec2}, we give some basic estimates on the linearized and nonlinear collision operators. In particular, we obtain Lemma \ref{CK} which is crucially used to obtain the smallness of the nonlocal operator $\CK$ for large velocity. In Section \ref{sec3}, we revisit  Ukai's trace theorem in both the steady and time-dependent cases for the transport operator with the shear force in the 1D setting under consideration. In Section \ref{sp.sec} and Section \ref{sec.spr}, we establish estimates on the first order correction  $G_1$ and the remainder $G_R$, respectively, and hence complete the proof of Theorem \ref{st.sol} without showing nonnegativity of the stationary
solution. Then we study the time-dependent problem for the local in time existence in Section \ref{loc.ex} and the exponential time asymptotic stability of the stationary solution in Section \ref{ust-pro} so that  the nonnegativity of stationary solution follows.
The appendix Section \ref{app-sec} includes some estimates on the boundary product measure  when there are multiple bounces induced by  the diffuse boundary condition.


\smallskip
\noindent{\it Notations.} We list some notations and norms used in the paper. Throughout this paper,  $C$ denotes some generic positive (generally large) constant and $\la$ denote some generic positive (generally small) constant.
 $D\lesssim E$ means that  there is a generic constant $C>0$
such that $D\leq CE$. $D\sim E$
means $D\lesssim E$ and $E\lesssim D$. ${\bf 1}_A$ indicates the characteristic function on the set $A$.
 We denote $\Vert \cdot \Vert $ the $L^{2}((-1,1)
\times \R^{3})-$norm or the $L^{2}(-1,1)-$norm or $L^{2}(\R^3)-$norm.
Sometimes without any confusion, we use $\|\cdot \|_{L^\infty }$ to denote either the $L^{\infty }([-1,1]
\times \R^{3})-$norm or the $L^{\infty }(\R^3)-$norm. Moreover, 
$(\cdot,\cdot)$ denotes the $L^{2}$ inner product in
$(-1,1)\times {\R}^{3}$  with
the $L^{2}$ norm $\|\cdot\|$ and $\langle\cdot\rangle$ denotes the $L^{2}$ inner product in $\R^3_v$. We denote by $\ga_+=\{(1,v)|v\in\R^3,v_y>0\}\cup\{(-1,v)|v\in\R^3,v_y<0\}$ the outgoing set, by $\ga_-=\{(1,v)|v\in\R^3,v_y<0\}\cup\{(-1,v)|v\in\R^3,v_y>0\}$ the incoming set, and by $\ga_0=\{(\pm1,v)|v\in\R^3,v_y=0\}$ the grazing set. Furthermore $|f|_{2,\pm}= |f \mathbf{1}_{\gamma_{\pm}}|_2$
represent the $L^2$ norm of $f(y,v)$ at the boundary $y=\pm1.$
 Finally, we define
\begin{equation*}
P_{\gamma }f(\pm1,v)=\sqrt{\mu (v)}\int_{n(\pm1)\cdot v^{\prime }>0}f(x,v^{\prime
})\sqrt{\mu (v^{\prime })}(n(\pm1)\cdot v^{\prime })dv^{\prime },
\label{pgamma}
\end{equation*}%
where $n(\pm1)=(0,\pm1,0).$
One sees that $P_{\gamma }f$ defined on $\{\pm1\}\times\R^3$, 
is an $L_{v}^{2}$-projection with respect to the measure $|v_y|\sqrt{\mu (v)}dv$ for
any  function $f$ defined on $\gamma _{+}$.

\section{Basic estimates}\label{sec2}
In this section we summarize some basic estimates to be used in the following sections.
Let us first give some elementary estimates for the linearized collision operator $L$ and nonlinear collision operator $\Ga$, defined by
\begin{align}\label{def-L}
Lg=-\mu^{-1/2}\left\{Q(\mu,\sqrt{\mu}g)+Q(\sqrt{\mu}g,\mu)\right\},
\end{align}
and
\begin{align}\label{def.Ga}
\Gamma (f,g)=\mu^{-1/2} Q(\sqrt{\mu}f,\sqrt{\mu}g)=\int_{\R^3}\int_{\S^2}B_0\mu^{1/2}(v_\ast)[f(v_\ast')g(v')-f(v_\ast)g(v)]\,d\om dv_\ast,
\end{align}
respectively. It is known that
$$
Lf=\nu f-Kf
$$
with
\begin{align}\label{sp.L}
\left\{\begin{array}{rll}
&\nu=\dis{\int_{\R^3}\int_{\S^2}}B_0(\cos \ta)\mu(v_\ast)\, d\om dv_\ast=\nu_0,\\[2mm]
&Kf=\mu^{-\frac{1}{2}}\left\{Q(\mu^{\frac{1}{2}}f,\mu)+Q_{\textrm{gain}}(\mu,\mu^{\frac{1}{2}}f)\right\},
\end{array}\right.
\end{align}
where $Q_{\textrm{gain}}$ denotes the positive part of $Q$ in \eqref{def.Q}. Note that $\nu_0$ is a positive constant in the case of Maxwell molecule collision. The kernel of $L$, denoted as $\ker L$, is a five-dimensional space spanned by
$$
\{1,v,|v|^2-3\}\sqrt{\mu}:= \{\phi_i\}_{i=1}^5.
$$
Define a projection from $L^2$ to $\ker L$ by
\begin{align*}
\FP_0 g=\left\{a_g+\Fb_g\cdot v+(|v|^2-3)c_g\right\}\sqrt{\mu}
\end{align*}
for $g\in L^2$, and correspondingly denote the operator $\FP_1$ by $\FP_1g=g-\FP_0 g$, which is orthogonal to $\FP_0$ in $L^2$.

It is also convenient to define
\begin{align}\notag
\CL f=-\left\{Q(f,\mu)+Q(\mu,f)\right\}=\nu f-\CK f,
\end{align}
with
\begin{equation}\label{sp.cL}
\nu f=\nu_0f,\quad \CK f=Q(f,\mu)+Q_{\textrm{gain}}(\mu,f)=\sqrt{\mu}K(\frac{f}{\sqrt{\mu}}),
\end{equation}
according to \eqref{sp.L}.

The following lemma is concerned with the integral operator $K$ given by \eqref{sp.L}, and its proof in case of the hard sphere model was given by \cite[Lemma 3, pp.727]{Guo-2010}. Recall \eqref{def.vwwq} for the polynomial velocity weight $w_q$.

\begin{lemma}\label{Kop}
Let $K$ be defined as in \eqref{sp.L}, then it holds
\begin{align}\notag
Kf(v)=\int_{\R^3}\Fk(v,v_\ast)f(v_\ast)\,dv_\ast
\end{align}
with
\begin{equation*}
|\Fk(v,v_\ast)|\leq C\{1+|v-v_\ast|^{-2}\}e^{-
\frac{1}{8}|v-v_\ast|^{2}-\frac{1}{8}\frac{\left||v|^{2}-|v_\ast|^{2}\right|^{2}}{|v-v_\ast|^{2}}}. 
\end{equation*}
Moreover, let
\begin{align}\label{kw-def}
\Fk_w(v,v_\ast)=w_{q}(v)\Fk(v,v_\ast)w_{-q}(v_\ast)
\end{align}
with  $q\geq0$,
then it also holds
\begin{equation*}
\int_{\R^3} \Fk_w(v,v_\ast)e^{\frac{\varepsilon|v-v_\ast|^2}{8}}dv_\ast\leq \frac{C}{1+|v|},
\end{equation*}
for any $\varepsilon\geq 0$ small enough.
\end{lemma}

For the velocity weighted derivative-in-$v_x$ estimates on the nonlinear operator $\Ga$,
we have the following lemma.

\begin{lemma}\label{Ga}
In the Maxwell molecular case, it holds that
\begin{align}\label{es1.Ga}
\|w_q\pa_{v_x}^m\Ga(f,g)\|_{L_v^2}\leq C\sum\limits_{m'\leq m}\|w_q\pa_{v_x}^{m'}f\|_{L_v^2}\|w_q\pa_{v_x}^{m-m'}g\|_{L_v^2},
\end{align}
and
\begin{align}\label{es11.Ga}
\|w_q\pa_{v_x}^m\Ga(f,g)\|_{L^\infty}\leq C\sum\limits_{m'\leq m}\|w_q\pa_{v_x}^{m'}f\|_{L^\infty}\|w_q\pa_{v_x}^{m-m'}g\|_{L^\infty},
\end{align}
for any integers $m\geq 0$ and $q\geq0$.
\end{lemma}

\begin{proof}
We prove \eqref{es11.Ga} only, since the proof for \eqref{es1.Ga} is similar and it follows from the proof of \cite[Lemma 2.3, pp.1111]{Guo-vpb}.
By definition \eqref{def.Ga},  we have
\begin{align}
\pa_{v_x}^m\Ga(f,g)=&\pa_{v_x}^m\int_{\R^3}\int_{\S^2}B_0\mu^{1/2}(v_\ast)f(v_\ast')g(v')\,d\om dv_\ast
-\pa_{v_x}^m\int_{\R^3}\int_{\S^2}B_0\mu^{1/2}(v_\ast)f(v_\ast)g(v)\,d\om dv_\ast\notag\\
=&\pa_{v_x}^m\int_{\R^3}\int_{\S^2}B_0\mu^{1/2}(v_\ast)f(v_\ast')g(v')\,d\om dv_\ast
-c_0\pa_{v_x}^m g(v)\int_{\R^3}\mu^{1/2}(v_\ast)f(v_\ast)\,dv_\ast,
\notag
\end{align}
where we have used $\int_{\S^2}B_0\,d\omega=c_0$ for a constant $c_0>0$.  Recalling \eqref{v.re}, by a change of variable $\tilde{u}=v_\ast-v$, we then have
\begin{align}
\pa_{v_x}^m&\int_{\R^3}\int_{\S^2}B_0\mu^{1/2}(v_\ast)f(v_\ast')g(v')\,d\om dv_\ast\notag\\
=&\pa_{v_x}^m\int_{\R^3}\int_{\S^2}B_0\mu^{1/2}(\tilde{u}+v)f(v+\tilde{u}_{\perp})g(v+\tilde{u}_\parallel)\,d\om d\tilde{u}\notag\\
\notag =&\sum\limits_{m_1+m_2\leq m}C_{m}^{m_1,m_2}\int_{\R^3}\int_{\S^2}B_0(\pa_{v_x}^{m-m_1-m_2}\mu^{1/2})(\tilde{u}+v)
(\pa_{v_x}^{m_1}f)(v+\tilde{u}_{\perp})(\pa_{v_x}^{m_2}g)(v+\tilde{u}_\parallel)\,d\om d\tilde{u},\notag
\end{align}
where
$
\tilde{u}_\parallel=(\tilde{u}\cdot \om)\om$ and $\tilde{u}_{\perp}=\tilde{u}-\tilde{u}_\parallel.
$
Then, by taking directly the $L^\infty$ norm,
\eqref{es11.Ga}  holds because
$$
\int_{\R^3}\int_{\S^2}B_0(\partial_{v_x}^m\mu^{1/2})(v_\ast)\,d\omega dv_\ast<\infty,
$$
for any integer $m\geq 0$.
This completes the proof of Lemma \ref{Ga}.
\end{proof}

The following lemma can be found in \cite[Lemmas 3.2 and 3.3, pp.638-639]{Guo-2006}, where the case for the hard sphere model  was considered.

\begin{lemma}\label{es.L}
In the Maxwell molecular case, there is a constant $\de_0>0$ such that
\begin{align}\label{bL}
\lag Lf,f\rag=\lag L\FP_1f,\FP_1f\rag\geq\de_0\|\FP_1f\|^2.
\end{align}
Moreover, for any integer $m>0$, there are constants $\de_1>0$ and $C>0$ such that
\begin{align}\label{wd-L}
\lag \pa_{v_x}^m L f,\pa_{v_x}^m  f\rag\geq\de_1\|\pa_{v_x}^m f\|^2-C\|f\|^2.
\end{align}
\end{lemma}

\begin{proof}
Since \eqref{bL} is quite elementary,
we only show \eqref{wd-L}. As in Lemma \ref{Ga}, the key point here is to show that the action of the derivatives $\pa_{v_x}^m$ on the nonlocal operator $L$ does not involve any other partial derivatives such as $\pa_{v_y}$ or $\pa_{v_z}$. By \eqref{def-L} and \eqref{bL}, we have
\begin{align}\label{dL-sp}
\lag \pa_{v_x}^m Lf,\pa_{v_x}^m f\rag
=&\lag  L \pa_{v_x}^mf,\pa_{v_x}^m  f\rag
+\sum\limits_{m_1<m}C_m^{m_1}\lag \pa_{v_x}^{m-m_1}L\pa_{v_x}^{m_1}f,\pa_{v_x}^m f\rag\notag\\
\geq&\de_0\|\FP_1[\pa_{v_x}^mf]\|^2-\sum\limits_{m_1<m}C_m^{m_1}|\lag \pa_{v_x}^{m-m_1}L\pa_{v_x}^{m_1}f,\pa_{v_x}^m f\rag|,
\end{align}
with
\begin{align}
&{\bf 1}_{ m_1<m}\pa_{v_x}^{m-m_1}L\pa_{v_x}^{m_1}f
\notag\\=&-\sum\limits_{m_1+m_2< m}C_{m}^{m_1,m_2}\int_{\R^3}\int_{\S^2}B_0(\pa_{v_x}^{m-m_1-m_2}\mu^{1/2})(\tilde{u}+v)
(\pa_{v_x}^{m_1}f)(v+\tilde{u}_{\perp})(\pa_{v_x}^{m_2}\mu^{1/2})(v+\tilde{u}_\parallel)\,d\om d\tilde{u}\notag\\
&+\pa_{v_x}^{m}\mu^{1/2}(v)\int_{\R^3}\int_{\S^2}B_0\mu^{1/2}(v_\ast)f(v_\ast)\,d\om dv_\ast
\notag\\&-\sum\limits_{m_1+m_2< m}C_{m}^{m_1,m_2}\int_{\R^3}\int_{\S^2}B_0(\pa_{v_x}^{m-m_1-m_2}\mu^{1/2})(\tilde{u}+v)
(\pa_{v_x}^{m_1}f)(v+\tilde{u}_\parallel)(\pa_{v_x}^{m_2}\mu^{1/2})(v+\tilde{u}_{\perp})\,d\om d\tilde{u},\notag
\end{align}
where we have used the change of variable $\tilde{u}=v_\ast-v$ again.
Consequently, as for \eqref{es1.Ga}, it follows
\begin{align}\label{dL-err}
\sum\limits_{m_1<m}C_m^{m_1}|\lag \pa_{v_x}^{m-m_1}L\pa_{v_x}^{m_1}f,\pa_{v_x}^m f\rag|
\leq& \eta\|\pa_{v_x}^m f\|^2+C_\eta\sum\limits_{m_1<m}\|\pa_ {v_x}^{m_1} f\|^2\notag\\
\leq& \eta\|\pa_{v_x}^m f\|^2+C_\eta\eta_1\|\pa_{v_x}^{m} f\|^2+C_{\eta,\eta_1}\|f\|^2,
\end{align}
for small enough constants $\eta>0$ and $\eta_1>0$, where Sobolev's interpolation inequality $\|\pa_{v_x}^{m_1} f\|^2\leq \eta_1\|\pa_{v_x}^{m} f\|^2+C_{\eta_1}\|f\|^2$ has been used.

One the other hand, it can be easily checked that
\begin{align}\label{P1-P0}
\|\FP_1[\pa_{v_x}^m f]\|\geq \|\pa_{v_x}^m f\|-\|\FP_0[\pa_{v_x}^m f]\|\geq \|\pa_{v_x}^m f\|
-C\|f\|.
\end{align}
Finally, plugging \eqref{dL-err} and \eqref{P1-P0} into \eqref{dL-sp} gives \eqref{wd-L}.
This completes the proof of Lemma \ref{es.L}.
\end{proof}

Next, the following lemma which was proved in \cite[Proposition 3.1, pp.13]{DL-2020} plays a significant role in obtaining the $L^\infty$ estimates of the first component in the Caflisch's decomposition of solutions.

\begin{lemma}\label{CK}
Let $\CK$ be given by \eqref{sp.cL}, then for any nonnegative integer $m\geq 0$, there is $C>0$ such that for any arbitrarily large $q>0$ we have
\begin{align}\label{CK1}
\sup_{|v|\geq M} w_{q}|\pa_{v_x}^m\CK f|\leq \frac{C}{q} \sum\limits_{0\leq m'\leq m}\|w_{q}\pa_{v_x}^{m'}f\|_{L^\infty},
\end{align}
for some $M=M(q)>0$.
In particular, one can choose $M=q^2$.
\end{lemma}
\begin{proof}
Since the general case
\begin{align}
\sup_{|v|\geq M} w_{q}|\pa_{v}^m\CK f|\leq \frac{C}{q} \sum\limits_{0\leq m'\leq m}\|w_{q}\pa_{v}^{m'}f\|_{L^\infty}\notag
\end{align}
was proved in \cite[Proposition 3.1, pp.13]{DL-2020}, as in Lemma \ref{es.L} we only point out that the derivative  $\pa_{v_x}^m$ acting on the nonlocal operator $\CK$ does not involve other  derivatives such  $\pa_{v_y}$ or $\pa_{v_z}$. Indeed, in view of \eqref{sp.cL}, similar to the proof of Lemma \ref{es.L}, we have
\begin{align}
\pa_{v_x}^m\CK f=&\sum\limits_{m_1\leq m}C_m^{m_1}\int_{\R^3}\int_{\S^2}B_0
(\pa_{v_x}^{m_1}f)(v+\tilde{u}_{\perp})(\pa_{v_x}^{m-m_1}\mu)(v+\tilde{u}_\parallel)\,d\om d\tilde{u}
-\pa^m_{v_x}\mu(v)\int_{\R^3}\int_{\S^2}B_0
f(v_\ast)\,d\om dv_\ast\notag\\
&+\sum\limits_{m_1\leq m}C_m^{m_1}\int_{\R^3}\int_{\S^2}B_0
(\pa_{v_x}^{m_1}f)(v+\tilde{u}_\parallel)(\pa_{v_x}^{m-m_1}\mu)(v+\tilde{u}_{\perp})\,d\om d\tilde{u}\notag
\\=&\sum\limits_{m_1\leq m}C_m^{m_1}\int_{\R^3}\int_{\S^2}B_0
(\pa_{v_x}^{m_1}f)(v_\ast')(\pa_{v_x}^{m-m_1}\mu)(v')\,d\om dv_\ast
-(\pa^m_{v_x}\mu)(v)\int_{\R^3}\int_{\S^2}B_0
f(v_\ast)\,d\om dv_\ast\notag\\
&+\sum\limits_{m_1\leq m}C_m^{m_1}\int_{\R^3}\int_{\S^2}B_0
(\pa_{v_x}^{m_1}f)(v')(\pa_{v_x}^{m-m_1}\mu)(v_\ast')\,d\om dv_\ast.\notag
\end{align}
Then, by  the similar calculation for estimating $\CI_1$ and $\CI_2$ in \cite[Proposition 3.1, pp.13]{DL-2020} yields
\eqref{CK1}.  This completes the proof of Lemma \ref{CK}.
\end{proof}

\section{A trace theorem}\label{sec3}

In this section, we present the following Ukai's trace theorem, see also Lemma 2.3 in \cite[Page 22 of 119]{EGKM-18} and Lemma 3.2 in \cite[Page 56 of 119]{EGKM-18}, respectively.

\begin{lemma}\label{ukai}
Let $\vps>0$ and $y\in[-h,h]$ with $0<h<+\infty$, and denote the near-grazing set of $\gamma_+$ or $\gamma_-$ as
\begin{equation*}
\gamma^{\varepsilon}_{\pm}\ \equiv \
\left\{(y,v)\in \gamma_{\pm}: |v_y|\leq \varepsilon \ \text{or} \ |v_y|\geq \frac{1}{\varepsilon },\ v=(v_x,v_y,v_z)\right\}.
\end{equation*}%
Then, there exists a constant $C_{\varepsilon,h}>0$ depending on $\vps$ and $h$ such that
\begin{equation}\label{utrace}
\begin{split}
|f\mathbf{1}_{\gamma_\pm\backslash\gamma^{\varepsilon }_\pm}|_{L^1}
&\leq C_{\varepsilon,h}\left\{\|f\|_{L^1} +\|\{v_y\pa _{y}-\al v_y\pa_{v_x}\}f\| _{L^1}\right\}.
\end{split}
\end{equation}
Moreover, it also holds
\begin{equation}\label{ttrace}
\int_{0}^{T}|f\mathbf{1}_{\gamma_+\backslash\gamma^{\varepsilon}_+}(t)|_{L^1}dt
\leq  C_{\varepsilon,h}\left\{\| f(0)\|_{L^1}+\int_{0}^{T}
\left[\|f(t)\|_{L^1}+\|\{\partial _{t}+v_y\pa_{y}-\al v_y\pa_{v_x}\}f(t)\|_{L^1}\right]dt\right\},
\end{equation}
for any $T\geq0.$
\end{lemma}

\begin{proof}
To prove \eqref{utrace}, we only consider the case that the boundary phase is outgoing, because the incoming case can be treated similarly. We introduce a parameter $t\in\R$ and treat $(y,v)$ as functions of $t$.
Define the characteristic line $[s,Y(s;t,y,v),V(s;t,y,v)]$ passing through $(y,v)=(t,y(t),v(t))$ such that
\begin{align}\label{chl}
\frac{dY}{ds}=v_y,\quad
\frac{dV_x}{ds}=-\al v_y.
\end{align}
Then it follows
\begin{align}
Y(s;t,y,v)=y-(t-s)v_y,\ \ V(s;t,y,v)=(v_x+\al(t-s)v_y,v_y,v_z),
\label{traj}
\end{align}
for $(y,v)\in\gamma_+\backslash\gamma^{\varepsilon}_+$. Along this trajectory, one has the identity
\begin{align}\label{f-ex-ch}
f(y,v)=f(Y(s;t,y,v),V(s;t,y,v))+\int_{s}^t\frac{d}{d\tau}f(Y(\tau;t,y,v),V(\tau;t,y,v))d\tau.
\end{align}
On the other hand, $(y,v)\in\gamma_+\backslash\gamma^{\varepsilon}_+$ also implies $h\vps\leq t_{\Fb}(y,v)\leq\frac{h}{\vps}$, {where $t_{\Fb}$ is given as \eqref{def-tb}}.
Therefore, by taking $s\in[t-t_{\Fb}(y,v),t]$, we get from \eqref{f-ex-ch} that
\begin{align}\label{f-ex-ch-p1}
\int_{\gamma_+\backslash\gamma^{\varepsilon}_+}|f(y,v)||v_y|dv\leq& C_{\vps,h}\int_{\gamma_+\backslash\gamma^{\varepsilon}_+}\int_{t-{\Fb}(y,v)}^t|f(Y(s;t,y,v),V(s;t,y,v))||v_y|dsdv\notag\\
&+C_{\vps,h}\int_{\gamma_+\backslash\gamma^{\varepsilon}_+}\int_{t-{\Fb}(y,v)}^t|\frac{d}{ds}f(Y(s;t,y,v),V(s;t,y,v))||v_y|ds dv\notag\\
=& C_{\vps,h}\int_{\gamma_+\backslash\gamma^{\varepsilon}_+}\int_{t-{\Fb}(y,v)}^t|f(Y(s;t,y,v),V(s;t,y,v))||v_y|dsdv\notag\\
&+C_{\vps,h}\int_{\gamma_+\backslash\gamma^{\varepsilon}_+}\int_{t-{\Fb}(y,v)}^t|[v_y\pa_Y-\al v_y\pa_{V_x}]f(Y(s;t,y,v),V(s;t,y,v))||v_y|ds dv.
\end{align}
Next, in light of the Jacobian
\begin{eqnarray}\label{def.matrM}
\begin{aligned} \frac{\pa(Y(s),V(s))}{\pa(s,v)}=\left|\begin{array} {cccc}
v_y \  \  \  & 0 \ \ \ & s \ \ \ &0\\[3mm]
-\al v_y \  \  \  & 1 \ \ \ & -\al s \ \ \ &0\\[3mm]
0 \  \  \  & 0 \ \ \ & 1 \ \ \ &0\\[3mm]
0 \  \  \  & 0 \ \ \ & 0 \ \ \ &1
\end{array} \right|=v_y,
\end{aligned}
\end{eqnarray}
and by a change of variable
$$[\tilde{y},u]=[Y(s;t,y,v),V(s;t,y,v)]=[y-(t-s)v_y,v_x+\al(t-s)v_y,v_y,v_z],$$
one gets
\begin{align}\label{bb-p1}
\int_{\gamma_+\backslash\gamma^{\varepsilon}_+}\int_{t-t_b(y,v)}^t|f(Y(s;t,y,v),V(s;t,y,v))||v_y|dsdv
\leq \int_{\R^3}\int_{-h}^h|f(\tilde{y},u)|d\tilde{y}du.
\end{align}
Similarly, by noticing that $\pa_{Y}f(Y,V)=\pa_{\tilde{y}}f(\tilde{y},u)$, $\pa_{V_x}f(Y,V)=\pa_{u_x}f(\tilde{y},u)$ and $v_y=u_y$, one has
\begin{align}\label{bb-p2}
\int_{\gamma_+\backslash\gamma^{\varepsilon}_+}&\int_{t-t_b(y,v)}^t|[v_y\pa_Y-\al v_y\pa_{V_x}]f(Y(s;t,y,v),V(s;t,y,v))||v_y|ds dv\notag\\
\leq& \int_{\R^3}\int_{-h}^h|[u_y\pa_{\tilde{y}}-\al u_y\pa_{u_x}]f(\tilde{y},u))|d\tilde{y}du.
\end{align}
Consequently, the desired estimate \eqref{utrace} in the case of outgoing boundary follows from \eqref{bb-p1}, \eqref{bb-p2} and \eqref{f-ex-ch-p1}.

We now turn to prove \eqref{ttrace}. For $f\in L^1([T_1,T]\times[-h,h]\times\R^3)$, we first show that
\begin{align}\label{axch}
\int_{T_1}^T&\int_{u\cdot n(y_f)>0}\int_{\max\{-t_{b}(y_f,u),T_1-\tilde{t}\}}^0
|f(\tilde{t}+s,Y(\tilde{t}+s;\tilde{t},y_f,u),V(\tilde{t}+s;\tilde{t},y_f,u))||u_y|dsdud\tilde{t}\notag\\
\leq&\int_{T_1}^T\int_{-h}^h\int_{\R^3}|f(t,y,v)|dydvdt,
\end{align}
where $y_f=\pm h$, $T\geq T_1\geq0$ and
\begin{align}
Y(\tilde{t}+s;\tilde{t},y_f,u)=y_f+su_y,\ \ V(\tilde{t}+s;\tilde{t},y_f,u)=(u_x-\al su_y,u_y,u_z),\notag
\end{align}
with
\begin{align}
Y(\tilde{t};\tilde{t},y_f,u)=y_f,\ \ V(\tilde{t};\tilde{t},y_f,u)=u=(u_x,u_y,u_z).\notag
\end{align}
Actually, given $(t,y,u)\in[T_1,T]\times[-h,h]\times\R^3$, let us define $y_f=y+t_{\Fb}(y,-u)u_y=\pm h$, and
denote
\begin{align}
y=Y(t;t-s,y_f,u)=y_f+su_y,\ \ v=V(t;t-s,y_f,u)=(u_x-\al su_y,u_y,u_z),\notag
\end{align}
for $u\cdot n(y_f)>0$. It is easy to see that $0\geq s\geq-t_{\Fb}(y_f,u),$ and it is natural to require that $t-s\leq T.$
By a change of variable $(y,v)\rightarrow(s,u)$ and using \eqref{def.matrM},
one has
\begin{align}
\int_{T_1}^T&\int_{u\cdot n(y_f)>0}\int_{\max\{-t_{b}(y_f,u),-(T-t)\}}^0
|f(t,Y(t;t-s,y_f,u),V(t;t-s,y_f,u))||u_y|dsdudt\notag\\
\leq&\int_{T_1}^T\int_{-h}^h\int_{\R^3}|f(t,y,v)|dydvdt.\label{chv-v}
\end{align}
On the other hand, if we denote $\tilde{t}=t-s$, then it follows
$s\geq T_1-\tilde{t}$ due to $t\geq T_1.$  In summary, one has
$$
s\geq\max\{-t_{\Fb}(y_f,u),T_1-\tilde{t}\}, \ T_1\leq\tilde{t}\leq T.
$$
Therefore, we have by  change of variable  $t\rightarrow\tilde{t}$ that
\begin{align}
\int_{T_1}^T&\int_{u\cdot n(y_f)>0}\int_{\max\{-t_{\Fb}(y_f,u),-(T-t)\}}^0
|f(t,Y(t;t-s,y_f,u),V(t;t-s,y_f,u))||u_y|dsdudt\notag\\
=&\int_{T_1}^T\int_{u\cdot n(y_f)>0}\int_{\max\{-t_{\Fb}(y_f,u),T_1-\tilde{t}\}}^0
|f(\tilde{t}+s,Y(\tilde{t}+s;\tilde{t},y_f,u),V(\tilde{t}+s;\tilde{t},y_f,u))|u_y|dsdud\tilde{t}.\label{cht-t}
\end{align}
Consequently,
\eqref{chv-v} and \eqref{cht-t} imply \eqref{axch}.
In addition, it follows that
\begin{align}\label{difff}
f(t,y_f,u)=&f(t+s,Y(t+s;t,y_f,u),V(t+s;t,y_f,u))\notag\\&+\int_{s}^0\frac{d}{d\tau}f(t+\tau,Y(t+\tau;t,y_f,u),V(t+\tau;t,y_f,u))d\tau\notag\\
=&f(t+s,Y(t+s;t,y_f,u),V(t+s;t,y_f,u))\notag\\&
+\int_{s}^0[\pa_t+u_y\pa_y-\al u_y\pa_{u_x}]f(t+\tau,Y(t+\tau;t,y_f,u),V(t+\tau;t,y_f,u))d\tau.
\end{align}
For any $(t,y_f,u)\in[\vps_1,T]\times\gamma_+\backslash\gamma^{\varepsilon}_+$ with $\vps_1>0$ to be determined later and for $0\geq s\geq\max\{-t_{\Fb}(y_f,u),\vps_1-t\}$,
we then get from \eqref{difff} and \eqref{axch} that
\begin{align}\label{bd.es1}
\min\{&h\vps,\vps_1\}\int_{\vps_1}^T\int_{u\cdot n(y_f)>0}|f(t,y_f,u)||u_y|dudt\notag\\
\leq&\int_{\vps_1}^T\int_{u\cdot n(y_f)>0}\int_{\max\{-t_{\Fb}(y_f,u),-t\}}^0|f(t+s,Y(t+s;t,y_f,u),V(t+s;t,y_f,u))||u_y|dtdsdu\notag\\&
+\int_{\vps_1}^T\int_{\max\{-t_{\Fb}(y_f,u),-t\}}^0\int_{u\cdot n(y_f)>0}
\int_{s}^0|[\pa_t+u_y\pa_y-\al u_y\pa_{u_x}]f(t+\tau,Y(t+\tau),V(t+\tau))||u_y|d\tau dudt\notag\\
\leq&\int_{0}^T\int_{u\cdot n(y_f)>0}\int_{\max\{-t_{\Fb}(y_f,u),-t\}}^0|f(t+s,Y(t+s;t,y_f,u),V(t+s;t,y_f,u))||u_y|dtdsdu\notag\\&
+\int_{0}^T\int_{\max\{-t_{\Fb}(y_f,u),-t\}}^0\int_{u\cdot n(y_f)>0}
\int_{s}^0|[\pa_t+u_y\pa_y-\al u_y\pa_{u_x}]f(t+\tau,Y(t+\tau),V(t+\tau))||u_y|d\tau dudt
\notag\\
\leq&\int_{0}^T\int_{-h}^h\int_{\R^3}|f(t,y,u))|dtdydu\notag\\&
+\int_{0}^T\int_{\max\{-t_{\Fb}(y_f,u),-t\}}^0\int_{u\cdot n(y_f)>0}
\int_{s}^0|[\pa_t+u_y\pa_y-\al u_y\pa_{u_x}]f(t+\tau,Y(t+\tau),V(t+\tau))||u_y|d\tau dudt,
\end{align}
where we have used the fact that  $h\vps\leq t_{\Fb}(y_f,u)\leq\frac{h}{\vps}$ due to $(y_f,u)\in\gamma_+\backslash\gamma^{\varepsilon}_+.$

Next, applying Fubini's Theorem and  using \eqref{axch} again, one also has
\begin{align}\label{bd.es2}
\int_{0}^T&\int_{u\cdot n(y_f)>0}\int_{\max\{-t_{\Fb}(y_f,u),-t\}}^0\int_{s}^0|[\pa_t+u_y\pa_y-\al u_y\pa_{u_x}]f(t+\tau,Y(t+\tau),V(t+\tau))||u_y|d\tau dudtds\notag\\
=& \int_{0}^Tdt\int_{u\cdot n(y_f)>0}\int^{\tau}_{\max\{-t_{\Fb}(y_f,u),-t\}}ds\int_{\max\{-t_{\Fb}(y_f,u),-t\}}^0 d\tau
|[\pa_t+u_y\pa_y-\al u_y\pa_{u_x}]f(t+\tau)||u_y|
\notag\\
\leq&\int_{0}^Tdt\int_{u\cdot n(y_f)>0}\int_{\max\{-t_{\Fb}(y_f,u),-t\}}^0 d\tau|\max\{-t_{\Fb}(y_f,u),-t\}|
|[\pa_t+u_y\pa_y-\al u_y\pa_{u_x}]f(t+\tau)||u_y|
\notag\\
\leq&\max\{h\vps,\vps_1\}\int_{0}^Tdt\int_{u\cdot n(y_f)>0}\int_{\max\{-t_{\Fb}(y_f,u),-t\}}^0 d\tau
|[\pa_t+u_y\pa_y-\al u_y\pa_{u_x}]f(t+\tau)||u_y|
\notag\\
\leq&\max\{h\vps,\vps_1\}\int_{0}^Tdt\int_{-h}^hdy\int_{\R^3}du
|[\pa_t+u_y\pa_y-\al u_y\pa_{u_x}]f(t,y,u)|.
\end{align}
Once \eqref{bd.es1} and \eqref{bd.es2} are obtained, it remains now to compute
\begin{align*}
\int^{\vps_1}_0\int_{u\cdot n(y_f)>0}|f(t,y_f,u)||u_y|dudt.
\end{align*}
In fact, if we choose $\vps_1$ to be small enough so that $\vps_1\leq h\vps$, at this stage, the backward trajectory hits the initial plane first. Therefore, for $(t,y_f,u)\in[0,\vps_1]\times\gamma_+\backslash\gamma^{\varepsilon}_+$, by directly using \eqref{def.matrM} and applying \eqref{axch} once again, it follows
\begin{align}
\int_0^{\vps_1}&\int_{u\cdot n(y_f)>0}|f(t,y_f,u)||u_y|dudt\notag\\
\leq&\int_0^{\vps_1}\int_{u\cdot n(y_f)>0}|f(0,Y(0;t,y_f,u),V(0;t,y_f,u))||u_y|dudt\notag\\
&+\int_0^{\vps_1}\int_{u\cdot n(y_f)>0}\int_{-t}^0|[\pa_t+u_y\pa_y-\al u_y\pa_{u_x}]|f(t+\tau)||u_y|d\tau dudt
\notag\\
\leq&C\int_{-h}^h\int_{\R^3}|f(0,y,u)|dydu+C\int_0^{\vps_1}\int_{-h}^h\int_{\R^3}|[\pa_t+u_y\pa_y-\al u_y\pa_{u_x}]|f(t)|dydudt.\notag
\end{align}
The proof of Lemma \ref{ukai} is then completed.
\end{proof}

\section{Steady problem: the first order correction}\label{sp.sec}

In this and the next sections, we are going to show Theorem \ref{st.sol} for the existence of solutions to the steady problem \eqref{Fst}. Recall \eqref{Fst.ex} and \eqref{Fst.ex.mc}. For the purpose, we will first study in this section the first order correction term $G_1$ determined by the boundary value problem \eqref{G1} and \eqref{G1bd}. Notice that \eqref{G1odd} and \eqref{G1drbc} are satisfied. Existence of the remainder  $G_R$ for the boundary value problem \eqref{Gr} and \eqref{Grbd} will be considered in the next section. Indeed, we have the following proposition.

\begin{proposition}\label{G1.lem}
The boundary value problem \eqref{G1} and \eqref{G1bd} admits a unique solution $G_1=G_1(y,v)$ satisfying
\begin{align}\label{G1-pp}
G_1(-v_x)=-G_1(v_x),\ \ \int_{-1}^1\int_{\R^3}G_1(y,v)dvdy=0,
\end{align}
and
\begin{align}\label{G1-bdd}
\|w_{q}\pa_{v_x}^mG_1\|_{L^\infty}\leq \tilde{C}_1,
\end{align}
for any integers $m\geq 0$ and $q\geq 0$, where $\tilde{C}_1>0$ is a constant depending only on $m$ and $q$.
\end{proposition}

To prove this proposition, let $0<\eps<1$ and $0\leq \si\leq1$, then we consider the following general approximation equations
\begin{align}\label{G1-it}
\eps G_1+v_y\pa_yG_1+\nu_0G_1=\si KG_1+\RF,
\end{align}
and
\begin{align}\label{G1bd-it}
G_1(\pm1,v)|_{v_y\lessgtr0}=0,
\end{align}
where the source term $\RF=\RF(y,v)$ is given and satisfies $\RF(-v_x)=-\RF(v_x)$. Recall that $\nu_0$ and $K$ are defined by \eqref{sp.L}. The above boundary value problem can be formally reduced to
\begin{align}\notag
v_y\pa_yG_1+LG_1=\RF,
\end{align}
and
\begin{align}\notag
G_1(\pm1,v)|_{v_y\lessgtr0}=0,
\end{align}
as $\si\rightarrow1^-$ and $\eps\rightarrow0^+$.  To prove this rigorously,
we deduce the following {\it a priori} estimate.

\begin{lemma}[{\it a priori} estimate]\label{lem.stap}
The solution to the boundary value problem \eqref{G1-it} and \eqref{G1bd-it} satisfies the following uniform estimate with respect to both $\si$ and $\eps$:
\begin{align}\label{ubdG1}
\sum\limits_{0\leq m\leq N_0}\|w_q\pa_{v_x}^mG_1\|_{L^\infty}\leq\RC_0\sum\limits_{0\leq m\leq N_0}\|w_q\pa_{v_x}^m\RF\|_{L^\infty},
\end{align}
where $N_0$ is an arbitrary non-negative integer and the constant $\RC_0>0$ is independent of  $\eps$ and $\si$.
\end{lemma}

\begin{proof}
The proof of \eqref{ubdG1} is divided into two steps.

\medskip
\noindent\underline{{\it $L^\infty$ estimates.}}
Let $\mathfrak{G}_m=w_q\pa_{v_x}^mG_1$ for $m\geq0$ and $q\geq 0$, then $\RG_m$ satisfies
\begin{align}\label{G1m-it}
\eps \RG_m&+v_y\pa_y\RG_m+\nu_0\RG_m\notag\\=&\si w_qK\pa^m_{v_x}G_1
+\si{\bf 1}_{m>0}\sum\limits_{m'<m}C_m^{m'} w_q(\pa_{v_x}^{m-m'}K)(\pa^{m'}_{v_x}G_1)-w_q\pa^m_{v_x}\RF,
\end{align}
and
\begin{align}\label{G1mbd-it}
\RG_m(\pm1,v)|_{v_y\lessgtr0}=0.
\end{align}
We write the solution of \eqref{G1m-it} and \eqref{G1mbd-it} in the following mild form
\begin{align}\label{G1-np}
\RG_m(y,v)=&\si\int_{-1}^ye^{-\frac{\nu_0+\eps}{v_y}(y-y')} \frac{w_q}{v_y}K(w_{-q}\RG_m)(y')dy'
\notag\\&
+\si{\bf 1}_{m>0}\sum\limits_{m'<m}C_m^{m'}\int_{-1}^ye^{-\frac{\nu_0+\eps}{v_y}(y-y')}\frac{w_q}{v_y}(\pa_{v_x}^{m-m'}K)(\pa^{m'}_{v_x}G_1)(y')dy'
\notag\\&-\int_{-1}^ye^{-\frac{\nu_0+\eps}{v_y}(y-y')}\frac{w_q}{v_y}\pa^m_{v_x}\RF~dy':=\sum\limits_{i=1}^3\RI_i,\ \ \textrm{for}\ v_y>0,
\end{align}
and
\begin{align}
\RG_m(y,v)=&-\si\int_{y}^1e^{-\frac{\nu_0+\eps}{v_y}(y-y')} \frac{w_q}{v_y}K(w_{-q}\RG_m)(y')dy'
\notag\\&
-\si{\bf 1}_{m>0}\sum\limits_{m'<m}C_m^{m'}\int_{y}^1e^{-\frac{\nu_0+\eps}{v_y}(y-y')}\frac{w_q}{v_y}(\pa_{v_x}^{m-m'}K)(\pa^{m'}_{v_x}G_1)(y')dy'
\notag\\&+\int_{y}^1e^{-\frac{\nu_0+\eps}{v_y}(y-y')}\frac{w_q}{v_y}\pa^m_{v_x}\RF~dy':=\sum\limits_{i=4}^6\RI_i,\ \ \textrm{for}\ v_y<0.\notag
\end{align}
We next compute $\RI_i$ $(1\leq i\leq6)$ term by term.
Since
\begin{align}\label{bfr}
\left\{\begin{array}{rll}
&{\bf 1}_{v_y>0}\dis{\int_{-1}^y}e^{-\frac{\nu_0+\eps}{v_y}(y-y')}v_y^{-1}dy'
\leq \frac{1}{\nu_0+\eps}(1-e^{-\frac{2(\nu_0+\eps)}{|v_y|}})< \frac{1}{\nu_0+\eps},\\[2mm]
&{\bf 1}_{v_y<0}\left|\dis{\int_{y}^1}e^{-\frac{\nu_0+\eps}{v_y}(y-y')}v_y^{-1}dy'\right|< \frac{1}{\nu_0+\eps},
\end{array}\right.
\end{align}
we see that
\begin{align*}
|\RI_3|, |\RI_6|\leq C\|w_q\pa^m_{v_x}\RF\|_{L^\infty}.
\end{align*}
In view of definition \eqref{sp.L} and Lemma \ref{Ga}, it follows
\begin{align*}
|\RI_2|, |\RI_5|\leq C{\bf 1}_{m>0}\sum\limits_{m'<m}\|w_q(\pa_{v_x}^{m-m'}K)(\pa^{m'}_{v_x}G_1)\|_{L^\infty}
\leq C{\bf 1}_{m>0}\sum\limits_{m'<m}\|w_q\pa^{m'}_{v_x}G_1\|_{L^\infty}.
\end{align*}
Consequently, we have
\begin{align}\label{RG-st}
|\RG_m(y,v)|\leq& {\bf 1}_{v_y>0}\si\int_{-1}^ye^{-\frac{\nu_0+\eps}{v_y}(y-y')}v_y^{-1} \int_{\R^3}{\bf k}_w(v,v')|\RG_m(v',y')|dv'dy'\notag\\
&+{\bf 1}_{v_y<0}\si\int_{y}^1e^{-\frac{\nu_0+\eps}{v_y}(y-y')}|v_y|^{-1} \int_{\R^3}{\bf k}_w(v,v')|\RG_m(v',y')|dv'dy'\notag\\
&+C{\bf 1}_{m>0}\sum\limits_{m'<m}\|w_q\pa^{m'}_{v_x}G_1\|_{L^\infty}+C\|w_q\pa^m_{v_x}\RF\|_{L^\infty},
\end{align}
where ${\bf k}_w$ is given in Lemma \ref{Kop}.
Then we iterate \eqref{RG-st} once more to obtain
\begin{align}\label{RG-sec}
|\RG_m(y,v)|\leq&\sum\limits_{i=1}^6\RI_{1,i},
\end{align}
with
\begin{align*}
\RI_{1,1}=& {\bf 1}_{v_y>0}\si^2\int_{-1}^ye^{-\frac{\nu_0+\eps}{v_y}(y-y')}v_y^{-1}
\int_{\R^3}{\bf k}_w(v,v'){\bf 1}_{v'_y>0}\int_{-1}^{y'}e^{-\frac{\nu_0+\eps}{v'_y}(y'-y'')}{v'_y}^{-1}\notag\\
&\qquad\times\int_{\R^3}{\bf k}_w(v',v'')|\RG_m(v'',y'')|dv''dy'' dv'dy',
\end{align*}
\begin{align*}
\RI_{1,2}=&{\bf 1}_{v_y>0}\si^2\int_{-1}^ye^{-\frac{\nu_0+\eps}{v_y}(y-y')}v_y^{-1}
\int_{\R^3}{\bf k}_w(v,v'){\bf 1}_{v'_y<0}\int_{y'}^{1}e^{-\frac{\nu_0+\eps}{v'_y}(y'-y'')}|v'_y|^{-1}\notag\\
&\qquad\times\int_{\R^3}{\bf k}_w(v',v'')|\RG_m(v'',y'')|dv''dy'' dv'dy',
\end{align*}
\begin{align*}
\RI_{1,3}=&{\bf 1}_{v_y<0}\si^2\int_{y}^1e^{-\frac{\nu_0+\eps}{v_y}(y-y')}|v_y|^{-1}
\int_{\R^3}{\bf k}_w(v,v'){\bf 1}_{v'_y>0}\int_{-1}^{y'}e^{-\frac{\nu_0+\eps}{v'_y}(y'-y'')}{v'_y}^{-1}\notag\\
&\qquad\times\int_{\R^3}{\bf k}_w(v',v'')|\RG_m(v'',y'')|dv''dy'' dv'dy',
\end{align*}
\begin{align*}
\RI_{1,4}=&{\bf 1}_{v_y<0}\si^2\int_{y}^1e^{-\frac{\nu_0+\eps}{v_y}(y-y')}|v_y|^{-1}
\int_{\R^3}{\bf k}_w(v,v'){\bf 1}_{v'_y<0}\int_{y'}^{1}e^{-\frac{\nu_0+\eps}{v'_y}(y'-y'')}|v'_y|^{-1}\notag\\
&\qquad\times\int_{\R^3}{\bf k}_w(v',v'')|\RG_m(v'',y'')|dv''dy'' dv'dy',
\end{align*}
\begin{align*}
\RI_{1,5}=&{\bf 1}_{v_y>0}\si\int_{-1}^ye^{-\frac{\nu_0+\eps}{v_y}(y-y')}v_y^{-1} \int_{\R^3}{\bf k}_w(v,v')
\left(C{\bf 1}_{m>0}\sum\limits_{m'<m}\|w_q\pa^{m'}_{v_x}G_1\|_{L^\infty}+C\|w_q\pa^m_{v_x}\RF\|_{L^\infty}\right)dv'dy',
\end{align*}
\begin{align*}
\RI_{1,6}=&{\bf 1}_{v_y<0}\si\int_{y}^1e^{-\frac{\nu_0+\eps}{v_y}(y-y')}|v_y|^{-1} \int_{\R^3}{\bf k}_w(v,v')
\left(C{\bf 1}_{m>0}\sum\limits_{m'<m}\|w_q\pa^{m'}_{v_x}G_1\|_{L^\infty}+C\|w_q\pa^m_{v_x}\RF\|_{L^\infty}\right)dv'dy'.
\end{align*}
By using \eqref{bfr} and Lemma \ref{Kop}, we see that the last two terms can be bounded as
$$
|\RI_{1,5}|,\ |\RI_{1,6}|\leq C{\bf 1}_{m>0}\sum\limits_{m'<m}\|w_q\pa^{m'}_{v_x}G_1\|_{L^\infty}+C\|w_q\pa^m_{v_x}\RF\|_{L^\infty}.
$$
For the other four  terms, we only compute $\RI_{1,2}$  because the other three terms can be treated similarly.
The estimates are divided into three cases. First of all, we take $M>0$ large enough.

\noindent\underline{{\it Case 1. $|v|> M$.}} In this case, Lemma \ref{Kop} and \eqref{bfr} directly give
\begin{align*}
\RI_{1,2}\leq\frac{C}{1+M}\|\RG_m\|_{L^\infty}.
\end{align*}

\noindent\underline{{\it Case 2. $|v|\leq M$ and $|v'|> 2M$, or $|v'|\leq 2M$ and $|v''|> 3M$.}} In this case, we  have either $|v-v'|> M$ or $|v'-v''|> M$ so that  one of the following two estimates holds correspondingly
\begin{equation*}
\begin{split}
{\bf k}_{w}(v,v')
\leq Ce^{-\frac{\vps M^2}{16}}{\bf k}_{w}(v,v')e^{\frac{\vps |v-v'|^2}{16}},\
{\bf k}_{w}(v',v'')
\leq Ce^{-\frac{\vps M^2}{16}}{\bf k}_{w}(v',v'')e^{\frac{\vps |v'-v''|^2}{16}}.
\end{split}
\end{equation*}
This together with Lemma
\ref{Kop} and \eqref{bfr} gives
\begin{align*}
\RI_{1,2}\leq Ce^{-\frac{\vps M^2}{16}}\|\RG_m\|_{L^\infty}.
\end{align*}

\noindent\underline{{\it Case 3. $|v|\leq M$, $|v'|\leq 2M$ and $|v''|\leq 3M$.}} In this situation, we make use of the boundedness of the operator $K$ on the complement of
a singular set.
For any large $N>0$,
we choose a number $M(N)$ to define
\begin{equation}\label{kw-M}
{\bf k}_{w,M}(v,v')\equiv \mathbf{1}
_{|v-v^{\prime }|\geq \frac{1}{M},|v^{\prime}|\leq 2M}{\bf k}_{w}(v,v'), {\bf k}_{w,M}(v',v'')\equiv \mathbf{1}
_{|v'-v''|\geq \frac{1}{M},|v''|\leq 3M}{\bf k}_{w}(v',v''),
\end{equation}%
such that
$$\sup_{v}\int_{\R^{3}}|{\bf k}_{w,M}(v,v^{\prime})-{\bf k}_{w}(v ,v^{\prime})|dv^{\prime}\leq\frac{1}{N},$$
and
$$\sup_{v'}\int_{\R^{3}}|{\bf k}_{w,M}(v',v'')-{\bf k}_{w}(v',v'')|dv''\leq\frac{1}{N}.$$
Moreover, note that
${\bf k}_{w,M}(v,v'),\ {\bf k}_{w,M}(v',v'')\leq C_M.$
We further rewrite
\begin{align}
{\bf k}_{w}(v ,v^{\prime}){\bf k}_{w}(v',v'')=&[{\bf k}_{w}(v ,v^{\prime})-{\bf k}_{w,M}(v,v')]{\bf k}_{w}(v',v'')
\notag\\&+{\bf k}_{w,M}(v,v')[{\bf k}_{w}(v',v'')-{\bf k}_{w,M}(v',v'')]+{\bf k}_{w,M}(v',v''){\bf k}_{w,M}(v,v').\notag
\end{align}
The first two difference terms lead to the small contribution of $\RI_{1,2}$ as
$$
\frac{C}{N}\|\RG_m\|_{L^\infty}.
$$
For the last  term, we use the following decomposition
\begin{align*}
{\bf 1}_{v_y>0}&\si^2\int_{-1}^ye^{-\frac{\nu_0+\eps}{v_y}(y-y')}v_y^{-1}
\int_{|v'|\leq 2M,\ |v''|\leq3M}{\bf k}_{w,M}(v,v'){\bf k}_{w,M}(v',v'')\notag\\
&\times{\bf 1}_{v'_y<0}\left[\int_{y'+\eta_0}^{1}+\int_{y'}^{y'+\eta_0}\right]
e^{-\frac{\nu_0+\eps}{v'_y}(y'-y'')}|v'_y|^{-1}
 |\RG_m(v'',y'')|dv''dy'' dv'dy':=\RI_{1,2}^I+\RI_{1,2}^{II},
\end{align*}
where $\eta_0>0$ is suitably small.
For $\RI_{1,2}^I$, since $y''-y'\geq\eta_0$, it follows that
\begin{align*}
{\bf 1}_{v'_y<0}e^{-\frac{\nu_0+\eps}{v'_y}(y'-y'')}|v'_y|^{-1}\leq \frac{C}{\eta_0},
\end{align*}
which together with Lemma \ref{Kop} as well as \eqref{bfr} implies
\begin{align}
\RI_{1,2}^I\leq\frac{C_M}{\eta_0}\left\{\int_{|v''|\leq 3M}\int_{-1}^1|\pa_{v_x}^mG_1(v'',y'')|^2dv''dy''\right\}^{\frac{1}{2}}.\notag
\end{align}
As to $\RI_{1,2}^{II}$, since $y''-y'\leq\eta_0$, we have that for $\beta\in(0,1)$,
\begin{align}
\int_{|v'|\leq2M}&{\bf 1}_{v'_y<0}\int_{y'}^{y'+\eta_0}
e^{-\frac{\nu_0+\eps}{v'_y}(y'-y'')}|v'_y|^{-1}dy''dv'\notag\\=&\int_{|v'|\leq2M}{\bf 1}_{v'_y<0}\int_{y'}^{y'+\eta_0}
e^{-\frac{\nu_0+\eps}{v'_y}(y'-y'')}\left|\frac{y'-y''}{v'_y}\right|^{\bet}|y'-y''|^{-\bet}|v'_y|^{-1+\beta}dy''dv'
\notag\\
\leq & C\int_{|v'|\leq2M}|v'_y|^{-1+\beta}dv'\int_{y'}^{y'+\eta_0}|y'-y''|^{-\bet}dy''\leq C_M\eta_0^{1-\beta},\label{k-es1}
\end{align}
where we have used the fact that
$$
e^{-\frac{\nu_0+\eps}{|v'_y|}|y'-y''|}\left|\frac{y'-y''}{v'_y}\right|^{\bet}<+\infty.
$$
Plugging \eqref{k-es1} into $\RI_{1,2}^{II}$, we get
\begin{align}
\RI_{1,2}^{II}\leq C_M\eta_0^{1-\beta}\|\RG_m\|_{L^\infty}.\notag
\end{align}
As a consequence, one has
\begin{align}
\RI_{1,2}
\leq& \left\{\frac{C}{N}+C_M\eta_0^{1-\beta}+Ce^{-\frac{\vps M^2}{16}}\right\}\|\RG_m\|_{L^\infty}+C_M\|\pa_{v_x}^mG_1\|.\notag
\end{align}
Substituting the above estimates into \eqref{RG-sec}, we conclude
\begin{align}\label{G1m-sum1}
\|\RG_m\|_{L^\infty}\leq C{\bf 1}_{m>0}\sum\limits_{m'<m}\|w_q\pa^{m'}_{v_x}G_1\|_{L^\infty}+C\|\pa_{v_x}^mG_1\|+C\|w_q\pa^m_{v_x}\RF\|_{L^\infty}.
\end{align}
A linear combination of \eqref{G1m-sum1} from $m=0$ to $m=N_0$ gives the following {\it a priori} estimate
\begin{align}\label{G1m-sum2}
\sum\limits_{0\leq m\leq N_0}\|\RG_m\|_{L^\infty}\leq C\sum\limits_{0\leq m\leq N_0}\|\pa_{v_x}^mG_1\|+C\sum\limits_{0\leq m\leq N_0}\|w_q\pa^m_{v_x}\RF\|_{L^\infty},
\end{align}
where $C>0$ depends on $N_0$ and $q$. This concludes the $L^\infty$ estimate.

\medskip
\noindent\underline{{\it $L^2$ estimates.}} To close the $L^\infty$ estimate \eqref{G1m-sum2}, we need to derive the $L^2$ estimate for $G_1$. For this, we first consider the zero-th order $L^2$ estimate $\|G_1\|$. Notice that $G_1=\FP_0G_1+\FP_1G_1$ and
$\FP_0G_1=[a_1+\Fb_1\cdot v+c_1(|v|^2-3)]\sqrt{\mu}$ with $\Fb_1=[b_{1,1},b_{1,2},b_{1,3}]$.
Moreover, it holds
\begin{align*}
a_1=\lag G_1,\sqrt{\mu}\rag,\ \Fb_1=\lag G_1,v\sqrt{\mu}\rag, \ c_1=\frac{1}{6}\lag G_1,|v|^2\sqrt{\mu}\rag.
\end{align*}
On the other hand, from \eqref{G1-np} with $m=0$, it holds $G_1(y,-v_x,v_y,v_z)=-G_1(y,v_x,v_y,v_z)$, namely $G_1$ is odd in $v_x$. This implies
\begin{align}\label{mG10}
a_1=b_{1,2}=b_{1,3}=c_1=0.
\end{align}
To obtain the $L^2$ estimate of the macroscopic component, it remains now to deduce the $L^2$ estimate of $b_{1,1}.$
Actually, one can show that
\begin{align}\label{mG1}
\|b_{1,1}\|^2\leq C\|\FP_1G_1\|^2+C\int_{v_y\gtrless0}|v_y|G^2_1(\pm1)dv+C\|w_q\RF\|_{L^\infty}^2,
\end{align}
where $C>0$ is a constant independent of $\eps$ and $\si$. For this, we
define
\begin{align*}
\Psi=\Psi_{b_{1,1}}=v_yv_{x}\frac{d}{dy}\phi_{b_{1,1}}(y)\sqrt{\mu},
\end{align*}
where
\begin{align*}
-\phi''_{b_{1,1}}=b_{1,1},\ \phi_{b_{1,1}}(\pm1)=0.
\end{align*}
For the above boundary value problem on $b_{1,1}$, it follows
\begin{align}\label{ep-G1b}
\|\phi_{b_{1,1}}\|_{H^2}\leq C\|b_{1,1}\|, \ |\phi'_{b_{1,1}}(\pm1)|\leq C\|b_{1,1}\|.
\end{align}
Taking the inner product of \eqref{G1-it} and $\Psi_{b_{1,1}}$
over $(-1,1)\times\R^3$, we get
\begin{align}\label{G1-tt}
\eps(G_1,\Psi_{b_{1,1}})&-(v_yG_1,\pa_y\Psi_{b_{1,1}})+\lag v_yG_1(1),\Psi_{b_{1,1}}(1)\rag-\lag v_yG_1(-1),\Psi_{b_{1,1}}(-1)\rag
\notag\\& +(1-\si)\nu_0(G_1,\Psi_{b_{1,1}})+\si (LG_1,\Psi_{b_{1,1}})
=(\RF,\Psi_{b_{1,1}}).
\end{align}
We now compute the terms in \eqref{G1-tt} one by one.
By Cauchy-Schwarz's inequality and  \eqref{ep-G1b}, one has
\begin{align}
[\eps+(1-\si)\nu_0]|(G_1,\Psi_{b_{1,1}})|\leq&[\eps+(1-\si)\nu_0]|(\FP_0G_1,\Psi_{b_{1,1}})|+[\eps+(1-\si)\nu_0]|(\FP_1G_1,\Psi_{b_{1,1}})|
\notag\\ \leq& \eta[\eps+(1-\si)\nu_0]\|b_{1,1}\|^2+C_\eta[\eps+(1-\si)\nu_0]\|\FP_1G_1\|^2,\notag
\end{align}
\begin{align}
-(v_yG_1,\pa_y\Psi_{b_{1,i}})=&-(v_y\FP_0G_1,\pa_y\Psi_{b_{1,i}})-(v_y\FP_1G_1,\pa_y\Psi_{b_{1,i}})\notag\\
\geq&\|b_{1,1}\|^2-\eta\|b_{1,1}\|^2-C_\eta\|\FP_1G_1\|^2,\notag
\end{align}
$$
|(\RF,\Psi_{b_{1,1}})|\leq \eta\|b_{1,1}\|^2+C_\eta\|w_q\RF\|_{L^\infty}^2.
$$
And by Lemma \ref{Ga}, it follows
\begin{align*}
\si|(LG_1,\Psi_{b_{1,1}})|
\leq& \eta\|b_{1,1}\|^2+C_\eta\|\FP_1G_1\|^2.
\end{align*}
For the boundary term, one has from \eqref{G1bd-it} and \eqref{ep-G1b} that
\begin{align}
\lag v_yG_1(1),\Psi_{b_{1,1}}(1)\rag&-\lag v_yG_1(-1),\Psi_{b_{1,1}}(-1)\rag\notag
\\=&\int_{v_y>0}v_yG_1(1)\Psi_{b_{1,1}}(1)dv-\int_{v_y<0}v_yG_1(-1)\Psi_{b_{1,1}}(-1)dv\notag\\
\leq& \eta\|b_{1,1}\|^2+C_\eta\int_{v_y\gtrless0}|v_y|G^2_1(\pm1)dv.\notag
\end{align}
Combining the above estimates for the terms in \eqref{G1-tt}, we have \eqref{mG1}.

We now deduce the $L^2$ estimate on the microscopic component $\FP_1G_1.$  Direct energy estimate
 on \eqref{G1-it} gives
\begin{equation}\label{mi-G1}
[\eps+(1-\si)\nu_0]\|G_1\|^2+\de_0\si\|\FP_1G_1\|^2+\frac{1}{2}\int_{v_y\gtrless0}|v_y|G^2_1(\pm1)dv\leq \eta\|G_1\|^2+C_\eta\|w_q\RF\|_{L^\infty}^2.
\end{equation}
Thus, \eqref{mG1} and \eqref{mi-G1} as well as \eqref{mG10} yield
\begin{align}\label{G1-bl2}
\|G_1\|^2+\int_{v_y\gtrless0}|v_y|G^2_1(\pm1)dv\leq C\|w_q\RF\|_{L^\infty}^2.
\end{align}
Furthermore, for the higher order $L^2$ estimates on $G_1$, we have from $(\pa_{v_x}^mG_1,\pa_{v_x}^m\eqref{G1-it})$ with $m\geq1$ that
\begin{align}\label{mih-G1}
[\eps+(1-\si)\nu_0]&\|\pa_{v_x}^mG_1\|^2+\de_0\si\|\pa_{v_x}^mG_1\|^2\notag\\&+\frac{1}{2}\int_{v_y\gtrless0}|v_y|\pa_{v_x}^mG^2_1(\pm1)dv\leq C\|G_1\|^2+C\|w_q\pa_{v_x}^m\RF\|_{L^\infty}^2,
\end{align}
where Lemma \ref{es.L} has been used for $\si(\pa_{v_x}^m LG_1,\pa_{v_x}^mG_1).$

Finally, the {\it a priori} estimate \eqref{ubdG1} follows from \eqref{G1m-sum2}, \eqref{G1-bl2} and \eqref{mih-G1}. This completes the proof of Lemma \ref{lem.stap}.
\end{proof}

With the {\it a priori} estimate \eqref{ubdG1}, we now prove the following existence result for general linear equations \eqref{G1-it} and \eqref{G1bd-it}.
Before doing this, we first define the following function space
$$
\RX_{N_0}=\{g=g(y,v)|\sum\limits_{0\leq m\leq N_0}\|w_q\pa^m_{v_x}g\|_{L^\infty}<+\infty,\ g(-v_x)=-g(v_x)\},
$$
associated with the norm
$$
\|g\|_{\RX_{N_0}}=\sum\limits_{0\leq m\leq N_0}\|w_q\pa^m_{v_x}g\|_{L^\infty}.
$$
And for  convenience, we also define a linear operator $\RL_\si$ by
$$
\RL_\si g=[\eps +v_y\pa_y+\nu_0-\si K]g.
$$

\begin{lemma}\label{ex-G1-rf}
Assume $\RF=\RF(y,v)$ satisfies
\begin{align}\label{RF-as}
\RF(-v_x)=-\RF(v_x),\
\sum\limits_{0\leq m\leq N_0}\|w_q\pa^m_{v_x}\RF\|_{L^\infty}<+\infty,
\end{align}
then there exists a unique solution $G_1=G_1(y,v)$
to \eqref{G1-it} and \eqref{G1bd-it} with $\si=1$ satisfying
\begin{align}\notag
G_1(-v_x)=-G_1(v_x),
\end{align}
and
\begin{align}\label{G1-rf}
\sum\limits_{0\leq m\leq N_0}\|w_{q}\pa_{v_x}^{m}G_1\|_{L^\infty}
\leq
\RC_0\sum\limits_{0\leq m\leq N_0}\|w_q\pa^m_{v_x}\RF\|_{L^\infty},
\end{align}
where  $\RC_0>0$ is a constant depending only on $N_0$ and $q$.
\end{lemma}

\begin{proof}
The proof is based on a bootstrap argument in the following three steps.

\medskip
\noindent\underline{{\it Step 1. Existence for $\si=0$.}} If $\si=0$, then \eqref{G1-it} and \eqref{G1bd-it} are reduced to
\begin{align}\notag
\eps G_1+v_y\pa_yG_1+\nu_0G_1=\RF,
\end{align}
and
\begin{align}\notag
G_1(\pm1,v)|_{v_y\lessgtr0}=0,
\end{align}
which has a unique explicit solution
\begin{equation}\label{def.G1noK}
G_1(y,v)={\bf 1}_{v_y>0}\int_{-1}^ye^{-\frac{(\nu_0+\eps)(y-y')}{v_y}}v_y^{-1}\RF(y',v)dy'+{\bf 1}_{v_y<0}\int_{y}^1e^{-\frac{(\nu_0+\eps)(y-y')}{v_y}}v_y^{-1}\RF(y',v)dy'.
\end{equation}
Moreover, one sees that $G_1(-v_x)=-G_1(v_x)$ according to \eqref{RF-as},
and a direct calculation implies
\begin{align}\notag
\sum\limits_{0\leq m\leq N_0}\|w_{q}\pa_{v_x}^{m}G_1\|_{L^\infty}
\leq
\RC_0\sum\limits_{0\leq m\leq N_0}\|w_q\pa^m_{v_x}\RF\|_{L^\infty}.
\end{align}
\noindent\underline{{\it Step 2. Existence for any $\si\in[0,\si_\ast]$ with some $\si_\ast>0$.}}
Suppose $\si\in(0,1]$, we consider a more general equation
\begin{align}\label{G1-it-si1}
\eps G_1+v_y\pa_yG_1+\nu_0G_1=\si KG_1+\RF,
\end{align}
with
\begin{align}\label{G1bd-it-si1}
G_1(\pm1,v)|_{v_y\lessgtr0}=0.
\end{align}
To solve this boundary value problem, we design the following approximation equations
\begin{align}\notag
\eps G^{n+1}_1+v_y\pa_yG^{n+1}_1+\nu_0G^{n+1}_1=\si KG^n_1+\RF,
\end{align}
with
\begin{align}\notag
G^{n+1}_1(\pm1,v)|_{v_y\lessgtr0}=0,
\end{align}
starting from $G_1^0=0.$ Once $G_1^n$ is given, $G_1^{n+1}$ is well defined by step 1  and satisfies the estimate
\begin{align}\label{G1-s1-es}
\sum\limits_{0\leq m\leq N_1}\|w_{q}\pa_{v_x}^{m}G^{n+1}_1\|_{L^\infty}
\leq& \RC_0\si\sum\limits_{0\leq m\leq N_1}\|w_{q}\pa_{v_x}^{m}KG^n_1\|_{L^\infty}+
\RC_0\sum\limits_{0\leq m\leq N_0}\|w_q\pa^m_{v_x}\RF\|_{L^\infty}\notag\\
\leq& \RC_0\RC_1\si\sum\limits_{0\leq m\leq N_1}\|w_{q}\pa_{v_x}^{m}G^n_1\|_{L^\infty}+
\RC_0\sum\limits_{0\leq m\leq N_0}\|w_q\pa^m_{v_x}\RF\|_{L^\infty},
\end{align}
where $\RC_1>0$ depends only on $K$. If we choose $\si_\ast>0$ such that $\RC_0\RC_1\si_\ast\leq\frac{1}{2}$, then \eqref{G1-s1-es} implies
\begin{align}\label{G1-s1-es2}
\sum\limits_{0\leq m\leq N_0}\|w_{q}\pa_{v_x}^{m}G^{n}_1\|_{L^\infty}
\leq
2\RC_0\sum\limits_{0\leq m\leq N_0}\|w_q\pa^m_{v_x}\RF\|_{L^\infty},
\end{align}
for any $n\geq0.$ Furthermore, one can also show that for $\si\in[0,\si_\ast]$,
\begin{align}
\|[G^{n+1}_1-G^{n}_1]\|_{\RX_{N_0}}
\leq& \RC_0\RC_1\si\|[G^{n}_1-G^{n-1}_1]\|_{\RX_{N_0}}
\leq\frac{1}{2}\|[G^{n}_1-G^{n-1}_1]\|_{\RX_{N_0}},\label{G1-s1-cau}
\end{align}
which implies that $G_1^{n}\rightarrow G_1$ strongly in $\RX_{N_0}$. In addition, it is easy to see $G_1^{n+1}(-v_x)=-G_1^{n+1}(v_x)$
if $G_1^{n}(-v_x)=-G_1^{n}(v_x)$ holds.
Thus, for $\si\in[0,\si_\ast]$, there exists a unique solution $G_1\in\RX_{N_0}$ to the problem \eqref{G1-it-si1}
and \eqref{G1bd-it-si1}. Actually, the {\it a priori} estimate \eqref{ubdG1} implies that we still have the bound
\begin{align*}
\sum\limits_{0\leq m\leq N_0}\|w_{q}\pa_{v_x}^{m}G_1\|_{L^\infty}
\leq
\RC_0\sum\limits_{0\leq m\leq N_0}\|w_q\pa^m_{v_x}\RF\|_{L^\infty}.
\end{align*}
In other words, it follows
\begin{align}\label{RL-si1}
\|\RL_{\si_\ast}^{-1}\RF\|_{\RX_{N_0}}\leq \RC_0\|\RF\|_{\RX_{N_0}}.
\end{align}

\medskip
\noindent\underline{{\it Step 3. Existence for $\si\in[0,2\si_\ast]$.}} By using \eqref{RL-si1} and performing similar calculations as for obtaining \eqref{G1-s1-es2} and \eqref{G1-s1-cau},  one can see that there exists a unique solution $G_1\in\RX_{N_0}$ to the lifted equation
\begin{align}\notag
\eps G_1+v_y\pa_yG_1+\nu_0G_1-\si_\ast K G_1=\si KG_1+\RF,\quad G_1(\pm1,v)|_{v_y\lessgtr0}=0,
\end{align}
with $\si\in[0,\si_\ast]$.
Therefore, the solution mapping $\RL^{-1}_{2\si_\ast}$ is also well-defined on $\RX_{N_0}$ and the estimate \eqref{G1-rf} holds  for $\si=2\si_\ast$.

Finally, repeating the above procedure step by step,
one can reach $\si=1$ so that $\RL^{-1}_1$ exists and \eqref{G1-rf} also follows simultaneously. This completes the proof of Lemma \ref{ex-G1-rf}.
\end{proof}

\noindent{\it Proof of Proposition \ref{G1.lem}:}
By
setting $\RF=-v_xv_y\sqrt{\mu}$ in Lemma \ref{ex-G1-rf}, we see that for any $\eps>0$ there exists a unique solution $G_1^\eps\in\RX_{N_0}$
to the boundary value problem
\begin{align}\notag
\eps G^\eps_1+v_y\pa_yG^\eps_1+LG^\eps_1=-v_xv_y\sqrt{\mu},\quad
G^\eps_1(\pm1,v)|_{v_y\lessgtr0}=0.
\end{align}
Notice that $G_1^\eps$ satisfies \eqref{G1-pp} and the estimate
\begin{align}
\|G_1^\eps\|_{\RX_{N_0}}\leq \tilde{C}_1,\notag
\end{align}
where $\tilde{C}_1>0$ is independent of $\eps.$
Furthermore, we
define a positive sequence $\{\eps_n\}_{n=1}^\infty$ such that $|\eps_{n+1}-\eps_{n}|\leq 2^{-n}$,
then $\eps_n\rightarrow0^+$ as $n\to+\infty$. We consider the following approximation equations
\begin{align}\notag
\eps_n G^{\eps_n}_1+v_y\pa_yG^{\eps_n}_1+LG^{\eps_n}_1=-v_xv_y\sqrt{\mu},
\end{align}
with
\begin{align}\notag
G^{\eps_n}_1(\pm1,v)|_{v_y\lessgtr0}=0.
\end{align}
Then letting $\bar{\RG}_{n+1}=G^{\eps_{n+1}}_1-G^{\eps_n}_1$, one sees that $\bar{\RG}_{n+1}$ satisfies
\begin{align}
\eps^{n+1} \bar{\RG}_{n+1}+v_y\pa_y\bar{\RG}_{n+1}+LG\bar{\RG}_{n+1}=-(\eps^{n+1}-\eps^{n})G^{\eps_n}_1,\notag
\end{align}
with
\begin{align}
\bar{\RG}_{n+1}|_{v_y\lessgtr0}=0.\notag
\end{align}
Thanks to Lemma \ref{ex-G1-rf}, it follows that
\begin{align}
\|\bar{\RG}_{n+1}\|_{\RX_{N_0}}\leq \RC_0|\eps^{n+1}-\eps^{n}|\|G^{\eps_n}_1\|_{\RX_{N_0}}\leq \RC \tilde{C}_1|\eps^{n+1}-\eps^{n}|.\notag
\end{align}
This means that $\{G^{\eps_n}_1\}_{n=1}^\infty$ is a Cauchy sequence in $\RX_{N_0}$. Thus,
letting $n\to \infty$, the limit function
denoted by $G_1$ is the unique solution of \eqref{G1} and \eqref{G1bd}. Moreover, $G_1$ satisfies \eqref{G1-pp} and the bound \eqref{G1-bdd}. The proof of Proposition \ref{G1.lem} is then completed.
\qed

\section{Steady problem: remainder}\label{sec.spr}

Based on Proposition \ref{G1.lem}, one can further show the following existence result for the remainder $G_R$. Recall the steady problem \eqref{Fst} as well as \eqref{Fst.ex} and \eqref{Fst.ex.mc}.

\begin{proposition}\label{Gr.lem}
The boundary value problem \eqref{Gr} and \eqref{Grbd} admits a unique solution $G_R=G_R(y,v)$ with $\widetilde{G}_R=\sqrt{\mu}G_R$ satisfying
\begin{align}\notag
\int_{-1}^1\int_{\R^3}\widetilde{G}_R(y,v)\,dvdy=0.
\end{align}
And  there is an integer $q_0>0$ such that for any integer $q\geq q_0$, there is $\al_0=\al_0(q)>0$ depending on $q$ such that for any $\al\in (0,\al_0)$ and any integer $m\geq 0$, it holds that
\begin{equation}
\|w_q\pa_{v_x}^m\widetilde{G}_R\|_{L^\infty}\leq \tilde{C}_R,\notag
\end{equation}
where $\tilde{C}_R>0$ is a constant depending only on $m$ and $q$ but independent of $\al$.
%
\end{proposition}


\subsection{Caflisch's decomposition}

To prove Proposition \ref{Gr.lem}, we follow the strategy of the proof in \cite{DL-2020} for treating the shear force term in the framework of perturbation. In fact, notice that there is a  growth term $\frac{\al}{2}v_xv_yG_{R}$ in the equation \eqref{Gr}. To treat this growth in velocity, the key point is to use the Caflisch's decomposition \cite{Ca-1980} and an algebraic weighted estimate introduced  originally by Arkeryd-Esposito-Pulvirenti \cite{AEP-87}. For the purpose, we first decompose the remainder $G_R$ as
\begin{equation}\label{def.split}
\sqrt{\mu}G_R=G_{R,1}+\sqrt{\mu}G_{R,2},
\end{equation}
where $G_{R,1}$ and $G_{R,2}$ satisfy the following two boundary value problems, respectively,
\begin{align}\label{Gr1}
v_y&\pa_yG_{R,1}-\al v_y\pa_{v_x}G_{R,1}+\nu_0 G_{R,1}\notag\\=&\chi_{M}\CK G_{R,1}-\frac{\al}{2}\sqrt{\mu}v_xv_yG_{R,2}
-\frac{1}{2}\sqrt{\mu}v_xv_yG_{1}+\sqrt{\mu}v_y\pa_{v_x}G_{1}
+ Q(\sqrt{\mu}G_1,\sqrt{\mu}G_1)\notag\\&+\al\{Q(\sqrt{\mu}G_R,\sqrt{\mu}G_1)+Q(\sqrt{\mu}G_1,\sqrt{\mu}G_R)\}
+\al^2Q(\sqrt{\mu}G_R,\sqrt{\mu}G_R),
\ y\in(-1,1),\ v\in\R^3,
\end{align}
\begin{align}\label{Gr1bd}
G_{R,1}(\pm1,v)|_{v_y\lessgtr0}=0,\ v\in\R^3,
\end{align}
and
\begin{align}\label{Gr2}
v_y&\pa_yG_{R,2}-\al v_y\pa_{v_x}G_{R,2}+LG_{R,2}=(1-\chi_{M})\mu^{-\frac{1}{2}}\CK G_{R,1},\ y\in(-1,1),\ v\in\R^3,
\end{align}
\begin{align}\label{Gr2bd}
G_{R,2}(\pm1,v)|_{v_y\lessgtr0}=\sqrt{2\pi \mu}\dis{\int_{v_y\gtrless0}}\sqrt{\mu}G_{R}(\pm1,v)|v_y|dv,\ v\in\R^3.
\end{align}
Here $\chi_{M}(v)$ is a non-negative smooth cutoff function such that
\begin{align}
\chi_{M}(v)=\left\{\begin{array}{rll}
1,&\ |v|\geq M+1,\\[2mm]
0,&\ |v|\leq M,
\end{array}\right.\notag
\end{align}
and $\CK$ is defined by \eqref{sp.cL}.
The existence of \eqref{Gr1}, \eqref{Gr1bd}, \eqref{Gr2}  and  \eqref{Gr2bd} can be constructed via the approximation sequence by iteratively solving the following system
\begin{align}\label{Gr1n}
\eps G^{n+1}_{R,1}&+v_y\pa_yG^{n+1}_{R,1}-\al v_y\pa_{v_x}G^{n+1}_{R,1}+\nu_0 G^{n+1}_{R,1}\notag\\=&\chi_{M}\CK G^{n+1}_{R,1}-\frac{\al}{2}\sqrt{\mu}v_xv_yG^{n+1}_{R,2}
-\frac{1}{2}\sqrt{\mu}v_xv_yG_{1}+\sqrt{\mu}v_y\pa_{v_x}G_{1}
+Q(\sqrt{\mu}G_1,\sqrt{\mu}G_1)\notag\\&+\al\{Q(\sqrt{\mu}G^{n}_R,\sqrt{\mu}G_1)+Q(\sqrt{\mu}G_1,\sqrt{\mu}G^{n}_R)\}
+\al^2Q(\sqrt{\mu}G^{n}_R,\sqrt{\mu}G^{n}_R),
\ y\in(-1,1),\ v\in\R^3,
\end{align}
\begin{align}\label{Gr1nbd}
G^{n+1}_{R,1}(\pm1,v)|_{v_y\lessgtr0}=0,\ v\in\R^3,
\end{align}
and
\begin{align}\label{Gr2n}
\eps &G^{n+1}_{R,2}+v_y\pa_yG^{n+1}_{R,2}-\al v_y\pa_{v_x}G^{n+1}_{R,2}+LG^{n+1}_{R,2}=(1-\chi_{M})\mu^{-\frac{1}{2}}\CK G^{n+1}_{R,1},\ y\in(-1,1),\ v\in\R^3,
\end{align}
\begin{align}\label{Gr2nbd}
G^{n+1}_{R,2}(\pm1,v)|_{v_y\lessgtr0}=\sqrt{2\pi \mu}\dis{\int_{v_y\gtrless0}}\sqrt{\mu}G^{n+1}_{R}(\pm1,v)|v_y|dv,\ v\in\R^3,
\end{align}
for a small parameter $\eps>0$, where we have set $[G^{0}_{R,1},G^{0}_{R,2}]=[0,0]$ for $n=0$.

The proof of Proposition \ref{Gr.lem} follows by three steps. First, similarly for treating the existence of $G_1$, we introduce a modified coupled boundary value problem with two parameters $\eps>0$ and $0\leq \si\leq 1$. This
boundary value problem is directly solvable via the characteristic method in case of $\si=0$ corresponding to the homogeneous inflow data, and we then lift the value of $\si$ from $\si=0$ for the zero inflow data to $\si=1$ for the full diffuse reflection boundary condition by a bootstrap argument. Second, we establish the limit $n\rightarrow+\infty$ for any fixed parameter $\eps>0$. Third, we pass the limit $\eps\rightarrow 0^+$ to obtain the
desired solution which satisfies \eqref{Gr1}, \eqref{Gr1bd}, \eqref{Gr2}  and  \eqref{Gr2bd}. As a result, with the help of \eqref{def.split}, we get the solution to the original boundary value problem \eqref{Gr} and \eqref{Grbd}.


\subsection{A priori estimates with  parameters $\epsilon$ and $\si$}

Let us first show that $[G^{n+1}_{R,1},G^{n+1}_{R,2}]$ is well-defined once $[G^{n}_{R,1},G^{n}_{R,2}]$ is given.
To do this, we apply the method of contraction mapping.
We define the  linear vector operator
parameterized by $\si\in[0,1]$ as follows:
$$
\mathscr{L}_{\si}[\CG_1,\CG_2]=[\mathscr{L}^1_{\si},\mathscr{L}^2_{\si}][\CG_1,\CG_2],
$$
where
\begin{eqnarray*}
\mathscr{L}^1_{\si}[\CG_1,\CG_2]=\left\{\begin{array}{rll}
\begin{split}&\eps\CG_1+v_y\pa_y\CG_1-\al v_y\pa_{v_x}\CG_1
+\nu_0\CG_1-\si \chi_{M}\CK G_{1}
+\al\frac{v_xv_y}{2}\sqrt{\mu}\CG_2,\ y\in(-1,1),\\
&\CG_{1}(\pm1,v){\bf 1}_{\{v_y\lessgtr0\}},
\end{split}
\end{array}\right.
\end{eqnarray*}
and
\begin{eqnarray*}
\mathscr{L}^2_{\si}[\CG_1,\CG_2]=\left\{\begin{array}{rll}
\begin{split}&\eps\CG_2+v_y\pa_y\CG_2-\al v_y\pa_{v_x}\CG_2
+\nu_0\CG_2-\si  K\CG_2-\si (1-\chi_{M})\mu^{-\frac{1}{2}}\CK \CG_1,\ y\in(-1,1),\\
&\CG_{2}(\pm1,v){\bf 1}_{\{v_y\lessgtr0\}}-\si \sqrt{2\pi \mu}\dis{\int_{v_y\gtrless0}}(\CG_1+\sqrt{\mu}\CG_2)(\pm1,v)|v_y|dv.
\end{split}
\end{array}\right.
\end{eqnarray*}
We then consider the solvability of the following coupled linear system
\begin{align}\label{pals}
\left\{\begin{array}{rll}
&\mathscr{L}^1_{\si}[\CG_1,\CG_2]=\CF_1,\ \mathscr{L}^2_{\si }[\CG_1,\CG_2]=\CF_2,\ y\in(-1,1),\\[2mm]
&\mathscr{L}^1_{\si }[\CG_1,\CG_2]=0,\ \mathscr{L}^2_{\si }[\CG_1,\CG_2]=\CF_{2,b},\ y=\pm1,
 \end{array}\right.
\end{align}
where $\CF_1$, $\CF_1$ and $\CF_{2,b}$ are given, and $\lag\CF_1,1\rag+\lag\CF_2,\sqrt{\mu}\rag=0$.
In the rest of the proof, for  brevity, we denote
\begin{align}\notag
\tilde{\CF}_1=\left\{\begin{array}{rll}
&\CF_1,\ y\in(-1,1),\\
&0,\ y=\pm1,\end{array}\right.\ \ \
\tilde{\CF}_2=\left\{\begin{array}{rll}
&\CF_2,\ y\in(-1,1),\\
&\CF_{2,b}(\pm1,v),\ y=\pm1.\end{array}\right.
\end{align}

In what follows, we look for solutions to the system \eqref{pals} in the Banach space
\begin{align}
\FX_{\al,N_0}=\bigg\{[\CG_1,\CG_2]&\big|\sum\limits_{0\leq m\leq N_0}\|w_{q}\pa_{v_x}^m[\CG_1,\CG_2]\|_{L^\infty}<+\infty
\bigg\},
\label{def.sX}
\end{align}
supplemented with the norm
$$
\|[\CG_1,\CG_2]\|_{\FX_{\al,N_0}}=\sum\limits_{0\leq m\leq N_0}\left\{\|w_{q}\pa_{v_x}^m\CG_1\|_{L^\infty}
+\|w_{q}\pa_{v_x}^m\CG_2\|_{L^\infty}\right\}.
$$
Let us now deduce the {\it a priori} estimate for the parameterized linear system \eqref{pals}.

\begin{lemma}[{\it a priori} estimate]\label{lifpri}
Let $[\CG_1,\CG_2]\in\FX_{\al,N_0}$ with $\al>0$ and $N_0\geq 0$ be a solution to \eqref{pals} with $\eps>0$ suitably small and $\si \in[0,1]$, and let $[\tilde{\CF}_1,\tilde{\CF}_2]\in\FX_{\al,N_0}$ with $\lag\CF_1,1\rag+\lag\CF_2,\sqrt{\mu}\rag=0$. There is $q_0>0$ such that for any $q\geq q_0$ arbitrarily large, there are $\al_0=\al_0(q)>0$ and large $M=M(q)>0$ such that for any $0<\al<\al_0$,  the solution $[\CG_1,\CG_2]$ satisfies the following estimate
\begin{align}\label{Lif.es1}
\|[\CG_1,\CG_2]\|_{\FX_{\al,N_0}}=&\|\mathscr{L}_{\si}^{-1}[\tilde{\CF}_1,\tilde{\CF}_2]\|_{\FX_{\al,N_0}}\notag\\
\leq& C_\mathscr{L}\sum\limits_{0\leq m\leq N_0}\left\{\|w_{q}\pa_{v_x}^m\CF_1\|_{L^\infty}
+\|w_{q}\pa_{v_x}^m\CF_2\|_{L^\infty}+\|w_q\CF_{2,b}\|_{L^\infty}\right\},
\end{align}
where the constant $C_\mathscr{L}>0$ may depend on $\eps$ but not on $\si$ and $\al$.
\end{lemma}

\begin{proof}
The proof is divided into two steps.

\medskip
\noindent{\bf Step 1. $L^\infty$ estimates.} Let $0\leq m\leq N_0$ and $q>0$, we denote
$$[H_{1,m},H_{2,m}]=[w_{q}\pa_{v_x}^m\CG_1,w_{q}\pa_{v_x}^m\CG_2],$$
then we see that $[H_{1,m},H_{2,m}]$ satisfies
\begin{align}\label{H1k}
\eps H_{1,m}&+v_y\pa_yH_{1,m}-\al v_y\pa_{v_x}H_{1,m}+2q \al \frac{v_yv_x}{{1+|v|^2}}H_{1,m}
+\nu_0H_{1,m}-\si \chi_{M}w_{q}\CK \left(\frac{H_{1,m}}{w_{q}}\right)\notag\\&
-\si {\bf 1}_{m>0}\sum\limits_{1\leq m'\leq m}C_m^{m'}w_{q}\pa^{m'}_{v_x}(\chi_{M}\CK) \pa_{v_x}^{m-m'}\CG_1
\notag\\&+\al \sum\limits_{0\leq m'\leq m}C_m^{m'}w_q\pa_{v_x}^{m'}\left(\frac{v_xv_y}{2}\sqrt{\mu}\right)\pa_{v_x}^{m-m'}\CG_{2}
=w_{q}\pa_{v_x}^m\CF_{1},
\end{align}
\begin{align}\label{H1kbd}
H_{1,m}(\pm1,v)|_{\{v_y\lessgtr0\}}=0,
\end{align}
and
\begin{align}\label{H2k}
\eps H_{2,m}&+v_y\pa_yH_{2,m}-\al v_y\pa_{v_x}H_{2,m}+2q \al \frac{v_yv_x}{{1+|v|^2}}H_{1,m}
+\nu_0H_{2,m}-\si  w_{q}K\left(\frac{H_{2,m}}{w_{q}}\right)
\notag\\&-\si {\bf 1}_{m>0}\sum\limits_{1\leq m'\leq m}C_m^{m'}w_{q}\pa^{m'}_{v_x}K \pa_{v_x}^{m-m'}\CG_2
\notag\\&-\si \sum\limits_{0\leq m'\leq m}C_m^{m'}w_q\pa_{v_x}^{m'}\left((1-\chi_{M})\mu^{-\frac{1}{2}}\right)\pa_{v_x}^{m-m'}\CK \CG_1=w_{q}\pa_{v_x}^m\CF_2,
\end{align}
\begin{align}\label{H2kbd}
H_{2,m}(\pm1,v){\bf 1}_{\{v_y\lessgtr0\}}
-\si w_{q}\pa_{v_x}^m(\sqrt{2\pi \mu})\dis{\int_{v_y\gtrless0}}(\CG_1+\sqrt{\mu}\CG_{2})(\pm1,v)|v_y|dv=w_{q}\pa_{v_x}^{m}\CF_{2,b}(\pm1).
\end{align}
Recall the trajectory $[Y(s;t,y,v),V(s;t,y,v)]$ defined as \eqref{traj}. In addition, for $(y,v)\in[-1,1]\times\R^3$, we define the backward exit time $t_{\Fb}(y,v)$ as
\begin{align}
t_{\mathbf{b}}(y,v)=\inf\{s:y-sv_y\notin(-1,1), s>0\},\label{def-tb}
\end{align}
which is the first time at which the backward characteristic line $[Y(s; t, y, v),$ $V (s; t, y, v)]$
emerges from $(-1,1)$.
Note that at the boundary $y=\pm1$, $t_{\mathbf{b}}(y,v)$ is well-defined if $\pm v_y>0.$  For any $(y,v)$,
we use $t_{\mathbf{b}}(y,v)$ when it is well-defined.
Furthermore, we denote $y_{\mathbf{b}}(y,v)=y-t_{\mathbf{b}}(y,v)v_y\in\{-1,1\},$
and for random variable $v_k$, we define the backward time cycle
\begin{align}\label{cyc}
(t_0,y_0,v_0)=(t,y,v),\ (t_{k+1},y_{k+1},v_{k+1})=(t_k-t_{\mathbf{b}}(y_k,v_k),y_{\mathbf{b}}(y_k,v_k),v_{k+1}), \ k\geq0.
\end{align}
We also set
\begin{eqnarray*}
Y^{l}_{\mathbf{cl}}(s;t,y,v) &=&\mathbf{1}_{[t_{l+1},t_{l})}(s)%
\{y_{l}+(s-t_{l})v_{ly}\}, \\
V^{l}_{\mathbf{cl}}(s;t,y,v) &=&\mathbf{1}%
_{[t_{l+1},t_{l})}(s)(v_{lx}+\al(t_l-s)v_{ly},v_{ly},v_{lz}).
\end{eqnarray*}
Note that $[Y^{0}_{\mathbf{cl}}(s),V^{0}_{\mathbf{cl}}(s)]=[Y(s),V(s)]$ and $y_{l}=\pm1$ for $l\geq1$. Moreover, $t_l$ can be negative.

Define $
\mathcal{V}_{j}=\{v_j\in \R^{3}\ |\ v_j\cdot
n(y_{j})>0\},$
where $n(y_{j})=(0,1,0)$ if $y_{j}=1$ and $n(y_{j})=(0,-1,0)$ if $y_{j}=-1$. Let the iterated integral for $k\geq 2$ be defined as
\begin{equation}\label{def.sil}
\int_{\prod\limits_{l=1}^{k-1}
\mathcal{V}_{l}}\prod\limits_{l=1}^{k-1}d\sigma_{l}\equiv \int_{\mathcal{V}_{1}}\cdots
\left\{ \int_{\mathcal{V}%
_{k-1}}d\sigma_{k-1}\right\} d\sigma_{1},
\end{equation}%
where $d\sigma _{l}=\sqrt{2\pi}\mu (v_l)|v_{ly}|dv_l$ is a probability
measure.

Without lost of generality, we assume $\lim\limits_{|t|\rightarrow+\infty}[H_{1,m},H_{2,m}](t)=0$.
Along the characteristic line \eqref{chl}, for $(y,v)\in[-1,1]\times\R^3\backslash(\ga_-\cup\ga_0)$, we write the solution of the system \eqref{H1k} and \eqref{H1kbd}
in  the  mild form as follows:
\begin{align}
H_{1,m}(t)=&H_{1,m}(y(t),v(t))
\notag\\=&\si \int_{t_1}^{t}e^{-\int_{s}^t\CA^\eps(\tau,V(\tau))d\tau}\left\{\chi_{M}w_{q}\CK
\left(\frac{H_{1,m}}{w_{q}}\right)\right\}(Y(s),V(s))\,ds\notag\\
&+\si {\bf1}_{m>0}\int_{t_1}^{t}e^{-\int_{s}^t\CA^\eps(\tau,V(\tau))d\tau}\sum\limits_{1\leq m'\leq m}C_m^{m'}\left\{w_{q}\pa^{m'}_{v_x}(\chi_{M}\CK) \pa_{v_x}^{m-m'}\CG_1\right\}(Y(s),V(s))\,ds\notag
\\
&-\al\int_{t_1}^{t}e^{-\int_{s}^t\CA^\eps(\tau,V(\tau))d\tau}
\sum\limits_{0\leq m'\leq m}C_m^{m'}\left\{w_q\pa_{v_x}^{m'}
\left(\frac{v_xv_y}{2}\sqrt{\mu}\right)\pa_{v_x}^{m-m'}\CG_{2}\right\}(Y(s),V(s))\,ds\notag\\
&+\int_{t_1}^{t}e^{-\int_{s}^t\CA^\eps(\tau,V(\tau))d\tau}\left(w_{q}\pa_{v_x}^m\CF_1\right)(Y(s),V(s))\,ds,\notag
\end{align}
where
\begin{align}
\CA^\eps(\tau,V(\tau))=\nu_0+\eps+2q \al \frac{V_y(\tau)V_x(\tau)}{{1+|V(\tau)|^2}}\geq\nu_0/2,\label{CAe-lbd}
\end{align}
provided that $\eps>0$ and $q\al>0$ are suitably small.
By Lemma \ref{CK}, it is straightforward to see
\begin{align}\label{H1m.p1}
\sup\limits_{-\infty<t<+\infty}\|H_{1,m}(t)\|_{L^\infty}\leq& \frac{C}{q}\sum\limits_{m'\leq m}\sup\limits_{-\infty<t<+\infty}\|H_{1,m'}(t)\|_{L^\infty}+C\al \sum\limits_{m'\leq m}\sup\limits_{-\infty<t<+\infty}\|H_{2,m'}(t)\|_{L^\infty}
\notag\\&+\sup\limits_{-\infty<t<+\infty}\|w_{q}\pa_{v_x}^m\CF_1(t)\|_{L^\infty}.
\end{align}
By taking $q$ sufficiently large, \eqref{H1m.p1} further gives
\begin{align}\label{H1sum}
\sum\limits_{0\leq m\leq N_0}\|H_{1,m}\|_{L^\infty}\leq C\al\sum\limits_{0\leq m\leq N_0}\|H_{2,m}\|_{L^\infty}
+\sum\limits_{0\leq m\leq N_0}\|w_{q}\pa_{v_x}^m\CF_1\|_{L^\infty}.
\end{align}
Similarly, one can also write the solution of \eqref{H2k} and \eqref{H2kbd} in the mild form of
\begin{align}\label{H2m}
H_{2,m}(t)
=&\underbrace{\si \int_{t_1}^{t}e^{-\int_{s}^t\CA^\eps(\tau,V(\tau))d\tau}\left\{w_{q}K
\left(\frac{H_{2,m}}{w_{q}}\right)\right\}(Y(s),V(s))\,ds}_{I_1}\notag\\
&+\underbrace{\si {\bf 1}_{m>0}\int_{t_1}^{t}e^{-\int_{s}^t\CA^\eps(\tau,V(\tau))d\tau}
\sum\limits_{1\leq m'\leq m}C_m^{m'}\left\{w_{q}(\pa^{m'}_{v_x}K) \pa_{v_x}^{m-m'}\CG_2\right\}(Y(s),V(s))\,ds}_{I_2}\notag
\\
&+\underbrace{\si \int_{t_1}^{t}e^{-\int_{s}^t\CA^\eps(\tau,V(\tau))d\tau}
\sum\limits_{0\leq m'\leq m}C_m^{m'}\left\{w_q\pa_{v_x}^{m'}\left((1-\chi_{M})\mu^{-\frac{1}{2}}\right)\pa_{v_x}^{m-m'}\CK \CG_1\right\}(Y(s),V(s))\,ds}_{I_3}\notag\\
&+\underbrace{\int_{t_1}^{t}e^{-\int_{s}^t\CA^\eps(\tau,V(\tau))d\tau}\left(w_{q}\pa_{v_x}^m\CF_2\right)(Y(s),V(s))\,ds
+e^{-\int_{t_1}^t\CA^\eps(\tau,V(\tau))d\tau}\left(w_{q}\pa_{v_x}^m\CF_{2,b}\right)(t_1,y_1,V(t_1))}_{I_4}\notag\\
&+\sum\limits_{n=5}^{10}I_n,
\end{align}
where $(y,v)\in[-1,1]\times\R^3\backslash(\ga_-\cup\ga_0)$, and for $k\geq2$,
\begin{align}
I_5=&\si ^{k-1}\underbrace{\sqrt{2\pi}e^{-\int_{t_1}^t\CA^\eps(\tau,V(\tau))d\tau}
\left[w_{q}\pa_{v_x}^m(\sqrt{\mu})\right](V(t_1))}_{\CW}\int_{\prod\limits_{j=1}^{k-1}\CV_j}
(w_q\CG_2)(t_k,y_k,V_{\mathbf{cl}}^{k-1}(t_k))\,d\Sigma_{k-1}(t_k),\notag
\end{align}
\begin{align}
I_6=&\sum\limits_{l=2}^{k-1}\si ^{l-1}\CW\int_{\prod\limits_{j=1}^{k-1}\CV_j}
(w_q\CF_{2,b})(t_l,y_l,V_{\mathbf{cl}}^{l-1}(t_l))\,d\Sigma_{l}(t_l),\notag
\end{align}
\begin{align}
I_7=&\si ^{l}\sum\limits_{l=1}^{k-1}\CW\int_{\prod\limits_{j=1}^{k-1}\CV_j}
\int_{t_{l+1}}^{t_l}w_q\CF_2(Y^{l}_{\mathbf{cl}},V^{l}_{\mathbf{cl}})(s)\,d\Sigma_{l}(s)ds,\notag
\end{align}
\begin{align}
I_8=& \si ^{l}\sum\limits_{l=1}^{k-1}\CW\int_{\prod\limits_{j=1}^{k-1}\CV_j}
\int_{t_{l+1}}^{t_l}\left\{w_qK\left(\frac{H_{2,0}}{w_{q}}\right)\right\}(Y^{l}_{\mathbf{cl}},V^{l}_{\mathbf{cl}})(s)\,d\Sigma_{l}(s)ds,\notag
\end{align}
\begin{align}
I_9=& \si ^{l}\sum\limits_{l=1}^{k-1}\CW\int_{\prod\limits_{j=1}^{k-1}\CV_j}
\int_{t_{l+1}}^{t_l}\left\{(1-\chi_{M})w_{q}\mu^{-\frac{1}{2}}\CK
\left(\frac{H_{1,0}}{w_{q}}\right)\right\}(Y^{l}_{\mathbf{cl}},V^{l}_{\mathbf{cl}})(s)\,d\Sigma_{l}(s)ds,\notag
\end{align}
\begin{align}
I_{10}=&\si^l\CW \sum\limits_{l=1}^{k-1}\int_{\prod\limits_{j=1}^{k-1}\CV_j}\left(\frac{w_q}{\sqrt{\mu}}\CG_1\right)(t_l,y_l,V_{\mathbf{cl}}^{l-1}(t_l))d\Sigma_{l}(t_l).\notag
\end{align}
Moreover, in the above expressions we have used the following notations
\begin{equation}\label{Sigma}
\Sigma_{l}(s)=\prod\limits_{j=l+1}^{k-1}d\si_j e^{-\int_s^{t_l}
\CA^\eps(\tau,V_{\mathbf{cl}}^l(\tau))d\tau}\tilde{w}_2(v_l)d\si_l \prod\limits_{j=1}^{l-1}\frac{\tilde{w}_2(v_j)}{\tilde{w}_2(V^{j}_{\mathbf{cl}}(t_{j+1}))}
\prod\limits_{j=1}^{l-1}e^{-\int_{t_{j+1}}^{t_j}
\CA^\eps(\tau,V_{\mathbf{cl}}^l(\tau))d\tau}d\si_j,
\end{equation}
and
\begin{align}\label{wt-2}
\tilde{w}_2(v)=(\sqrt{2\pi}w_{q}\sqrt{\mu})^{-1}.
\end{align}

The $L^\infty$ estimates for $H_{2,m}$ is more complicated because $K$ has no smallness property. To overcome this, we have to iterate \eqref{H2m} twice. Let us first compute $I_n$ $(1\leq n\leq10)$ term by term.
Recalling the definition \eqref{kw-def} for ${\bf k}_{w}$, one directly has by \eqref{CAe-lbd}
\begin{align*}
|I_1|\leq \int_{t_1}^te^{-\frac{\nu_0}{2}(t-s)}\int_{\R^3}{\bf k}_{w}(V(s),v')|H_{2,m}(s,Y(s),v')|dv'ds.
\end{align*}
By Lemma \ref{Ga}, it follows
\begin{align*}
|I_2|\leq C{\bf 1}_{m>0}\sum\limits_{m'\leq m-1}\|w_q\pa_{v_x}^{m'}\CG_2\|_{L^\infty}
\int_{t_1}^te^{-\frac{\nu_0}{2}(t-s)}ds\leq C{\bf 1}_{m>0}\sum\limits_{m'\leq m-1}\|w_q\pa_{v_x}^{m'}\CG_2\|_{L^\infty},
\end{align*}
and similarly,
\begin{align*}
|I_3|\leq C\sum\limits_{m'\leq m}\|w_q\pa_{v_x}^{m'}\CG_1\|_{L^\infty}.
\end{align*}
It is straightforward to see
\begin{align*}
|I_4|\leq C\|w_q\pa_{v_x}^{m}\CF_2\|_{L^\infty}+C\|w_q\pa_{v_x}^{m}\CF_{2,b}\|_{L^\infty}.
\end{align*}
Next, notice that
\begin{align}
|\CW|\leq Cm!4^qq!e^{-\frac{\nu_0(t-t_1)}{2}}.
\notag
\end{align}
In the sequel, for simplicity, we denote $\CC_{m,q}$ for the constant $m!4^qq!$.
By Lemma \ref{k.cyc}, it follows
\begin{align*}
|I_5|\leq C\CC_{m,q}2^{-C_2T^{\frac{5}{4}}_0} e^{-\frac{\nu_0}{2}(t-t_1)}\|H_{2,0}\|_{L^\infty},\ |I_6|+|I_7|\leq Ck\CC_{m,q}e^{-\frac{\nu_0}{2}(t-t_1)}\left\{\|w_q\CF_2\|_{L^\infty}+\|w_q\CF_{2,b}\|_{L^\infty}\right\},
\end{align*}
\begin{align*}
\ |I_9|,\ |I_{10}|\leq C\CC_{m,q}ke^{-\frac{\nu_0}{2}(t-t_1)}\|H_{1,0}\|_{L^\infty},
\end{align*}
and
\begin{align*}
|I_8|\leq C\CC_{m,q}e^{-\frac{\nu_0}{2}(t-t_1)}\int_{\prod\limits_{j=1}^{k-1}\CV_j}
\int_{t_{l+1}}^{t_l}\int_{\R^3}{\bf k}_{w}(V^{l}_{\mathbf{cl}}(s),v')|H_{2,0}(s,Y^{l}_{\mathbf{cl}}(s),v')|dv'\,d\Sigma_{l}(s)ds,
\end{align*}
where we have taken $T_0=t-t_k$ with $k=C_1T^{\frac{5}{4}}_0$, and  both $C_1>0$ and $C_2>0$ are given in Lemma \ref{k.cyc}.

Putting all the estimates for $I_n$ $(1\leq n\leq 10)$ above together and adjusting the constants, we have
\begin{align}\label{H2sum1}
|H_{2,m}(t)|
\leq& C\CC_{m,q}e^{-\frac{\nu_0}{2}(t-t_1)}\int_{t_1}^te^{-\frac{\nu_0}{2}(t-s)}\int_{\R^3}{\bf k}_{w}(V(s),v')|H_{2,m}(s,Y(s;t,y,v),v')|dv'ds\notag\\
&+C\CC_{m,q}e^{-\frac{\nu_0}{2}(t-t_1)}\sum\limits_{l=1}^{k-1}\int_{\prod\limits_{j=1}^{k-1}\CV_j}
\int_{t_{l+1}}^{t_l}\int_{\R^3}{\bf k}_{w}(V^{l}_{\mathbf{cl}}(s),v')|H_{2,0}(s,Y^{l}_{\mathbf{cl}}(s;t,y,v),v')|dv'\,d\Sigma_{l}(s)ds
\notag\\&+\CQ(t),
\end{align}
where
\begin{align}
\CQ(t)=&C{\bf 1}_{m>0}\sum\limits_{m'\leq m-1}\sup\limits_{-\infty< s\leq t}\|H_{2,m'}(s)\|_{L^\infty}
+C\sum\limits_{m'\leq m}\sup\limits_{-\infty< s\leq t}\|H_{1,m'}(s)\|_{L^\infty}
\notag\\&+C\sup\limits_{-\infty< s\leq t}\|w_q\pa_{v_x}^{m}\CF_2(s)\|_{L^\infty}
+C\sup\limits_{-\infty< s\leq t}\|w_q\pa_{v_x}^{m}\CF_{2,b}(s)\|_{L^\infty}
\notag\\&+C\CC_{m,q}2^{-C_2T^{\frac{5}{4}}_0} \sup\limits_{-\infty< s\leq t}\|H_{2,0}(s)\|_{L^\infty}
+C\CC_{m,q}k\sup\limits_{-\infty< s\leq t}\|H_{1,0}(s)\|_{L^\infty}
\notag\\&+C\CC_{m,q}k\sup\limits_{-\infty< s\leq t}\|w_q\CF_2(s)\|_{L^\infty}
+C\CC_{m,q}k\sup\limits_{-\infty< s\leq t}\|w_q\CF_{2,b}(s)\|_{L^\infty}.\notag
\end{align}
Then let us define a new  backward time cycle as
\begin{align}\notag
(t_{\ell+1}',y_{\ell+1}',v_{\ell+1}')=(t_{\ell}'-t_{\mathbf{b}}(y_{\ell}',v_\ell'),y_{\mathbf{b}}(y_\ell',v_\ell'),v_{\ell+1}'),
\end{align}
and the starting point
\begin{align}\notag
(t_{0}',y_{0}',v_{0}')=(s,y',v'):=(s,Y(s),v')\ \textrm{or}\ (s,Y^l_{\mathbf{cl}}(s),v'),
\end{align}
for some $s\in\R$ and $l\in\Z^+$. Furthermore, for $\ell\in\Z^+$, we also denote
\begin{eqnarray*}
\bar{Y}^\ell_{\mathbf{cl}}(s';s,y',v') &=&\mathbf{1}_{[t'_{\ell+1},t'_{\ell})}(s')\{y'_\ell+(s'-t'_\ell)v'_{\ell y}\}, \\
\bar{V}^{\ell}_{\mathbf{cl}}(s';s,y',v') &=&\mathbf{1}_{[t'_{\ell+1},t'_{\ell})}(s')(v'_{\ell x}
+\al(t_\ell'-s')v_{\ell y},v'_{\ell y},v'_{\ell z}).
\end{eqnarray*}
To be consistent, we set $[\bar{Y}^{0}_{\mathbf{cl}}(s'),\bar{V}^{0}_{\mathbf{cl}}(s')]:=[\bar{Y}(s'),\bar{V}(s')].$

Iterating \eqref{H2sum1} again, one has
\begin{align}\label{H2sum2}
|H_{2,m}(t)|\leq& C\CC_{m,q}^2\int_{t_1}^te^{-\frac{\nu_0}{2}(t-s)}\int_{\R^3}{\bf k}_{w}(V(s),v')
\int_{t'_1}^{s}e^{-\frac{\nu_0}{2}(s-s')}\int_{\R^3}{\bf k}_{w}(\bar{V}(s';Y(s),v'),v'')
\notag\\&\quad\times|H_{2,m}(s',\bar{Y}(s';Y(s),v'),v'')|~dv''ds'dv'ds\notag\\
&+C\CC_{m,q}^2\int_{t_1}^te^{-\frac{\nu_0}{2}(t-s)}\int_{\R^3}{\bf k}_{w}(V(s),v')e^{-\frac{\nu_0}{2}(s-t_1')}
\sum\limits_{\ell=1}^{\imath-1}\int_{\prod\limits_{\jmath=1}^{\imath-1}\CV_\jmath}
\int_{t'_{\ell+1}}^{t'_\ell}\int_{\R^3}{\bf k}_{w}(\bar{V}^{\ell}_{\mathbf{cl}}(s';Y(s),v'),v'')
\notag
\\&\qquad\times|H_{2,0}(s',\bar{Y}^{\ell}_{\mathbf{cl}}(s';Y(s),v'),v'')|dv''\,d\Sigma_{\ell}(s')ds'dv'ds\notag\\
&+C\CC_{m,q}\CC_{0,q}\sum\limits_{l=1}^{k-1}\int_{\prod\limits_{j=1}^{k-1}\CV_j}
\int_{t_{l+1}}^{t_l}\int_{\R^3}{\bf k}_{w}(V^{l}_{\mathbf{cl}}(s),v')
\int_{t'_1}^{s}e^{-\frac{\nu_0}{2}(s-s')}\int_{\R^3}{\bf k}_{w}(\bar{V}(s';Y^{l}_{\mathbf{cl}}(s),v'),v'')
\notag\\&\quad\times|H_{2,0}(s',\bar{Y}(s';Y^{l}_{\mathbf{cl}}(s),v'),v'')|~dv''ds'dv'\,d\Sigma_{l}(s)ds
\notag\\&+C\CC_{m,q}\CC_{0,q}\sum\limits_{l=1}^{k-1}\int_{\prod\limits_{j=1}^{k-1}\CV_j}
\int_{t_{l+1}}^{t_l}\int_{\R^3}{\bf k}_{w}(V^{l}_{\mathbf{cl}}(s;v),v')e^{-\frac{\nu_0}{2}(s-t_1')}
\sum\limits_{\ell=1}^{\imath-1}\int_{\prod\limits_{\jmath=1}^{\imath-1}\CV_\jmath}
\int_{t'_{\ell+1}}^{t'_\ell}\int_{\R^3} \notag\\
&\quad\times {\bf k}_{w}(\bar{V}^{\ell}_{\mathbf{cl}}(s';Y^{l}_{\mathbf{cl}}(s),v'),v'')
|H_{2,0}(s',\bar{Y}^{\ell}_{\mathbf{cl}}(s';Y^{l}_{\mathbf{cl}}(s),v'),v'')|dv''\,d\Sigma_{\ell}(s')ds'dv'\,d\Sigma_{l}(s)ds
\notag\\
&+C\CC_{m,q}\CC_{0,q}\int_{t_1}^te^{-\frac{\nu_0}{2}(t-s)}\int_{\R^3}{\bf k}_{w}(V(s),v')\CQ(s)dv'ds
\notag\\&+C\CC_{m,q}^2\sum\limits_{l=1}^{k-1}\int_{\prod\limits_{j=1}^{k-1}\CV_j}
\int_{t_{l+1}}^{t_l}\int_{\R^3}{\bf k}_{w}(V^{l}_{\mathbf{cl}}(s),v')\CQ(s)dv'\,d\Sigma_{l}(s)ds,
\end{align}
where according to Lemma \ref{k.cyc} and Remark \ref{ed-cyc}, we choose $\imath\in\Z^+$ such that $\imath\sim(\tilde{T}_0)^{\frac{5}{4}}$ with $\tilde{T}_0=s-t'_{\imath}$ being suitably large.
We claim that
\begin{align}\label{H2sum3}
\|H_{2,m}\|_{L^\infty}\leq \eta\left\{\|H_{2,m}\|_{L^\infty}+\|H_{2,0}\|_{L^\infty}\right\}
+C(\tilde{T}_0)\left\{\|\pa_{v_x}^m\CG_2\|+\|\CG_2\|\right\}
+C\sup\limits_{s\leq t}\CQ(s),
\end{align}
where $\eta>0$ is suitably small.
To prove \eqref{H2sum3}, we only estimate the fourth term on the right hand side of \eqref{H2sum2}, because the other terms can be estimated similarly. For any sufficiently small $\eta_0>0$, we first divide $[t_{\ell+1}',t_{\ell}']$ as $[t_{\ell+1}',t_{\ell}'-\eta_0]\cup(t_{\ell}'-\eta_0,t_{\ell}']$, then rewrite the fourth term on the right hand side of \eqref{H2sum2} as
\begin{align}
\CJ:=& C\CC_{m,q}\CC_{0,q}\sum\limits_{l=1}^{k-1}\int_{\prod\limits_{j=1}^{k-1}\CV_j}
\int_{t_{l+1}}^{t_l}\int_{\R^3}{\bf k}_{w}(V^{l}_{\mathbf{cl}}(s;v),v')e^{-\frac{\nu_0}{2}(s-t_1')}
\sum\limits_{\ell=1}^{\imath-1}\int_{\prod\limits_{\jmath=1}^{\imath-1}\CV_\jmath}
\left(\int_{t'_{\ell+1}}^{t'_\ell-\eta_0}+\int_{t'_\ell-\eta_0}^{t'_\ell}\right)\notag\\
&\times\int_{\R^3} {\bf k}_{w}(\bar{V}^{\ell}_{\mathbf{cl}}(s';Y^{l}_{\mathbf{cl}}(s),v'),v'')
|H_{2,0}(s',\bar{Y}^{\ell}_{\mathbf{cl}}(s';Y^{l}_{\mathbf{cl}}(s),v'),v'')|dv''\,d\Sigma_{\ell}(s')ds'dv'\,d\Sigma_{l}(s)ds\notag\\
\eqdef& \CJ_1+\CJ_2.\notag
\end{align}
By Lemma \ref{k.cyc}, it is easy to see
\begin{align*}
\CJ_2\leq C\CC_{m,q}\CC_{0,q}\eta_0\imath\|H_{2,0}\|_{L^\infty}.
\end{align*}
For  $\CJ_1$, the computation is divided into the following three cases.

\noindent\underline{{\it Case 1. $|V^{l}_{\mathbf{cl}}(s;v)|>M$ or $|\bar{V}^{\ell}_{\mathbf{cl}}(s';Y^{l}_{\mathbf{cl}}(s),v')|> M$.}} In this case, by  Lemma \ref{Kop}, it follows that
\begin{align}
\int_{\R^3}|{\bf k}_{w}(V^{l}_{\mathbf{cl}}(s;v),v')|dv',\ \textrm{or}\ \int_{\R^3}|{\bf k}_{w}(\bar{V}^{\ell}_{\mathbf{cl}}(s';Y^{l}_{\mathbf{cl}}(s),v'),v'')|dv''\leq\frac{C(q)}{1+M},\notag
\end{align}
where $C(q)>0$ and depends on $q!$.
Therefore, one has by using Lemma \ref{k.cyc} again
\begin{align}
|\CJ_1|\leq\frac{C\CC_{m,q}\CC_{0,q}C^2(q)}{1+M}\|H_{2,0}\|_{L^\infty}.\notag
\end{align}
Note that here and in the sequel, $l$ and $\ell$ run over $[1,k-1]$ and $[1,\imath-1]$, respectively.

\noindent\underline{{\it Case 2. $|V^{l}_{\mathbf{cl}}(s;v)|\leq M$ and $|v'|> 2M$, or $|\bar{V}^{\ell}_{\mathbf{cl}}(s';Y^{l}_{\mathbf{cl}}(s),v')|\leq M$ and $|v''|> 2M$.}} In this regime, we have either $|V^{l}_{\mathbf{cl}}(s;v)-v'|>M$ or $|\bar{V}^{\ell}_{\mathbf{cl}}(s';Y^{l}_{\mathbf{cl}}(s),v')-v''|>M$. Then either of the following two estimates holds correspondingly
\begin{equation*}
\begin{split}
&{\bf k}_{w}(V^{l}_{\mathbf{cl}}(s;v),v')
\leq Ce^{-\frac{\vps M^2}{16}}{\bf k}_{w}(V^{l}_{\mathbf{cl}}(s;v),v')e^{\frac{\vps |V^{l}_{\mathbf{cl}}-v'|^2}{16}},\\
&{\bf k}_{w}(\bar{V}^{\ell}_{\mathbf{cl}}(s';Y^{l}_{\mathbf{cl}}(s),v'),v'')
\leq Ce^{-\frac{\vps M^2}{16}}{\bf k}_{w}(\bar{V}^{\ell}_{\mathbf{cl}}(s';Y^{l}_{\mathbf{cl}}(s),v'),v'')e^{\frac{\vps |\bar{V}^{\ell}_{\mathbf{cl}}-v''|^2}{16}}.
\end{split}
\end{equation*}
This together with Lemma \ref{Kop} imply
\begin{align}
|\CJ_1|\leq C\CC_{m,q}\CC_{0,q}C^2(q)e^{-\frac{\vps M^2}{16}}\|H_{2,0}\|_{L^\infty}.\notag
\end{align}

\noindent\underline{{\it Case 3. $|V^{l}_{\mathbf{cl}}(s;v)|\leq M$, $|v'|\leq2M$, $|\bar{V}^{\ell}_{\mathbf{cl}}(s';Y^{l}_{\mathbf{cl}}(s),v')|\leq M$ and $|v''|\leq 2M$.}} The key point here is to convert  the $L^1$ integral with respect to the double $v$ variables into the $L^2$ norm with respect to the variables $y$ and $v$. To do so, for any large $N>0$,
we choose a number $M(N)$ to define ${\bf k}_{w,M}(u,v')$ as \eqref{kw-M}, then decompose
\begin{equation*}
\begin{split}
{\bf k}_{w}(V^{l}_{\mathbf{cl}},v')
{\bf k}_{w}(\bar{V}^{\ell}_{\mathbf{cl}},v^{\prime \prime})
=&\{{\bf k}_{w}(V^{l}_{\mathbf{cl}},v^{\prime})-{\bf k}_{w,M}(V^{l}_{\mathbf{cl}},v^{\prime})\}{\bf k}_{w}(\bar{V}^{\ell}_{\mathbf{cl}},v^{\prime \prime})
\\&+\{{\bf k}_{w}(\bar{V}^{\ell}_{\mathbf{cl}},v^{\prime \prime})
-{\bf k}_{w,M}(\bar{V}^{\ell}_{\mathbf{cl}},v^{\prime \prime})\}{\bf k}_{w,M}(V^{l}_{\mathbf{cl}},v')+{\bf k}_{w,M}(V^{l}_{\mathbf{cl}},v^{\prime}){\bf k}_{w,M}(\bar{V}^{\ell}_{\mathbf{cl}},v^{\prime \prime}).
\end{split}
\end{equation*}
From Lemma \ref{k.cyc}, the first two difference terms lead to a small
contribution in $\CJ_1$ as
\begin{equation}
\frac{C\CC_{m,q}\CC_{0,q}C^2(q)}{N}\|H_{2,0}\|_{L^\infty}.  \notag
\end{equation}
For the remaining main contribution of the bounded product ${\bf k}_{w,M}(V^{l}_{\mathbf{cl}},v^{\prime}){\bf k}_{w,M}(\bar{V}^{\ell}_{\mathbf{cl}},v^{\prime \prime})$, we denote $\tilde{y}=\bar{Y}^{\ell}_{\mathbf{cl}}(s';Y^{l}_{\mathbf{cl}}(s),v')=y_\ell'-(t_\ell'-s')v'_{\ell y}$ and apply a change of variable $v_{ly}'\rightarrow\tilde{y}$. Then one has
\begin{align}
|\frac{\pa \tilde{y}}{\pa v_{ly}'}|=\left|\frac{\pa(y_\ell'-(t_\ell'-s')v'_{\ell y})}{\pa v_{ly}'}\right|
=|t_\ell'-s'|\geq \eta_0.\notag
\end{align}
We now estimate this part as follows
\begin{align}
C\CC_{m,q}\CC_{0,q}&\sum\limits_{l=1}^{k-1}\int_{\prod\limits_{j=1}^{k-1}\CV_j}
\int_{t_{l+1}}^{t_l}\int_{|v'|\leq2M}{\bf k}_{w,M}(V^{l}_{\mathbf{cl}}(s;v),v')e^{-\frac{\nu_0}{2}(s-t_1')}\sum\limits_{\ell=1}^{\imath-1}
\int_{\prod\limits_{\jmath=1}^{\imath-1}\CV_\jmath}
\int_{t'_{\ell+1}}^{t'_\ell-\eta_0}\notag\\
&\times\int_{|v''|\leq2M} {\bf k}_{w,M}(\bar{V}^{\ell}_{\mathbf{cl}}(s';Y^{l}_{\mathbf{cl}}(s),v'),v'')
|H_{2,0}(s',\bar{Y}^{\ell}_{\mathbf{cl}}(s';Y^{l}_{\mathbf{cl}}(s),v'),v'')|dv''\,d\Sigma_{\ell}(s')ds'dv'\,d\Sigma_{l}(s)ds\notag\\
\leq &C(M,m,q)\sum\limits_{l=1}^{k-1}\int_{\prod\limits_{j=1}^{k-1}\CV_j}
\int_{t_{l+1}}^{t_l}e^{-\frac{\nu_0}{2}(s-t_1')}\sum\limits_{\ell=1}^{\imath-1}\int_{\prod\limits_{\jmath=1}^{\imath-1}\CV_\jmath}
\int_{t'_{\ell+1}}^{t'_\ell-\eta_0}\notag\\
&\times\int_{|v'|\leq 2M, |v''|\leq2M}
|\CG_2(s',\bar{Y}^{\ell}_{\mathbf{cl}}(s';Y^{l}_{\mathbf{cl}}(s),v'),v'')|dv''dv'\,d\Sigma_{\ell}(s')ds'\,d\Sigma_{l}(s)ds\notag\\
\leq& \frac{C(M,m,q)}{\sqrt{\eta_0}}\sup\limits_{s\leq t}\|\CG_2(s)\|\sup\limits_{v,v'}\left\{\int_{t_k}^{t_1}e^{-\frac{\nu_0(t_1-s)}{2}}e^{-\frac{\nu_0}{2}(s-t_1')}
\int_{t'_\imath}^{t'_1}e^{-\frac{\nu_0(t'_1-s')}{2}} ds'ds\right\}\notag\\
\leq& \frac{C(M,m,q)}{\sqrt{\eta_0}}\sup\limits_{s\leq t}\|\CG_2(s)\|.\notag
\end{align}

Putting all the estimates for $\CJ_1$ and $\CJ_2$ together, we now obtain
\begin{align}
\CJ\leq& C\CC_{m,q}\CC_{0,q}\eta_0\imath\|H_{2,0}\|_{L^\infty}+\frac{C\CC_{m,q}\CC_{0,q}C^2(q)}{1+M}\|H_{2,0}\|_{L^\infty}+
C\CC_{m,q}\CC_{0,q}C^2(q)e^{-\frac{\vps M^2}{16}}\|H_{2,0}\|_{L^\infty}\notag\\&+\frac{C\CC_{m,q}\CC_{0,q}C(q)}{N}\|H_{2,0}\|_{L^\infty}
+\frac{C(M,m,q)}{\sqrt{\eta_0}}\sup\limits_{s\leq t}\|\CG_2(s)\|.\notag
\end{align}

As mentioned before, by performing the similar calculations for the other terms on the right hand side of \eqref{H2sum2}, one has
\begin{align}
\|H_{2,m}\|_{L^\infty}\leq& \frac{C\CC_{m,q}\CC_{0,q}C^2(q)}{1+M}\|H_{2,m}\|_{L^\infty}
+C\CC_{m,q}\CC_{0,q}C^2(q)e^{-\frac{\vps M^2}{16}}\|H_{2,m}\|_{L^\infty}\notag\\&+C\CC_{m,q}\CC_{0,q}\eta_0\imath\|H_{2,0}\|_{L^\infty}
+\frac{C\CC_{m,q}\CC_{0,q}(4^qq!)^2}{1+M}\|H_{2,0}\|_{L^\infty}
\notag\\&+C\CC_{m,q}\CC_{0,q}C^2(q)e^{-\frac{\vps M^2}{16}}\|H_{2,0}\|_{L^\infty}
+\frac{C\CC_{m,q}\CC_{0,q}C(q)}{N}\|H_{2,0}\|_{L^\infty}
\notag\\&+\frac{C(M,m,q)}{\sqrt{\eta_0}}\sup\limits_{s\leq t}\|\CG_2(s)\|+C(M,m,q)\sup\limits_{s\leq t}\|\pa_{v_x}^m\CG_2(s)\|
+C\sup\limits_{s\leq t}\CQ(s).
\label{H2sum4}
\end{align}
Since $\imath\sim(\tilde{T}_0)^{\frac{5}{4}}$, by taking $M$ and $N$ large enough and $\eta_0=(\tilde{T}_0)^{-\frac{5}{2}}$ small enough, \eqref{H2sum4} further yields \eqref{H2sum3}.
Finally, taking a linear combination of \eqref{H2sum3} with $m=0,1,\cdots,N_0$, we conclude
\begin{align}\label{H2sum5}
\sum\limits_{0\leq m\leq N_0}\|H_{2,m}\|_{L^\infty}\leq& C(N_0,q,\tilde{T}_0)\sum\limits_{0\leq m\leq N_0}\|\pa_{v_x}^m\CG_2\|
+C(N_0,q,T_0)\sum\limits_{0\leq m\leq N_0}\|H_{1,m}\|_{L^\infty}
\notag\\&+C(N_0,q,T_0)\sum\limits_{0\leq m\leq N_0}\left\{\|w_q\pa_{v_x}^{m}\CF_2\|_{L^\infty}
+\|w_q\pa_{v_x}^{m}\CF_{2,b}\|_{L^\infty}\right\}.
\end{align}
\begin{remark}\label{ideps}
We point out that the estimates \eqref{H1sum} and \eqref{H2sum5} obtained above are independent of $\eps.$
Moreover, both $\tilde{T}_0$ and $T_0$ are independent of $t,$ because starting from any $t\in(-\infty,+\infty)$
we can  trace back $k$ times to some $t_k$ which can be negative.
\end{remark}

\noindent{\bf Step 2. $L^2$ estimate.} To close the final estimate, we turn to obtain the $L^2$ estimate of $\pa_{v_x}^m\CG_2$ with $0\leq m\leq N_0$. The goal is to prove that for given $\eps>0$ there exists $C(\eps)>0$ depending on $\eps$ such that
\begin{align}\label{g2.tl2}
\sum\limits_{0\leq m\leq N_0}\|\pa^m_{v_x}\CG_2\|^2
+\sum\limits_{0\leq m\leq N_0}|\pa^m_{v_x}\CG_2|^2_{2,+}
\leq&C(\eps)\sum\limits_{0\leq m\leq N_0}\|w_q\pa^{m}_{v_x}\CG_1\|^2_{L^\infty}
+C(\eps)\sum\limits_{0\leq m\leq N_0}\|w_q\pa^m_{v_x}\CF_2\|_{L^\infty}^2\notag\\&+C(\eps)\sum\limits_{0\leq m\leq N_0}\|w_q\pa^m_{v_x}\CF_{2,b}\|_{L^\infty}^2.
\end{align}
For this, we begin with the following equations for $\CG_2$
\begin{eqnarray}\label{F2.eq}
\left\{\begin{array}{rll}
\begin{split}&\eps\CG_2+v_y\pa_y\CG_2-\al v_y\pa_{v_x}\CG_2
+\nu_0\CG_2-\si  K\CG_2-\si (1-\chi_{M})\mu^{-\frac{1}{2}}\CK \CG_1=\CF_2,\ y\in(-1,1),\\
&\CG_{2}(\pm1,v){\bf 1}_{\{v_y\lessgtr0\}}-\si \sqrt{2\pi \mu}\dis{\int_{v_y\gtrless0}}\sqrt{\mu}\CG(\pm1,v)|v_y|dv=\CF_{2,b}.
\end{split}
\end{array}\right.
\end{eqnarray}
Taking the inner product of $\eqref{F2.eq}_1$ and $\CG_2$ over $(y,v)\in(-1,1)\times\R^3,$ we have, for $\eta>0$
\begin{multline}\label{g2.l2}
(\eps+(1-\si)\nu_0)\|\CG_2\|^2+\si\de_0\|\FP_1\CG_2\|^2
+\frac{1}{2}|\{I-P_\ga\}\CG_2|^2_{2,+}+\frac{1}{2}(1-\si)|P_\ga\CG_2|^2_{2,+}\\
\leq|(\CG_2,\CF_2)|+|((1-\chi_{M})\mu^{-\frac{1}{2}}\CK \CG_1,\CG_2)|
+\eta|P_\ga\CG_2|^2_{2,+}+C_{\eta}\|w_q\CG_1\|^2_{L^\infty}+C_{\eta}\|w_q\CF_{2,b}\|_{L^\infty}^2,
\end{multline}
where the following estimate on the boundary term has been used
\begin{align}
\int_{\R^3}&v_y\CG_2^2(1)dv-\int_{\R^3}v_y\CG_2^2(-1)dv\notag\\=&\int_{v_y>0}v_y\CG_2^2(1)dv
-\int_{v_y<0}v_y(\si P_\ga\CG_2+\si \bar{P}_{\ga}\CG_1+\CF_{2,b})^2(1)dv
\notag\\&-\int_{v_y<0}v_y\CG_2^2(-1)dv-\int_{v_y>0}v_y(\si P_\ga\CG_2+\si \bar{P}_{\ga}\CG_1+\CF_{2,b})^2(-1)dv\notag\\
\geq& (1-\si^2)\int_{v_y>0}v_y(P_\ga\CG_2)^2(1)dv+\int_{v_y>0}v_y(\{I-P_\ga\}\CG_2)^2(1)dv
\notag\\&+(1-\si^2)\int_{v_y<0}|v_y|(P_\ga\CG_2)^2(-1)dv+\int_{v_y<0}|v_y|(\{I-P_\ga\}\CG_2)^2(-1)dv
\notag\\&-\eta\int_{v_y<0}v_y(P_\ga\CG_2)^2(1)dv
-C_\eta\|w_q\CF_{2,b}(1)\|_{L^\infty}^2-\eta\int_{v_y>0}|v_y|(P_\ga\CG_2)^2(-1)dv
\notag\\
&-C_\eta\|w_q\CF_{2,b}(-1)\|_{L^\infty}^2-C_\eta\int_{v_y<0}|v_y||\bar{P}_{\ga}\CG_1(1)|^2dv
-C_\eta\int_{v_y>0}|v_y||\bar{P}_{\ga}\CG_1(-1)|^2dv\notag\\
\geq& |\{I-P_\ga\}\CG_2|^2_{2,+}+(1-\si)|P_\ga\CG_2|^2_{2,+}-\eta|P_\ga\CG_2|^2_{2,+}-C_{\eta}\|w_q\CF_{2,b}(\pm1)\|_{L^\infty}^2
-C_\eta\|w_q\CG_1\|^2_{L^\infty}.
\notag
\end{align}
Here we have used the notation
$$
\bar{P}_{\ga}\CG_1(\pm1)=\sqrt{2\pi\mu}\int_{v_y\gtrless0}\CG_1(\pm1)|v_y|dv,
$$
and the estimate
\begin{align*}
\left|\int_{v_y\gtrless0}\CG_1(\pm1)|v_y|dv\right|\leq C\|w_q\CG_1\|_{L^\infty},\ \textrm{for}\ q>5/2.
\end{align*}
Next, since
\begin{equation}
|P_{\gamma }\CG_2(\pm1)|_{2,\pm }^{2} =\int_{v_y\gtrless0}\left[ \int_{\{v_y\gtrless0\}}\CG_2(\pm1)\sqrt{\mu }|v_y|dv\right] ^{2}2\pi\mu (v)|v_y| dv=\left[ \int_{\{v_y\gtrless0\}}\CG_2(\pm1)\sqrt{\mu }|v_y|dv\right] ^{2},
\notag
\end{equation}%
by dividing the domain for integration as
\begin{equation}\nonumber
\begin{split}
\{v\in\R^{3}:v_y>0\}& =\underbrace{\{v\in \R^{3}:0<v_y<\varepsilon \ \text{or}\ v_y>1/\varepsilon\}}_{V^\vps}
\cup \{v\in \R^{3}:\varepsilon \leq v_y\leq 1/\varepsilon \},
\end{split}
\end{equation}%
one sees that
the grazing part of
$|P_{\gamma}\CG_2(1)|_{2,+}^{2}$ is bounded by the H\"{o}lder inequality as
\begin{equation}
\begin{split}
 \Bigg(\int_{V^\vps}\mu (v)|v_y|dv\Bigg)&
\int_{v_y>0}|\CG_2(1)|^{2}v_ydv  \lesssim \vps\int_{v_y>0}|\CG_2(1)|^{2}v_ydv.
\end{split}
\label{trace1}
\end{equation}
For non-grazing region, we have by using the trace  Lemma \ref{ukai} that
\begin{align}\label{trace2}
\int_{\{v\in \R^{3}:v_y>0\}\backslash V^\vps}|\CG_2(1)|^{2}v_ydv
\leq& C\|\CG_2\|^{2}+C\| v_y\pa_{y}\CG_2^{2}-\al v_{y}
\pa_{v_x}\CG_2^{2}\|_{L^1}\notag\\
\leq& C\|\CG_2\|^{2}+C|(L\CG_2,\CG_2)|+C|((1-\chi_{M})\mu^{-\frac{1}{2}}\CK \CG_1,,\CG_2)|+|(\CF_2,\CG_2)|\notag\\
\leq& C\|\CG_2\|^{2}+C\| w_q\CG_1\|_{L^\infty}^{2}+C\|\CF_2\|^2.
\end{align}%
Putting \eqref{trace1} and \eqref{trace2} together, one has
\begin{align}\label{Pbd}
|P_{\gamma }\CG_2|_{2,+}^{2}\leq \vps|\{I-P_\ga\}\CG_2|^2_{2,+}
+C\|\CG_2\|^{2}+C\| w_q\CG_1\|_{L^\infty}^{2}+C\|\CF_2\|^2.
\end{align}
Consequently,
\eqref{g2.l2} and \eqref{Pbd} give
\begin{align}\label{g2.zl2}
\|\CG_2\|^2
+|\CG_2|^2_{2,+}
\leq&C(\eps)\{\|w_q\CF_2\|_{L^\infty}^2+\|w_q\CF_{2,b}\|_{L^\infty}^2+\|w_q\CG_1\|^2_{L^\infty}\}.
\end{align}
\begin{remark}
Note that the constant $C(\eps)$ in \eqref{g2.zl2}  is independent of the parameter $\si$.
\end{remark}

It remains to deduce the $L^2$ estimate of the higher order velocity derivatives. For this, applying $\pa^m_{v_x}$ $(m\geq1)$ to $\eqref{F2.eq}$ to have
\begin{eqnarray}\label{dF2.eq}
\left\{\begin{array}{rll}
\begin{split}&\eps\pa^m_{v_x}\CG_2+v_y\pa_y\pa^m_{v_x}\CG_2-\al v_y\pa^{m+1}_{v_x}\CG_2
+\nu_0\pa^m_{v_x}\CG_2-\si  \pa^m_{v_x}K\CG_2\\&\qquad\qquad-\si \pa^m_{v_x}[(1-\chi_{M})\mu^{-\frac{1}{2}}\CK \CG_1]=\pa^m_{v_x}\CF_2,\ y\in(-1,1),\\
&\pa^m_{v_x}\CG_{2}(\pm1,v){\bf 1}_{\{v_y\lessgtr0\}}-\si \sqrt{2\pi}\pa^m_{v_x} (\mu^{\frac{1}{2}})\dis{\int_{v_y\gtrless0}}\sqrt{\mu}\CG(\pm1,v)|v_y|dv=\pa^m_{v_x}\CF_{2,b}.
\end{split}
\end{array}\right.
\end{eqnarray}
Taking the inner product of \eqref{dF2.eq} and $\pa^m_{v_x}\CG_2$, we  deduce
\begin{align}\label{dg2.l2}
(\eps&+(1-\si)\nu_0)\|\pa^m_{v_x}\CG_2\|^2+\si\de_0\|\pa^m_{v_x}\CG_2\|^2
+\frac{1}{2}|\pa^m_{v_x}\CG_2|^2_{2,+}\notag\\
\leq&C\|\CG_2\|^2+C\sum\limits_{m'\leq m}\|w_q\pa^{m'}_{v_x}\CG_1\|^2_{L^\infty}
+C(m)|P_\ga\CG_2|^2_{2,+}+C\|\pa^m_{v_x}\CF_2\|^2+C\|w_q\pa^m_{v_x}\CF_{2,b}\|_{L^\infty}^2,
\end{align}
where we have used the following estimates for the incoming boundary term by $\eqref{dF2.eq}_2$
\begin{align}
\int_{v_y\lessgtr0}|v_y||\pa^m_{v_x}\CG_2(\pm1,v){\bf 1}_{\{v_y\lessgtr0\}}|^2dv\leq C(m)|P_\ga\CG_2|^2_{2,+}
+C\| w_q\CG_1\|_{L^\infty}^{2}
+C\|w_q\pa^m_{v_x}\CF_{2,b}\|_{L^\infty}^2.\notag
\end{align}
Then, \eqref{dg2.l2} and \eqref{g2.zl2} give \eqref{g2.tl2}. With \eqref{g2.tl2}, \eqref{Lif.es1} follows from \eqref{H1sum}, \eqref{H2sum5} and \eqref{g2.tl2}. This completes the proof of the lemma.
\end{proof}

\subsection{Existence for the linear problem with $\si=1$ and $\epsilon>0$}
With Lemma \ref{lifpri}, we now turn to prove the existence of solution to \eqref{pals} for a fixed parameter $\eps>0$ in $L^\infty$ framework by the contraction mapping argument.

\begin{lemma}\label{ex.pals}
Under the same assumption of Lemma \ref{lifpri},
there exists a unique solution $[\CG_1,\CG_2]\in\FX_{\al,N_0}$
to \eqref{pals} with $\si=1$ satisfying
\begin{align}\label{Lif.es2}
\sum\limits_{0\leq m\leq N_0}&\left\{\|w_{q}\pa_{v_x}^{m}\CG_1\|_{L^\infty}+\|w_{q}\pa_{v_x}^{m}\CG_2\|_{L^\infty}\right\}\notag\\
\leq&
C\sum\limits_{0\leq m\leq N_0}\left\{\|w_{q}\pa_{v_x}^{m}\CF_1\|_{L^\infty}+\|w_{q}\pa_{v_x}^{m}\CF_2\|_{L^\infty}
+\|w_q\pa_{v_x}^{m}\CF_{2,b}(s)\|_{L^\infty}\right\}.
\end{align}
\end{lemma}

\begin{proof} The proof is based on the {\it a priori} estimate \eqref{Lif.es1} established in Lemma \ref{lifpri} and a bootstrap argument.
As for Lemma \ref{ex-G1-rf}, the proof  is also divided into the following three steps.

\noindent\underline{{\it Step 1. Existence for $\si=0$.}}
If $\si=0$, then \eqref{pals} is reduced to
\begin{align*}
\eps\CG_1+v_y\pa_y\CG_1-\al v_y\pa_{v_x}\CG_1
+\nu_0\CG_1+\al\frac{v_xv_y}{2}\sqrt{\mu}\CG_2=\CF_1,\ y\in(-1,1),
\end{align*}
\begin{align*}
&\CG_{1}(\pm1,v){\bf 1}_{\{v_y\lessgtr0\}}=0,
\end{align*}
and
\begin{align*}
\eps\CG_2+v_y\pa_y\CG_2-\al v_y\pa_{v_x}\CG_2
+\nu_0\CG_2=\CF_2,\ y\in(-1,1),
\end{align*}
\begin{align*}
\CG_{2}(\pm1,v){\bf 1}_{\{v_y\lessgtr0\}}=\CF_{2,b},
\end{align*}
respectively. Then, in this simple case, the existence of $L^\infty$-solutions can be directly proved by the characteristic method 
to have
\begin{align}\label{L0}
\|\mathscr{L}_0^{-1}[\tilde{\CF}_1,\tilde{\CF}_2]\|_{\FX_{\al,N_0}}\leq C_\mathscr{L}\|[\tilde{\CF}_1,\tilde{\CF}_2]\|_{\FX_{\al,N_0}}.
\end{align}

\noindent\underline{{\it Step 2. Existence for $\si\in[0,\si_\ast]$ for some $\si_\ast>0$.}}
Let $\si\in(0,1]$, we now consider
\begin{align}\label{CG1st}
\eps\CG_1+v_y\pa_y\CG_1-\al v_y\pa_{v_x}\CG_1
+\nu_0\CG_1+\al\frac{v_xv_y}{2}\sqrt{\mu}\CG_2=\si \chi_{M}\CK \CG_{1}+\CF_1,\ y\in(-1,1),
\end{align}
\begin{align}\label{CG1stbd}
&\CG_{1}(\pm1,v){\bf 1}_{\{v_y\lessgtr0\}}=0,
\end{align}
and
\begin{align}\label{CG2st}
\eps\CG_2+v_y\pa_y\CG_2-\al v_y\pa_{v_x}\CG_2
+\nu_0\CG_2=\si  K\CG_2+\si (1-\chi_{M})\mu^{-\frac{1}{2}}\CK\CG_1+\CF_2,\ y\in(-1,1),
\end{align}
\begin{align}\label{CG2stbd}
\CG_{2}(\pm1,v){\bf 1}_{\{v_y\lessgtr0\}}=\si \sqrt{2\pi \mu}\dis{\int_{v_y\gtrless0}}(\CG_1+\sqrt{\mu}\CG_2)(\pm1,v)|v_y|dv+\CF_{2,b}.
\end{align}
For the above system, we design the following approximation scheme
\begin{align}\label{CG1stn}
\eps\CG^{n+1}_1+v_y\pa_y\CG^{n+1}_1-\al v_y\pa_{v_x}\CG^{n+1}_1
+\nu_0\CG^{n+1}_1+\al\frac{v_xv_y}{2}\sqrt{\mu}\CG^{n+1}_2=\si \chi_{M}\CK \CG^{n}_{1}+\CF_{1}
:=\CF_1^{(1)},
\end{align}
\begin{align}\label{CG1bdstn}
\CG^{n+1}_{1}(\pm1,v){\bf 1}_{\{v_y\lessgtr0\}}=0,
\end{align}
and
\begin{align}\label{CG2stn}
\eps\CG^{n+1}_2+v_y\pa_y\CG^{n+1}_2-\al v_y\pa_{v_x}\CG^{n+1}_2
+\nu_0\CG^{n+1}_2=\si  K\CG^{n}_2+\si (1-\chi_{M})\mu^{-\frac{1}{2}}\CK\CG^{n}_1+\CF_2:=\CF^{(1)}_2,
\end{align}
\begin{align}\label{CG2bdstn}
\CG^{n+1}_{2}(\pm1,v){\bf 1}_{\{v_y\lessgtr0\}}=\si \sqrt{2\pi \mu}\dis{\int_{v_y\gtrless0}}(\CG^{n}_1+\sqrt{\mu}\CG^{n}_2)(\pm1,v)|v_y|dv+\CF_{2,b}:=\CF^{(1)}_{2,b},
\end{align}
with $[\CG^{0}_1,\CG^{0}_{2}]=[0,0]$. The goal in the following proof has twofold: (i) $[\CG_1^n,\CG_2^n]_{n=0}^\infty$ is uniformly bounded in $\FX_{\al,N_0}$, and (ii)
$[\CG_1^n,\CG_2^n]_{n=0}^\infty$ is a Cauchy sequence in $\FX_{\al,N_0}$. By \eqref{L0}, it follows
\begin{align}\label{siast.re}
\|[\CG_1^{n+1},\CG_2^{n+1}]\|_{\FX_{\al,N_0}}\leq& C_\SL \{\|\CF_1^{(1)},\CF_2^{(1)}]\|_{\FX_{\al,N_0}}+\sum\limits_{0\leq m\leq N_0}\|w_q\pa_{v_x}^m\CF_{2,b}\|_{L^\infty}\}\notag\\
\leq& C_\SL \si \bar{C}_1\|[\CG_1^{n},\CG_2^{n}]\|_{\FX_{\al,N_0}}\\&
+\underbrace{C_\mathscr{L}\sum\limits_{0\leq m\leq N_0}\left\{\|w_{q}\pa_{v_x}^m\CF_1\|_{L^\infty}
+\|w_{q}\pa_{v_x}^m\CF_2\|_{L^\infty}+\|w_q\pa_{v_x}^m\CF_{2,b}\|_{L^\infty}\right\}}_{\CM_0},\notag
\end{align}
where $\bar{C}_1>0$ is independent of $\si$ and $n$. Choosing $0<\si_\ast<1$ suitably small such that
\begin{align}\label{siast}
C_\SL \si_\ast \bar{C}_1\leq\frac{1}{2},
\end{align}
\eqref{siast.re} implies that
\begin{align}\label{CGumbd}
\|[\CG_1^{n},\CG_2^{n}]\|_{\FX_{\al,N_0}}\leq 2\CM_0,
\end{align}
for all $n\geq0$. Moreover, by \eqref{CG1stn}, \eqref{CG1bdstn}, \eqref{CG2stn} and \eqref{CG2bdstn} and applying \eqref{L0},  one has \begin{align}\label{CGcase}
\|[\CG_1^{n+1},\CG_2^{n+1}]-[\CG_1^{n},\CG_2^{n}]\|_{\FX_{\al,N_0}}\leq& C_\SL \si \bar{C}_1\|[\CG_1^{n},\CG_2^{n}]-[\CG_1^{n-1},\CG_2^{n-1}]\|_{\FX_{\al,N_0}}\notag\\
\leq&\frac{1}{2}\|[\CG_1^{n},\CG_2^{n}]-[\CG_1^{n-1},\CG_2^{n-1}]\|_{\FX_{\al,N_0}}
\end{align}
with the condition \eqref{siast}. Consequently, \eqref{CGcase} and \eqref{CGumbd} imply that the systems \eqref{CG1st}-\eqref{CG1stbd} and \eqref{CG2st}-\eqref{CG2stbd} has a unique solution $[\CG_1,\CG_2]\in\FX_{\al,N_0}$ for any $\si\in[0,\si_\ast].$
Moreover, by Lemma \ref{lifpri}, we have the following uniform estimate
\begin{align}
\|[\CG_1,\CG_2]\|_{\FX_{\al,N_0}}\leq
C_\mathscr{L}\sum\limits_{0\leq m\leq N_0}\left\{\|w_{q}\pa_{v_x}^m\CF_1\|_{L^\infty}
+\|w_{q}\pa_{v_x}^m\CF_2\|_{L^\infty}+\|w_q\pa_{v_x}^m\CF_{2,b}\|_{L^\infty}\right\},\notag
\end{align}
which is  equivalent to
\begin{align}\label{Last}
\|\mathscr{L}_{\si_\ast}^{-1}[\tilde{\CF}_1,\tilde{\CF}_2]\|_{\FX_{\al,N_0}}\leq C_\mathscr{L}\|[\tilde{\CF}_1,\tilde{\CF}_2]\|_{\FX_{\al,N_0}}.
\end{align}

\noindent\underline{{\it Step 3. Existence for $\si\in[0,2\si_\ast]$ for some $\si_\ast>0$.}} By using \eqref{Last} and performing the similar calculations as for obtaining \eqref{CGumbd} and \eqref{CGcase}, for $\si\in[0,\si_\ast],$ one can see that there exists a unique solution $[\CG_1,\CG_2]\in\FX_{\al,N_0}$ to the lifted system
\begin{align}\notag
\eps\CG_1+v_y\pa_y\CG_1-\al v_y\pa_{v_x}\CG_1
+\nu_0\CG_1+\al\frac{v_xv_y}{2}\sqrt{\mu}\CG_2-\si_\ast \chi_{M}\CK \CG_{1}=\si \chi_{M}\CK \CG_{1}+\CF_1,\ y\in(-1,1),
\end{align}
\begin{align}\notag
&\CG_{1}(\pm1,v){\bf 1}_{\{v_y\lessgtr0\}}=0,
\end{align}
and
\begin{align}
\eps\CG_2&+v_y\pa_y\CG_2-\al v_y\pa_{v_x}\CG_2
+\nu_0\CG_2-\si_\ast  K\CG_2-\si_\ast (1-\chi_{M})\mu^{-\frac{1}{2}}\CK\CG_1\notag\\&=\si  K\CG_2+\si (1-\chi_{M})\mu^{-\frac{1}{2}}\CK\CG_1+\CF_2,\ y\in(-1,1),\notag
\end{align}
\begin{align}
\CG_{2}(\pm1,v){\bf 1}_{\{v_y\lessgtr0\}}&-\si_\ast \sqrt{2\pi \mu}\dis{\int_{v_y\gtrless0}}(\CG_1+\sqrt{\mu}\CG_2)(\pm1,v)|v_y|dv
\notag\\&=\si \sqrt{2\pi \mu}\dis{\int_{v_y\gtrless0}}(\CG_1+\sqrt{\mu}\CG_2)(\pm1,v)|v_y|dv+\CF_{2,b}.\notag
\end{align}
In other words, we have shown the existence of $\mathscr{L}^{-1}_{2\si_\ast}$ on $\FX_{\al,N_0}$ and \eqref{Lif.es1} holds true for $\si=2\si_\ast$.

Therefore, by repeating this procedure in finite time, ,
one can see that $\mathscr{L}^{-1}_1$ exists when $\si=1$  and \eqref{Lif.es2}  follows correspondingly. This completes the proof of the lemma.
\end{proof}

\subsection{Estimates on  remainder}\label{sub5.4}

We are ready to complete the proof of Proposition \ref{Gr.lem}.

\begin{proof}[Proof of Proposition \ref{Gr.lem}] We now prove the existence of the coupled system \eqref{Gr1} and \eqref{Gr2} under the diffuse boundary conditions \eqref{Gr1bd} and  \eqref{Gr2bd}, respectively.

Let us first go back to the approximation system \eqref{Gr1n}, \eqref{Gr1nbd}, \eqref{Gr2n} and \eqref{Gr2nbd}.
By applying Lemma \ref{ex.pals}, for fixed $\eps>0$, we see that $[G_{R,1}^{n+1},G_{R,2}^{n+1}]$ is well defined when $[G_{R,1}^{n},G_{R,2}^{n}]$ is given and the solution
 belongs to $\FX_{\al,N_0}$ defined in \eqref{def.sX} for $N_0\geq0$.

We now show that $\{G_{R,1}^{n},G_{R,2}^{n}\}_{n=0}^\infty$ is a Cauchy sequence in $\FX_{\al,N_0}$, which hence implies that
its limit denoted by $[G^{\eps}_{R,1},G^{\eps}_{R,2}]$ is the unique solution of the following system
\begin{align}
\eps G^{\eps}_{R,1}&+v_y\pa_yG^{\eps}_{R,1}-\al v_y\pa_{v_x}G^{\eps}_{R,1}+\nu_0 G^{\eps}_{R,1}-\chi_{M}\CK G^{\eps}_{R,1}\notag\\=&-\frac{\al}{2}\sqrt{\mu}v_xv_yG^{\eps}_{R,2}
-\frac{1}{2}\sqrt{\mu}v_xv_yG_{1}+\sqrt{\mu}v_y\pa_{v_x}G_{1}
+Q(\sqrt{\mu}G_1,\sqrt{\mu}G_1)\notag\\&+\al\{Q(\sqrt{\mu}G^{\eps}_R,\sqrt{\mu}G_1)+Q(\sqrt{\mu}G_1,\sqrt{\mu}G^{\eps}_R)\}
+\al^2Q(\sqrt{\mu}G^{\eps}_R,\sqrt{\mu}G^{\eps}_R)\notag\\
:=&\CN_\eps,
\ y\in(-1,1),\ v\in\R^3,\label{lmeGr1}
\end{align}
\begin{align}\label{lmeGr1bd}
G^{\eps}_{R,1}(\pm1,v)|_{v_y\lessgtr0}=0,\ v\in\R^3,
\end{align}
and
\begin{align}\label{lmeGr2}
\eps G^{\eps}_{R,1}&+v_y\pa_yG^{\eps}_{R,2}-\al v_y\pa_{v_x}G^{\eps}_{R,2}+LG^{\eps}_{R,2}=(1-\chi_{M})\mu^{-\frac{1}{2}}\CK G^{\eps}_{R,1},\ y\in(-1,1),\ v\in\R^3,
\end{align}
\begin{align}\label{lmeGr2bd}
G^{\eps}_{R,2}(\pm1,v)|_{v_y\lessgtr0}=\sqrt{2\pi \mu}\dis{\int_{v_y\gtrless0}}\sqrt{\mu}G^{\eps}_{R}(\pm1,v)|v_y|dv,\ v\in\R^3.
\end{align}
Furthermore, we will show that the convergence of the sequence $\{G_{R,1}^{n},G_{R,2}^{n}\}_{n=0}^\infty$ is independent of $\eps$.
For this, we first prove that
\begin{align}\label{umbd}
\|[G^{n}_{R,1},G^{n}_{R,2}]\|_{\FX_{\al,N_0}}\leq 2\CC_0,
\end{align}
where $\CC_0>0$ is independent of $\eps$ and $n$ for all $n\geq0$.  We apply induction in $n$. Notice $[G^{0}_{R,1},G^{0}_{R,2}]=[0,0]$. If $n=1$, then the system \eqref{Gr1n}, \eqref{Gr1nbd}, \eqref{Gr2n} and \eqref{Gr2nbd} reads
\begin{align}\label{Gr10}
\eps G^{1}_{R,1}&+v_y\pa_yG^{1}_{R,1}-\al v_y\pa_{v_x}G^{1}_{R,1}+\nu_0 G^{1}_{R,1}-\chi_{M}\CK G^{1}_{R,1}+\frac{\al}{2}\sqrt{\mu}v_xv_yG^{1}_{R,2}
\notag\\&=-\frac{1}{2}\sqrt{\mu}v_xv_yG_{1}+\sqrt{\mu}v_y\pa_{v_x}G_{1}+Q(G_1,G_1):=\CS^0,
\ y\in(-1,1),\ v\in\R^3,
\end{align}
\begin{align}\label{Gr10bd}
G^{1}_{R,1}(\pm1,v)|_{v_y\lessgtr0}=0,\ v\in\R^3,
\end{align}
and
\begin{align}\label{Gr20}
\eps &G^{1}_{R,2}+v_y\pa_yG^{1}_{R,2}-\al v_y\pa_{v_x}G^{1}_{R,2}+LG^{1}_{R,2}=(1-\chi_{M})\mu^{-\frac{1}{2}}\CK G^{1}_{R,1},\ y\in(-1,1),\ v\in\R^3,
\end{align}
\begin{align}\label{Gr20bd}
G^{1}_{R,2}(\pm1,v)|_{v_y\lessgtr0}=\sqrt{2\pi \mu}\dis{\int_{v_y\gtrless0}}\sqrt{\mu}G^{1}_{R}(\pm1,v)|v_y|dv,\ v\in\R^3.
\end{align}
Performing similar calculations as for deriving \eqref{H1sum} and \eqref{H2sum5}, one has
\begin{align}\label{GR1}
\sum\limits_{0\leq m\leq N_0}\|w_q\pa_{v_x}^mG_{R,1}^1\|_{L^\infty}\leq& C\al\sum\limits_{0\leq m\leq N_0}\|w_q\pa_{v_x}^mG_{R,2}^1\|_{L^\infty}
+C\sum\limits_{0\leq m\leq N_0}\|w_{q}\pa_{v_x}^m\CS^0\|_{L^\infty}\notag\\
\leq& C\al\sum\limits_{0\leq m\leq N_0}\|w_q\pa_{v_x}^mG_{R,2}^1\|_{L^\infty}
+C,
\end{align}
and
\begin{align}\label{GR2}
\sum\limits_{0\leq m\leq N_0}\|w_q\pa_{v_x}^mG_{R,2}^1\|_{L^\infty}\leq& C\sum\limits_{0\leq m\leq N_0}\|\pa_{v_x}^mG^1_{R,2}\|
+C\sum\limits_{0\leq m\leq N_0}\|w_q\pa_{v_x}^mG_{R,1}^{k+1}\|_{L^\infty},
\end{align}
where the constant $C>0$ is independent of $\eps$, see also Remark \ref{ideps}.

Since the mass of $G_{R,2}^1$ is not conserved, in order to estimate the macroscopic component of $G_{R,2}^1$ we instead turn to obtain the $L^2$ estimate of $G_{R}^1$. Recall $\sqrt{\mu}G_{R}^1=G_{R,1}^1+\sqrt{\mu}G_{R,2}^1$. By \eqref{Gr10}, \eqref{Gr10bd}, \eqref{Gr20} and \eqref{Gr20bd}, it is easy to see that $G_{R}^1$ satisfies
\begin{align}\label{Grn1}
\eps G^1_R&+v_y\pa_yG^1_R-\al v_y\pa_{v_x}G^1_R+\frac{\al}{2}v_xv_yG^1_{R}+LG^1_R=\mu^{-\frac{1}{2}}\CS^0,
\end{align}
and
\begin{align}
G^1_R(\pm1,v)|_{v_y\lessgtr0}=\sqrt{2\pi\mu}\dis{\int_{v_y\gtrless0}}\sqrt{\mu}G^1_R(\pm1,v)|v_y|dv.\notag
\end{align}
Next, for $n\geq1$, denote
\begin{align}\label{abc.def}
\FP_0G^n_R=(a^n+\Fb^n\cdot v+c^n(|v|^2-3))\sqrt{\mu},\ \FP_0G^n_{R,2}=(a^n_{2}+\Fb^n_{2}\cdot v+c^n_{2}(|v|^2-3))\sqrt{\mu},
\end{align}
and define the projection $\bar{\FP}_0$, from $L^2$ to $\ker \CL$, as
\begin{align}\label{abc1.def}
\bar{\FP}_0G^n_{R,1}=(a^n_{1}+\Fb^n_{1}\cdot v+c^n_{1}(|v|^2-3))\mu.
\end{align}
In addition, we will also use the following notations
\begin{align}
\Fb^n_{i}=[b^n_{i,1},b^n_{i,2},b^n_{i,3}],\ i=1,2, \ \Fb^n=[b^n_{1},b^n_{2},b^n_{3}].\notag
\end{align}
Note that
\begin{align}\label{m-m}
a^n=a^n_1+a^n_2,\ \Fb^n=\Fb^n_1+\Fb^n_2,\ c^n=c^n_1+c^n_2,\ \int_{-1}^1a^n(y)dy=0.
\end{align}
Since
\begin{align}\label{m1-if}
\|[a^n_1,\Fb^n_1,c^n_1]\|\leq C\|\bar{\FP}_0G^n_{R,1}\|\leq C\|w_qG^n_{R,1}\|_{L^\infty},
\end{align}
for $q>5/2$, to obtain the estimate of $\|[a^1_2,\Fb^1_2,c^1_2]\|$, it suffices to derive the $L^2$ estimates of $[a^1,\Fb^1,c^1]$.
In what follows, we will show that the $L^2$ norm of the macroscopic part of $G^1_R$ can be indeed dominated by its microscopic component and other known terms.
We estimate $[a^1,\Fb^1,c^1]$ by the dual argument. First of all, we let $\Psi(y,v)\in C^\infty([-1,1]\times\R^3)$, and take the inner product of \eqref{Grn1} and $\Psi$
over $(-1,1)\times\R^3$, to obtain
\begin{align}\label{CGtt}
\eps(G^1_{R},\Psi)&-(v_yG^1_{R},\pa_y\Psi)+\lag v_yG^1_{R}(1),\Psi(1)\rag-\lag v_yG^1_{R}(-1),\Psi(-1)\rag
+\al (v_yG^1_{R},\pa_{v_x}\Psi)\notag\\&+\frac{\al}{2}(v_xv_yG^1_{R},\Psi)+ (LG^1_{R},\Psi)
=(\mu^{-\frac{1}{2}}\CS^0,\Psi).
\end{align}



\noindent\underline{Estimate on $a^1$.} Let
$$\Psi=\Psi_{a^1}=v_y\frac{d}{dy}\phi_{a^1}(y) (|v|^2-10)\sqrt{\mu},$$
where
\begin{align}\label{a1eq}
\phi''_{a^1}={a^1},\ \phi'_{a^1}(\pm1)=0.
\end{align}
Thus
\begin{align}\label{epa}
\|\phi_{a^1}\|_{H^2}\leq C\|{a^1}\|.
\end{align}
Plugging $\Psi=\Psi_{a^1}$ into \eqref{CGtt}, we now compute the equation term by term. First of all, by Cauchy-Schwarz inequality with $\eta>0$ and using \eqref{epa}, one has
\begin{align}
\eps|(G^1_{R},\Psi_{a^1})|\leq&\eps|(\FP_0G^1_{R},\Psi)|+\eps|(\FP_1G^1_{R},\Psi)|
\notag\\ \leq& C\eps\|a^1\|^2+C\eps\{\|\FP_1G^1_{R,2}\|^2+\|w_qG^1_{R,1}\|^2_{L^\infty}\}+C\eps\|b^1_{2}\|^2,\notag
\end{align}
\begin{align}
-(v_yG^1_{R},\pa_y\Psi_{a^1})=&-(v_y\FP_0G^1_{R},\pa_y\Psi_{a^1})-(v_y\FP_1G^1_{R},\pa_y\Psi_{a^1})\notag\\
\geq&5\|a^1\|^2-\eta\|a^1\|^2-C_\eta\{\|\FP_1G^1_{R,2}\|^2+\|w_qG^1_{R,1}\|^2_{L^\infty}\},\notag
\end{align}
\begin{align}
\al |(v_yG^1_{R},\pa_{v_x}\Psi_{a^1})|+\frac{\al}{2}|(v_xv_yG^1_{R},\Psi_{a^1})|\leq C\al(\|b_1^1\|^2+\|w_qG^1_{R,1}\|_{L^\infty}^2+\|G^1_{R,2}\|^2).\notag
\end{align}
Then by Lemmas \ref{G1.lem} and \ref{Ga}, it follows
\begin{align}
(\mu^{-\frac{1}{2}}\CS^0,\Psi_{a^1})=&|( v_y\pa_{v_x}G_1-\frac{1}{2}v_xv_yG_{1}
+\Ga(G_1,G_1),\Psi_{a^1})|\notag\\
\leq& \eta\|a^1\|^2+C_{\eta}\|G_1\|^2+C_\eta\int_{\R^3}\int_{-1}^1|\Ga(G_1,G_1)|^2dvdy
\notag\\
\leq& \eta\|a^1\|^2+C_\eta\|w_qG_1\|_{L^\infty}^2+C_\eta\int_{-1}^1\|G_1\|^4dy
\notag\\
\leq& \eta\|a^1\|^2+C_\eta\|w_qG_1\|_{L^\infty}^2+C_\eta\|w_qG_1\|_{L^\infty}^2\|G_1\|^2
\leq \eta\|a^1\|^2+C_\eta,\label{gal2}
\end{align}
provided that $q>3/2$.

Next, noting that $LG^1_{R}=-\{\Ga(G^1_{R},\sqrt{\mu})+\Ga(\sqrt{\mu},G^1_{R})\}$, one has by a similar argument as above that
\begin{align}\label{Ll2}
|(LG^1_{R},\Psi_{a^1})|\leq& |(L\left(G^1_{R,1}\mu^{-\frac{1}{2}}\right),\Psi_{a^1})|+|(LG^1_{R,2},\Psi_{a^1})|\notag\\
\leq& \eta\|a^1\|^2+C_\eta(\|w_qG^1_{R,1}\|_{L^\infty}^2+\|\FP_1G^1_{R,2}\|^2).
\end{align}
The last  boundary term $\lag v_yG_R^1(1),\Psi_{a^1}(1)\rag-\lag v_yG_R^1(-1),\Psi_{a^1}(-1)\rag$ vanishes because of
 \eqref{a1eq}.

Putting all the estimates above together, we have
\begin{align}\label{a.es}
\|a^1\|^2 \leq& C\|\FP_1G^1_{R,2}\|^2+C\|w_qG^1_{R,1}\|^2_{L^\infty}+C\al\{\|G^1_{R,2}\|^2+\|b_1^1\|^2\}+C\|b^1_2\|^2+C.
\end{align}

\noindent\underline{Estimate on $\Fb^1$.}
Let
\begin{align*}
\Psi=\Psi_{b^1_i}=\left\{\begin{array}{rll}
&v_yv_{x}\frac{d}{dy}\phi_{b^1_1}(y)\sqrt{\mu},\ i=1,\\[2mm]
&v_yv_{z}\frac{d}{dy}\phi_{b^1_3}(y)\sqrt{\mu},\ i=3,\\[2mm]
&v_y^2(|v|^2-5)\frac{d}{dy}\phi_{b^1_2}(y)\sqrt{\mu},\ i=2,
\end{array}
\right.
\end{align*}
where
\begin{align*}
-\phi''_{b^1_i}=b^1_i,\ \phi_{b_i}(\pm1)=0.
\end{align*}
Then it holds
\begin{align}\label{epb}
\|\phi_{b^1_i}\|_{H^2}\leq C\|b^1_i\|, \ |\phi'_{b^1_i}(\pm1)|\leq C\|b^1_i\|.
\end{align}
We now compute each  term in \eqref{CGtt} with $\Psi=\Psi_{b^1_i}$.
By Cauchy-Schwarz inequality and  \eqref{epb}, one has
\begin{align}
\eps|(G_R^1,\Psi_{b^1_i})|\leq&\eps|(\FP_0G_R^1,\Psi_{b^1_i})|+\eps|(\FP_1G_R^1,\Psi_{b^1_i})|
\notag\\ \leq& C\eps\|b^1_i\|^2+C\eps\{\|\FP_1G_{R,2}^1\|^2+\|w_qG_{R,1}^1\|^2_{L^\infty}\}+C\eps{\bf 1}_{i=2}\|c^1\|^2,\notag
\end{align}
\begin{align}
-(v_yG_R^1,\pa_y\Psi_{b^1_i})=&-(v_y\FP_0G_R^1,\pa_y\Psi_{b^1_i})-(v_y\FP_1G_R^1,\pa_y\Psi_{b^1_i})\notag\\
\geq&\left\{\begin{array}{rll}
&\|b^1_i\|^2-\eta\|b^1_i\|^2-C_\eta\{\|\FP_1G^1_{R,2}\|^2+\|w_qG^1_{R,1}\|^2_{L^\infty}\},\ i=1,3,\\[2mm]
&6\|b^1_i\|^2-\eta\|b^1_i\|^2-C_\eta\{\|\FP_1G^1_{R,2}\|^2+\|w_qG^1_{R,1}\|^2_{L^\infty}\},\ i=2,
\end{array}\right.\notag
\end{align}
\begin{align}
\al |(v_yG_R^1,\pa_{v_x}\Psi_{b^1_i})|+\frac{\al}{2}|(v_xv_yG_R^1,\Psi_{b^1_i})|\leq C\al(\|[a^1,c^1]\|^2+\|w_qG^1_{R,1}\|_{L^\infty}^2+\|G^1_{R,2}\|^2).\notag
\end{align}
Similar to  \eqref{gal2} and \eqref{Ll2}, it follows
\begin{align}
(\mu^{-\frac{1}{2}}\CS^0,\Psi_{b_i^1})\leq \eta\|b^1_i\|^2+C,\notag
\end{align}
and
\begin{align}
|(LG_R^1,\Psi_{b_i})|\leq& |(L\left(G_{R,1}^1\mu^{-\frac{1}{2}}\right),\Psi_{b^1_i})|+|(LG_{R,2},\Psi_{b^1_i})|\notag\\
\leq& \eta\|b_i\|^2+C_\eta(\|w_qG_{R,1}\|_{L^\infty}^2+\|\FP_1G_{R,2}\|^2).\notag
\end{align}
For the boundary term, noting that
\begin{align}\label{bddec}
G_R^1(\pm1)|_{v_y\neq0}=P_{\ga}G_R^1(\pm1)+\{I-P_{\ga}\}G_R^1(\pm1)|_{v_y\gtrless0},
\end{align}
we have
\begin{align}
\lag v_y&G_R^1(1),\Psi_{b_i}(1)\rag-\lag v_y\CG(-1),\Psi_{b_i}(-1)\rag\notag\\=& \lag v_yP_\ga G_R^1(1),\Psi_{b_i}(1)\rag
+\lag v_y\{I-P_\ga\}G_R^1(1)|_{v_y>0},\Psi_{b_i}(1)\rag\notag\\&-\lag v_yP_\ga G_R^1(-1),\Psi_{b_i}(-1)\rag
-\lag v_y\{I-P_\ga\}G_R^1(-1)|_{v_y<0},\Psi_{b_i}(-1)\rag\notag\\
=&\lag v_y\{I-P_\ga\}G_R^1(1)|_{v_y>0},\Psi_{b_i}(1)\rag-\lag v_y\{I-P_\ga\}G_R^1(-1)|_{v_y<0},\Psi_{b_i}(-1)\rag\notag
\\ \leq &\eta\|b^1_i\|^2+C_\eta|\{I-P_\ga\}G_{R,2}^1|^2_{2,+}+C_\eta\|w_qG_{R,1}^1\|^2_{L^\infty},\notag
\end{align}
where the fact that $\lag v_yP_\ga G_R^1(\pm1),\Psi_{b_i}(\pm1)\rag=0$ has been used.

We now conclude from the above estimates for $b^1_i$ with $1\leq i\leq3$ that
\begin{align}\label{b.es}
\|\Fb^1\|^2 \leq& C\|\FP_1G^1_{R,2}\|^2+C\|w_qG^1_{R,1}\|^2_{L^\infty}+C\al\{\|G^1_{R,2}\|^2+\|[a^1,c^1]\|^2\}\notag\\&
+C|\{I-P_\ga\}G^1_{R,2}|^2_{2,+}+C\|c^1\|^2+C.
\end{align}

\noindent\underline{Estimate on $c^1$.}
Let
\begin{align*}
\Psi=\Psi_{c^1}=v_y(|v|^2-5)\frac{d}{dy}\phi_{c^1}(y)\sqrt{\mu},
\end{align*}
where
$$
-\phi_{c^1}''=c^1,\ \phi_{c^1}(\pm1)=0.
$$
One has
\begin{align}\label{epc}
\|\phi_{c^1}\|_{H^2}\leq C\|c^1\|, \ |\phi'_{c^1}(\pm1)|\leq C\|c^1\|.
\end{align}
By Cauchy-Schwarz inequality and \eqref{epc}, it follows
\begin{align}
\eps|(G_{R}^1,\Psi_{c^1})|\leq&\eps|(\FP_0G_{R}^1,\Psi_{c^1})|
+\eps|(\FP_1G_{R}^1,\Psi_{c^1})|
\notag\\ \leq& C\eps\|c\|^2+C\eps\{\|\FP_1G_{R,2}^1\|^2+\|w_qG_{R,1}^1\|^2_{L^\infty}\},\notag
\end{align}
\begin{align}
-(v_yG_{R}^1,\pa_y\Psi_{c^1})=&-(v_y\FP_0G_{R}^1,\pa_y\Psi_{c^1})-(v_y\FP_1G_{R}^1,\pa_y\Psi_{c^1})\notag\\
\geq&30\|c^1\|^2-\eta\|c^1\|^2-C_\eta\{\|\FP_1G_{R,2}^1\|^2+\|w_qG_{R,1}^1\|^2_{L^\infty}\},\notag
\end{align}
\begin{align}
\al |(v_yG_{R}^1,\pa_{v_x}\Psi_{c^1})|+\frac{\al}{2}|(v_xv_yG_{R}^1,\Psi_{c^1})|
\leq C\al(\|b_1^1\|^2+\|w_qG_{R,1}^1\|_{L^\infty}^2+\|G_{R,2}^1\|^2).\notag
\end{align}
Also similar to  derive \ref{gal2} and \eqref{Ll2}, one has
\begin{align}
(\mu^{-\frac{1}{2}}\CS^0,\Psi_{c^1})\leq \eta\|c^1\|^2+C,\notag
\end{align}
and
\begin{align}
|(LG_R^1,\Psi_{c^1})|\leq& |(L\left(G_{R,1}^1\mu^{-\frac{1}{2}}\right),\Psi_{c^1})|+|(LG_{R,2},\Psi_{c^1})|\notag\\
\leq& \eta\|c^1\|^2+C_\eta(\|w_qG^1_{R,1}\|_{L^\infty}^2+\|\FP_1G^1_{R,2}\|^2).\notag
\end{align}
For the boundary term, by applying \eqref{bddec} and using
$$
\lag v_yP_\ga G_{R}^1(\pm1),\Psi_{c^1}(\pm1)\rag=0,
$$
we have
\begin{align}
\lag v_y&G_{R}^1(1),\Psi_{c^1}(1)\rag-\lag v_yG_{R}^1(-1),\Psi_{c^1}(-1)\rag\leq\eta\|c^1\|^2+C_\eta|\{I-P_\ga\}G_{R,2}^1|^2_{2,+}+C_\eta\|w_qG_{R,1}^1\|^2_{L^\infty}.\notag
\end{align}
Combining the above estimates on $c^1$ give
\begin{align}\label{c.es}
\|c^1\|^2 \leq& C\|\FP_1G_{R,2}^1\|^2+C\|w_qG_{R,1}^1\|^2_{L^\infty}+C\al\{\|G_{R,2}^1\|^2+\|b_1^1\|^2\}+C\|w_qG_{R,1}^1\|^2_{L^\infty}
\notag\\&+C|\{I-P_\ga\}G^1_{R,2}|^2_{2,+}+C\al^2.
\end{align}
Finally, a linear combination of \eqref{a.es}, \eqref{b.es} and \eqref{c.es} gives
\begin{align}\label{abc.es}
\|[a^1,\Fb^1,c^1]\|^2 \leq& C\|\FP_1G^1_{R,2}\|^2+C\|w_qG^1_{R,1}\|^2_{L^\infty}+C\al\|G^1_{R,2}\|^2\notag\\&+C|\{I-P_\ga\}G^1_{R,2}|^2_{2,+}+C.
\end{align}
This together with \eqref{m-m} and \eqref{m1-if}  implies that $[a^1_2,\Fb^1_2,c^1_2]$ satisfies
\begin{align}\label{abc2.es}
\|[a_2^1,\Fb_2^1,c_2^1]\|^2 \leq& C\|\FP_1G^1_{R,2}\|^2+C\|w_qG^1_{R,1}\|^2_{L^\infty}+C|\{I-P_\ga\}G^1_{R,2}|^2_{2,+}+C.
\end{align}
In order to obtain the estimates for $\|\FP_1G^1_{R,2}\|$, we have to further consider the BVP for $G^1_{R,2}$ as follows:
\begin{eqnarray}
\left\{\begin{array}{rll}
\begin{split}&\eps G^1_{R,2}+v_y\pa_yG^1_{R,2}-\al v_y\pa_{v_x}G^1_{R,2}
+LG^1_{R,2}-(1-\chi_{M})\mu^{-\frac{1}{2}}\CK G^1_{R,1}=0,\\
&G^1_{R,2}(\pm1,v){\bf 1}_{\{v_y\lessgtr0\}}=\sqrt{2\pi \mu}\dis{\int_{v_y\gtrless0}}\sqrt{\mu}G^1_{R}(\pm1,v)|v_y|dv.
\end{split}
\end{array}\right.\notag
\end{eqnarray}
Applying the estimates \eqref{g2.l2}, \eqref{Pbd}  and \eqref{dg2.l2} with $\si=1$, $\CF_2=0$ and $\CF_{2,b}=0$, one has
\begin{align}\label{mi-G1r2}
\eps\|G^1_{R,2}\|^2&+\de_0\|\FP_1G^1_{R,2}\|^2
+\frac{1}{2}|\{I-P_\ga\}G^1_{R,2}|^2_{2,+}
\leq\eta|P_\ga G^1_{R,2}|^2_{2,+}+\eta\|G^1_{R,2}\|^2+C_\eta\|w_qG^1_{R,1}\|^2_{L^\infty},
\end{align}
\begin{align}\label{trG1r2}
|P_{\gamma }G^1_{R,2}|_{2,+}^{2}\leq \vps|\{I-P_\ga\}G^1_{R,2}|^2_{2,+}
+C\|G^1_{R,2}\|^{2}+C\| w_qG^1_{R,1}\|_{L^\infty}^{2},
\end{align}
and
\begin{align}\label{hmi-G1r2}
\eps\|\pa^m_{v_x}G^1_{R,2}\|^2&+\de_0\|\pa^m_{v_x}G^1_{R,2}\|^2
+\frac{1}{2}|\pa^m_{v_x}G^1_{R,2}|^2_{2,+}\notag\\
\leq&C\|G^1_{R,2}\|^2+C\sum\limits_{m'\leq m}\|w_q\pa^{m'}_{v_x}G^1_{R,1}\|^2_{L^\infty}
+C(m)|P_\ga G^1_{R,2}|^2_{2,+},
\end{align}
where all the constants on the right hand side are independent of $\eps.$
Then \eqref{abc.es}, \eqref{mi-G1r2}, \eqref{trG1r2} and \eqref{hmi-G1r2} give
\begin{align}\label{ttG1r2}
\sum\limits_{0\leq m\leq N_0}\|\pa^m_{v_x}G^1_{R,2}\|^2
&+\sum\limits_{0\leq m\leq N_0}|\pa^m_{v_x}G^1_{R,2}|^2_{2,+}
\leq C\sum\limits_{0\leq m\leq N_0}\|w_q\pa^{m}_{v_x}G^1_{R,1}\|^2_{L^\infty}+C.
\end{align}
Consequently, a linear combination of \eqref{GR1}, \eqref{GR2} and \eqref{ttG1r2} gives
\begin{align}
\sum\limits_{0\leq m\leq N_0}\{\|w_q\pa^{m}_{v_x}G^1_{R,1}\|_{L^\infty}+\|w_q\pa^{m}_{v_x}G^1_{R,1}\|_{L^\infty}\}\leq \CC_0,\notag
\end{align}
for some suitably large $\CC_0>0$. Therefore \eqref{umbd} holds for $n=1$.

We now assume that \eqref{umbd} is valid for $n=k\geq 1$ and then prove that it  holds for $n=k+1$.
In fact, applying the estimates \eqref{H1sum} and \eqref{H2sum5} to the system \eqref{Gr1n}-\eqref{Gr1nbd}
and \eqref{Gr2n}-\eqref{Gr2nbd} with $n=k$, one has
\begin{align}\label{GRk1}
\sum\limits_{0\leq m\leq N_0}\|w_q\pa_{v_x}^mG_{R,1}^{k+1}\|_{L^\infty}\leq C\al\sum\limits_{0\leq m\leq N_0}\|w_q\pa_{v_x}^mG_{R,2}^{k+1}\|_{L^\infty}
+C\sum\limits_{0\leq m\leq N_0}\|w_{q}\pa_{v_x}^m\CS^k\|_{L^\infty},
\end{align}
and
\begin{align}\label{GRk2}
\sum\limits_{0\leq m\leq N_0}\|w_q\pa_{v_x}^mG_{R,2}^{k+1}\|_{L^\infty}\leq& C\sum\limits_{0\leq m\leq N_0}\|w_q\pa_{v_x}^mG_{R,2}^{k+1}\|
+C\sum\limits_{0\leq m\leq N_0}\|w_q\pa_{v_x}^mG_{R,1}^{k+1}\|_{L^\infty},
\end{align}
where
\begin{align}
\CS^k=&-\frac{1}{2}\sqrt{\mu}v_xv_yG_{1}+\sqrt{\mu}v_y\pa_{v_x}G_{1}
+Q(\sqrt{\mu}G_1,\sqrt{\mu}G_1)\notag\\&+\al\{Q(\sqrt{\mu}G^{k}_R,\sqrt{\mu}G_1)+Q(\sqrt{\mu}G_1,\sqrt{\mu}G^{k}_R)\}
+\al^2Q(\sqrt{\mu}G^{k}_R,\sqrt{\mu}G^{k}_R).\notag
\end{align}
By Lemma \ref{G1.lem}, Lemma \ref{Ga} and the induction assumption, we have
\begin{align}\label{Sk}
\sum\limits_{0\leq m\leq N_0}\|w_{q}\pa_{v_x}^m\CS^k\|_{L^\infty}\leq& C+C\al\sum\limits_{0\leq m\leq N_0}\{\|w_q\pa_{v_x}^mG^{k}_{R,1}\|_{L^\infty}+\|w_q\pa_{v_x}^mG^{k}_{R,1}\|_{L^\infty}\}
\notag\\&+C\al^2\sum\limits_{0\leq m\leq N_0}\{\|w_q\pa_{v_x}^mG^{k}_{R,1}\|^2_{L^\infty}+\|w_q\pa_{v_x}^mG^{k}_{R,1}\|^2_{L^\infty}\}.
\end{align}

For the $L^2$ estimate, by performing a parallel calculation as for  \eqref{abc2.es}, one has
\begin{align}\label{abck.es}
\|[a_1^{k+1},\Fb_1^{k+1},c_1^{k+1}]\|^2 \leq& C\|\FP_1G^{k+1}_{R,2}\|^2+C\|w_qG^{k+1}_{R,1}\|^2_{L^\infty}+C|\{I-P_\ga\}G^{k+1}_{R,2}|^2_{2,+}
\notag\\&+C\sum\limits_{j=1}^3|(\mu^{-\frac{1}{2}}\CS^k,\Psi_j)|.
\end{align}
Here $\Psi_j$ $(1\leq j\leq 3)$ are chosen as  $\Psi_{a^{k+1}}$, $\Psi_{b_i^{k+1}}$ and $\Psi_{c^{k+1}}$  in the same way as for
$\Psi_{a^{1}}$, $\Psi_{b_i^{1}}$ and $\Psi_{c^{1}}$, respectively. Hence, \eqref{abck.es} also gives
\begin{align}\label{abc1k.es}
\|[a_2^{k+1},\Fb_2^{k+1},c_2^{k+1}]\|^2 \leq& C\|\FP_1G^{k+1}_{R,2}\|^2+C\|w_qG^{k+1}_{R,1}\|^2_{L^\infty}+C|\{I-P_\ga\}G^{k+1}_{R,2}|^2_{2,+}
+C\notag\\&+C\al^2\{\|w_qG^{k}_{R,1}\|^2_{L^\infty}+\|w_qG^{k}_{R,2}\|^2_{L^\infty}\}
\notag\\&+C\al^4\{\|w_qG^{k}_{R,1}\|^4_{L^\infty}+\|w_qG^{k}_{R,2}\|^4_{L^\infty}\},
\end{align}
by applying Lemma \ref{Ga} and the relation \eqref{m-m}.

On the other hand, similar to the estimates \eqref{mi-G1r2}, \eqref{trG1r2} and \eqref{hmi-G1r2}, it also follows
\begin{align}\label{mi-Gkr2}
\eps\|G^{k+1}_{R,2}\|^2&+\de_0\|\FP_1G^{k+1}_{R,2}\|^2
+\frac{1}{2}|\{I-P_\ga\}G^{k+1}_{R,2}|^2_{2,+}\notag\\
\leq&\eta|P_\ga G^{k+1}_{R,2}|^2_{2,+}+\eta\|G^{k+1}_{R,2}\|^2+C_\eta\|w_qG^{k+1}_{R,1}\|^2_{L^\infty},
\end{align}
\begin{align}\label{trGkr2}
|P_{\gamma }G^{k+1}_{R,2}|_{2,+}^{2}\leq \vps|\{I-P_\ga\}G^{k+1}_{R,2}|^2_{2,+}
+C\|G^{k+1}_{R,2}\|^{2}+C\| w_qG^{k+1}_{R,1}\|_{L^\infty}^{2},
\end{align}
and
\begin{align}\label{hmi-Gkr2}
\eps\|\pa^m_{v_x}G^{k+1}_{R,2}\|^2&+\de_0\|\pa^m_{v_x}G^{k+1}_{R,2}\|^2
+\frac{1}{2}|\pa^m_{v_x}G^{k+1}_{R,2}|^2_{2,+}\notag\\
\leq&C\|G^{k+1}_{R,2}\|^2+C\sum\limits_{m'\leq m}\|w_q\pa^{m'}_{v_x}G^{k+1}_{R,1}\|^2_{L^\infty}
+C|P_\ga G^{k+1}_{R,2}|^2_{2,+}.
\end{align}
As a consequence,  combining estimates \eqref{abc1k.es}, \eqref{mi-Gkr2}, \eqref{trGkr2} and \eqref{hmi-Gkr2}  gives
\begin{align}\label{ttGkr2}
\sum\limits_{0\leq m\leq N_0}&\|\pa^m_{v_x}G^{k+1}_{R,2}\|^2
+\sum\limits_{0\leq m\leq N_0}|\pa^m_{v_x}G^{k+1}_{R,2}|^2_{2,+}\notag\\
\leq& C\sum\limits_{0\leq m\leq N_0}\|w_q\pa^{m}_{v_x}G^{k+1}_{R,1}\|^2_{L^\infty}
+C\al^2\{\|w_qG^{k}_{R,1}\|^2_{L^\infty}+\|w_qG^{k}_{R,1}\|^2_{L^\infty}\}
\notag\\&+C\al^4\{\|w_qG^{k}_{R,1}\|^4_{L^\infty}+\|w_qG^{k}_{R,2}\|^4_{L^\infty}\}.
\end{align}
Finally, by taking $C_1>0$ suitably large, we have from \eqref{GRk1}, \eqref{GRk2}, \eqref{Sk} and \eqref{ttGkr2} that
\begin{align}
\|[G^{k+1}_{R,1},G^{k+1}_{R,2}]\|_{\FX_{\al,N_0}}
\leq& \CC_0+C_1\al\|[G^{k}_{R,1},G^{k}_{R,2}]\|_{\FX_{\al,N_0}}+C_1\al^2\|[G^{k}_{R,1},G^{k}_{R,2}]\|_{\FX_{\al,N_0}}^2
\notag\\ \leq& \CC_0\{1+2\CC_0C_1\al+4C_1\CC^2_0\al^2\}\leq\frac{5}{4}\CC_0,\notag
\end{align}
provided that $\al$ is  chosen to be sufficiently small.
Thus \eqref{umbd} holds for $n=k+1$. Therefore,  \eqref{umbd} holds for all $n\geq0$.

We now turn to prove that $[G_{R,1}^n,G_{R,1}^n]_{n=0}^{\infty}$ is a Cauchy sequence in $\FX_{\al,N_0}$.
For this, denote
$$\mu^{1/2}\tilde{G}_R^{n+1}=\tilde{G}^{n+1}_{R,1}+\mu^{1/2}\tilde{G}^{n+1}_{R,2},$$
with
$$
[\tilde{G}^{n+1}_{R,1},\tilde{G}^{n+1}_{R,2}]=[G_{R,1}^{n+1}-G_{R,1}^{n},G_{R,2}^{n+1}-G_{R,2}^{n}].
$$
Then $[\tilde{G}^{n+1}_{R,1},\tilde{G}^{n+1}_{R,2}]$ satisfies
\begin{align}
\eps \tilde{G}^{n+1}_{R,1}&+v_y\pa_y\tilde{G}^{n+1}_{R,1}-\al v_y\pa_{v_x}\tilde{G}^{n+1}_{R,1}+\nu_0 \tilde{G}^{n+1}_{R,1}
-\chi_{M}\CK \tilde{G}^{n+1}_{R,1}+\frac{\al}{2}\sqrt{\mu}v_xv_y\tilde{G}^{n+1}_{R,2}
\notag\\=&\al\{Q(\sqrt{\mu}\tilde{G}^{n}_R,\sqrt{\mu}G_1)+Q(\sqrt{\mu}G_1,\sqrt{\mu}\tilde{G}^{n}_R)\}
\notag\\&+\al^2\{Q(\sqrt{\mu}\tilde{G}^{n}_R,\sqrt{\mu}\tilde{G}^{n}_R)+Q(\sqrt{\mu}\tilde{G}^{n}_R,\sqrt{\mu}G^{n-1}_R)
+Q(\sqrt{\mu}G^{n-1}_R,\sqrt{\mu}\tilde{G}^{n}_R)\}
\notag\\:=&\CN,
\ y\in(-1,1),\ v\in\R^3,\notag
\end{align}
\begin{align}
\tilde{G}^{n+1}_{R,1}(\pm1,v)|_{v_y\lessgtr0}=0,\ v\in\R^3,\notag
\end{align}
and
\begin{align}
\eps &\tilde{G}^{n+1}_{R,2}+v_y\pa_y\tilde{G}^{n+1}_{R,2}-\al v_y\pa_{v_x}\tilde{G}^{n+1}_{R,2}+L\tilde{G}^{n+1}_{R,2}-(1-\chi_{M})\mu^{-\frac{1}{2}}\CK \tilde{G}^{n+1}_{R,1}=0,\ y\in(-1,1),\ v\in\R^3,\notag
\end{align}
\begin{align}
\tilde{G}^{n+1}_{R,2}(\pm1,v)|_{v_y\lessgtr0}=\sqrt{2\pi \mu}
\dis{\int_{v_y\gtrless0}}\sqrt{\mu}\tilde{G}^{n+1}_{R}(\pm1,v)|v_y|dv,\ v\in\R^3.\notag
\end{align}
We  claim that
\begin{align}\label{cau}
\|[\tilde{G}^{n+1}_{R,1},\tilde{G}^{n+1}_{R,2}]\|_{\FX_{\al,N_0}}\leq C_m\al\|[\tilde{G}^{n}_{R,1},\tilde{G}^{n}_{R,2}]\|_{\FX_{\al,N_0}},
\end{align}
under the condition \eqref{umbd}. In fact, on the one hand, by performing a similar calculation as for obtaining \eqref{GRk1}, \eqref{GRk2} and \eqref{ttGkr2}, one has
\begin{align*}
\|[\tilde{G}^{n+1}_{R,1},\tilde{G}^{n+1}_{R,2}]\|_{\FX_{\al,N_0}}\leq C\sum\limits_{0\leq m\leq N_0}\|w_{q}\pa_{v_x}^m\CN(\tilde{G}_R^{n},\tilde{G}_R^{n})\|_{L^{\infty}}.
\end{align*}
On the other hand, we have from Lemma \ref{Ga} that
\begin{align}
\sum\limits_{0\leq m\leq N_0}\|w_{q}\pa_{v_x}^m\CN(\tilde{G}_R^{n},\tilde{G}_R^{n})\|_{L^{\infty}}\leq&
C\al\sum\limits_{0\leq m\leq N_0}\left\{\|w_{q}\pa_{v_x}^m\tilde{G}^{n}_{R,1}\|_{L^\infty}
+\|w_{q}\pa_{v_x}^m\tilde{G}^{n}_{R,2}\|_{L^\infty}\right\}\notag
\\&+C\al^2\sum\limits_{0\leq m\leq N_0}\|w_{q}\pa_{v_x}^m\tilde{G}_R^{n}\|_{L^\infty}
\sum\limits_{0\leq m\leq N_0}\|w_{q}\pa_{v_x}^mG^n_R\|_{L^\infty},\notag
\end{align}
which is further bounded by
\begin{align}
C\al\sum\limits_{0\leq m\leq N_0}\left\{\|w_{q}\pa_{v_x}^m\tilde{G}^{n}_{R,1}\|_{L^\infty}
+\|w_{q}\pa_{v_x}^m\tilde{G}^{n}_{R,2}\|_{L^\infty}\right\},\notag
\end{align}
according to \eqref{umbd}. Thus, the  claim \eqref{cau} holds. In other words, let $\al>0$ be suitably small, then $\{[G^{n}_{R,1},G^{n}_{R,2}]\}_{n=0}^{\infty}$
is a Cauchy sequence in $\FX_{\al,N_0}$. Hence,
$$
[G^{n}_{R,1},G^{n}_{R,2}]\rightarrow[G_{R,1}^\eps,G_{R,2}^\eps]
$$
strongly in $\FX_{\al,N_0}$ as $n\rightarrow+\infty$.
Moreover,  the convergence is uniform with respect to $\eps$,
and the limit $[G^{\eps}_{R,1},G^{\eps}_{R,2}]$ is a unique solution to \eqref{lmeGr1}-\eqref{lmeGr1bd} and
 \eqref{lmeGr2}-\eqref{lmeGr2bd}.
In addition, $[G^{\eps}_{R,1},G^{\eps}_{R,2}]$ also satisfies
\begin{align}\label{Gebd}
\|[G^\eps_{R,1},G^\eps_{R,2}]\|_{\FX_{\al,N_0}}\leq C,
\end{align}
where $C>0$ is independent of $\eps.$

Furthermore, by taking the limit $\eps\to 0$, we can repeat the same procedure like  letting $n\to \infty$ so that the limit function
$[G_{R,1},G_{R,2}]\in\FX_{\al,N_0}$ is the unique solution of \eqref{Gr1}-\eqref{Gr1bd} and \eqref{Gr2}-\eqref{Gr2bd}
with the same bound as \eqref{Gebd}. Then the proof of Proposition \ref{Gr.lem} is completed.
\end{proof}

Finally, Theorem \ref{st.sol} is an immediate consequence of Proposition \ref{G1.lem} and Proposition \ref{Gr.lem}, except for the non-negativity of the solution $F_{st}(y,v)$ that will be  proved from the dynamical stability of $F_{st}(y,v)$  in Theorem \ref{ust.mth}.\qed

\section{Unsteady problem: local existence}\label{loc.ex}

We now turn to the time-dependent situation.
To solve the initial boundary value problem \eqref{F}, we  set the perturbation as
\begin{equation}
\label{f.usp}
F(t,y,v)=F_{st}(y,v)+\sqrt{\mu}f(t,y,v),
\end{equation}
then $f=f(t,y,v)$ satisfies
\begin{equation}\label{f}
\left\{\begin{split}
&\pa_tf+v_y\pa_{y}f-\al v_y\pa_{v_x}f+\frac{\al}{2}v_xv_yf+Lf\\
&\quad=\Ga(f,f)+\al\{\Ga(G_1+\al G_R,f)+\Ga(f,G_1+\al G_R)\},\\
&\qquad\qquad\qquad t>0,\ y\in(-1,1),\ v=(v_x,v_y,v_z)\in\R^3,\\
&\sqrt{\mu}f(0,y,v)\eqdef f_0(y,v)=F(0,y,v)-F_{st}(y,v),\ y\in(-1,1),\ v\in\R^3,\\ 
&f(t,\pm1,v)|_{v_y\lessgtr0}=\sqrt{2\pi \mu}\dis{\int_{v_y\gtrless0}}f(t,\pm1,v)\sqrt{\mu}|v_y|dv,\ t\geq0,\ v\in\R^3.
\end{split}\right.
\end{equation}
The goal of this section is to construct the local-in-time solution to the initial boundary value problem \eqref{f}. The proof of the global existence of solutions as well as the large time behavior will be left to the next section. To resolve the difficulty caused by the growth term $\frac{\al}{2}v_xv_yf$, it is still necessary to introduce the decomposition
\begin{equation}
\label{def.tpf12}
\sqrt{\mu}f=f_1+\sqrt{\mu}f_2,
\end{equation}
where $f_1$ and $f_2$ satisfy the following initial boundary value problem
\begin{align}\label{f1}
\pa_tf_1&+v_y\pa_yf_1-\al v_y\pa_{v_x}f_1+\nu_0 f_1\notag\\=&\chi_{M}\CK f_1-\frac{\al}{2}\sqrt{\mu}v_xv_yf_2
+\al\{Q(\sqrt{\mu}f,\sqrt{\mu}\{G_1+\al G_R\})+Q(\sqrt{\mu}\{G_1+\al G_R\},\sqrt{\mu}f)\}\notag\\&+Q(\sqrt{\mu}f,\sqrt{\mu}f),
\end{align}
\begin{align}\label{f1bd}
f_1(0,y,v)=f_{0}(y,v)=F_0-F_{st},\
\ f_1(\pm1,v)|_{v_y\lessgtr0}=\sqrt{2\pi}\mu\dis{\int_{v_y\gtrless0}}f_1(\pm1,v)|v_y|dv,
\end{align}
and
\begin{align}\label{f2}
\pa_tf_2&+v_y\pa_yf_2-\al v_y\pa_{v_x}f_2+Lf_2=(1-\chi_{M})\mu^{-\frac{1}{2}}\CK f_1,
\end{align}
\begin{align}\label{f2bd}
f_2(0,y,v)=0,\ \ f_2(\pm1,v)|_{v_y\lessgtr0}=\sqrt{2\pi\mu}\dis{\int_{v_y\gtrless0}}\sqrt{\mu}f_2(\pm1,v)|v_y|dv,
\end{align}
respectively. Note that initial data for $f_2$ is set to be zero.

We will look for solutions to \eqref{f1}-\eqref{f1bd} and \eqref{f2}-\eqref{f2bd} in the following function space
\begin{equation*}
\begin{split}
\FY_{\al,T}=&\Big\{(\CG_1,\CG_2)\bigg|\sup\limits_{0\leq t\leq T}\left\{\|w_{q}\CG_1(t)\|_{L^\infty}+\|w_{q}\CG_2(t)\|_{L^\infty}\right\}<+\infty
\Big\},
\end{split}
\end{equation*}
associated with the norm
$$
\|[\CG_1,\CG_2]\|_{\FY_{\al,T}}=\sup\limits_{0\leq t\leq T}\left\{\|w_{q}\CG_1(t)\|_{L^\infty}+\|w_{q}\CG_2(t)\|_{L^\infty}\right\}.
$$
\begin{theorem}[Local existence]\label{loc.th}
Under the conditions  in Theorem \ref{ust.mth}, there exits  $T_\ast>0$ depending on $\al$ such that the coupled system \eqref{f1}-\eqref{f1bd} and \eqref{f2}-\eqref{f2bd} admits a unique local in time solution $[f_1(t,y,v),f_2(t,y,v)]$ satisfying
\begin{align*}
\|[f_1,f_2]\|_{\FY_{\al,T_\ast}}\leq C_0\vps_0,
\end{align*}
for some $C_0>0.$
\end{theorem}
\begin{proof}
We first consider the following system for approximation solutions
\begin{align}\label{af1}
\pa_tf^{n+1}_1&+v_y\pa_yf^{n+1}_1-\al v_y\pa_{v_x}f^{n+1}_1+\nu_0 f^{n+1}_1\notag\\=&\chi_{M}\CK f^{n}_1-\frac{\al}{2}\sqrt{\mu}v_xv_yf^{n}_2
+H(f_1^n,f_2^n),
\end{align}
\begin{align}\label{af1bd}
f^{n+1}_1(0,y,v)=f_{0}(y,v)=F_0-F_{st},\ \ f^{n+1}_1(\pm1,v)|_{v_y\lessgtr0}=\sqrt{2\pi}\mu\dis{\int_{v_y\gtrless0}}f^n_1(\pm1,v)|v_y|dv,
\end{align}
and
\begin{align}\label{af2}
\pa_tf^{n+1}_2&+v_y\pa_yf^{n+1}_2-\al v_y\pa_{v_x}f^{n+1}_2+\nu_0f^{n+1}_2=Kf_2^{n}+(1-\chi_{M})\mu^{-\frac{1}{2}}\CK f^{n}_1,
\end{align}
\begin{align}\label{af2bd}
f^{n+1}_2(0,y,v)=0,\ \
f^{n+1}_2(\pm1,v)|_{v_y\lessgtr0}=\sqrt{2\pi \mu}\dis{\int_{v_y\gtrless0}}\sqrt{\mu}f_2^{n}(\pm1,v)|v_y|dv,
\end{align}
where 
$$
H(f^{n}_1,f^{n}_2)=\al\{Q(\sqrt{\mu}f^{n},\sqrt{\mu}\{G_1+\al G_R\})+Q(\sqrt{\mu}\{G_1+\al G_R\},\sqrt{\mu}f^{n})\}
+Q(\sqrt{\mu}f^{n},\sqrt{\mu}f^{n}),
$$
and
$\sqrt{\mu}f^{n}=f_1^{n}+\sqrt{\mu}f^{n}_2$. Set  $[f_1^0,f_2^0]=[f_0,0].$

Next, one can show inductively that there exists a finite $T_{\ast}>0$ such that
\begin{align}
\sup\limits_{0\leq t\leq T_{\ast}}\|w_q[f_1^{m},f_2^{m}](t)\|_{L^\infty}\leq C_0\vps_0,\label{fn1.bd}
\end{align}
for any $m\geq0$, provided that
\begin{align}\notag
\|w_q[f_1^0,f_2^{0}]\|_{L^\infty}=\left\|w_q\left[F_0(y,v)-F_{st}(y,v)\right]\right\|_{L^\infty}\leq \vps_0.
\end{align}
This also implies that $[f_1^{n+1},f_2^{n+1}]$
is well-defined by  \eqref{af1}-\eqref{af1bd} and  \eqref{af2}-\eqref{af2bd} if $[f_1^{n},f_2^{n}]$ is bounded as in \eqref{fn1.bd}.
Denote
$$
[\mathfrak{G}^{n}_1,\RG^{n}_2]=w_q[f^{n}_1,f^{n}_2],\ \sqrt{\mu}\RG^n=\RG_1^n+\sqrt{\mu}\RG_2^n,
$$
then $\RG^n_1$ and $\RG^n_2$ satisfy
\begin{align}\label{ag1}
\pa_t\RG^{n+1}_1&+v_y\pa_y\RG^{n+1}_1-\al v_y\pa_{v_x}\RG^{n+1}_1+2q\al\frac{v_xv_y}{1+|v|^2}\RG^{n+1}_1+\nu_0 \RG^{n+1}_1
\notag\\=&\chi_{M}w_q\CK \left(\frac{\RG^{n}_1}{w_q}\right)-\frac{\al}{2}\sqrt{\mu}v_xv_y\RG^{n}_2
+w_qH(f_1^{n},f_2^{n}),
\end{align}
\begin{align}\label{ag1bd}
\RG^{n+1}_1(0,y,v)=w_qf_0(y,v),\ \ \RG^{n+1}_1(\pm1,v)|_{v_y\lessgtr0}
=\tilde{w}^{-1}_1\dis{\int_{v_y\gtrless0}}\tilde{w}_1\sqrt{2\pi}\mu\RG_1^{n}(\pm1,v)|v_y|dv,
\end{align}
and
\begin{align}\label{ag2}
\pa_t\RG^{n+1}_2&+v_y\pa_y\RG^{n+1}_2-\al v_y\pa_{v_x}\RG^{n+1}_2+2q\al\frac{v_xv_y}{1+|v|^2}\RG^{n+1}_2+\nu_0\RG^{n+1}_2\notag\\&=w_qK\left(\frac{\RG_2^{n}}{w_q}\right)
+(1-\chi_{M})w_q\mu^{-\frac{1}{2}}\CK f^{n}_1,
\end{align}
\begin{align}\label{ag2bd}
\RG^{n+1}_2(0,y,v)=0,\ \ \RG^{n+1}_2(\pm1,v)|_{v_y\lessgtr0}
=\tilde{w}^{-1}_2\dis{\int_{v_y\gtrless0}}\tilde{w}_2\sqrt{2\pi}\mu\RG_2^{n}(\pm1,v)|v_y|dv,
\end{align}
with $[\RG^0_1,\RG_2^0]=w_q[f_1^0,f_2^0]=w_q[f_0,0]$. Here,
$$
\tilde{w}_1=\tilde{w}_1(v)=(\sqrt{2\pi}w_q \mu)^{-1},
$$
and $\tilde{w}_2$ is given by \eqref{wt-2}.

Along the same characteristic line \eqref{chl} by noting that $s$ is no longer a parameter and it is non-negative,
 \eqref{ag1}-\eqref{ag1bd} and  \eqref{ag2}-\eqref{ag2bd} are equivalent to
\begin{align}
\RG_1^{n+1}(t,y,v)
=&{\bf 1}_{t_1\leq0}e^{-\int_{0}^t\CA(\tau,V(\tau))\,d\tau}(w_{q}f_{0})(Y(0),V(0))\notag\\
&+\underbrace{{\bf 1}_{t_1>0}e^{-\int_{t_1}^t\CA(\tau,V(\tau))\,d\tau}
\tilde{w}_1^{-1}(V(t_1))\dis{\int_{n(y_1)\cdot v_1>0}}\tilde{w}_1\sqrt{2\pi}\mu\RG_1^{n}(t_1,y_1,v_1)|v_{1y}|dv_1}_{\CI^{(1)}_b}\notag\\
&+\int_{\max\{0,t_1\}}^{t}e^{-\int_{s}^t\CA(\tau,V(\tau))\,d\tau}\left\{\chi_{M}w_q\CK
\left(\frac{\RG^{n}_1}{w_{q}}\right)\right\}(V(s))\,ds\notag\\
&-\al\int_{\max\{0,t_1\}}^{t}e^{-\int_{s}^t\CA(\tau,V(\tau))\,d\tau}
\frac{V_x(s)V_y(s)}{2}\sqrt{\mu}(V(s))\RG^{n}_2(V(s))\,ds\notag\\
&+\int_{\max\{0,t_1\}}^{t}e^{-\int_{s}^t\CA(\tau,V(\tau))\,d\tau}\left(w_{q}H(f_1^n,f_2^n)\right)(V(s))\,ds,\label{gn1.st}
\end{align}
and
\begin{align}
\RG_2^{n+1}(t,y,v)
=&\underbrace{{\bf 1}_{t_1>0}e^{-\int_{t_1}^t\CA(\tau,V(\tau))\,d\tau}
\tilde{w}_2^{-1}(V(t_1))\dis{\int_{n(y_1)\cdot v_1>0}}\tilde{w}_2\sqrt{2\pi}\mu\RG_2^{n}(t_1,y_1,v_1)|v_{1y}|dv_1}_{\CI^{(2)}_b}\notag\\
&+\int_{\max\{0,t_1\}}^{t}e^{-\int_{s}^t\CA(\tau,V(\tau))\,d\tau}\left\{(1-\chi_{M})\mu^{-\frac{1}{2}}w_{q}\CK
\left(\frac{\RG^n_1}{w_{q}}\right)\right\}(V(s))\,ds\notag\\
&+\int_{\max\{0,t_1\}}^{t}e^{-\int_{s}^t\CA(\tau,V(\tau))\,d\tau}\left[w_{q}K\left(\frac{\RG^n_2}{w_{q}}\right)\right](V(s))\,ds,\label{gn2.st}
\end{align}
where 
\begin{align}
\CA(\tau,V(\tau))=\nu_0+2q \al \frac{V_y(\tau)V_x(\tau)}{{1+|V(\tau)|^2}}\geq\nu_0/2,\notag
\end{align}
provided that $q\alpha$ is suitably small.

For the boundary terms $\CI^{{(1)}}_b$ and $\CI^{{(2)}}_b$, we use the equations \eqref{gn1.st} and \eqref{gn2.st} recursively to obtain
$$
\CI^{(1)}_b=\sum\limits_{j=1}^5\CI^{(1)}_j,\ \CI^{(2)}_b=\sum\limits_{j=1}^3\CI^{(2)}_j
$$
with
\begin{align}
\CI^{(1)}_1=&\underbrace{{\bf 1}_{t_1>0}e^{-\int_{t_1}^t\CA(\tau,V(\tau))d\tau}
\tilde{w}^{-1}_1(V(t_1))}_{\CW^{(1)}_0}\int_{\prod\limits_{j=1}^{k-1}\CV_j}{\bf 1}_{t_k>0}
\RG_1^{n+1-k}(t_k,y_k,V_{\mathbf{cl}}^{k-1}(t_k))\,d\bar{\Sigma}^{(1)}_{k-1}(t_k),\notag
\end{align}
\begin{align}
\CI^{(1)}_2=&\CW^{(1)}_0\sum\limits_{l=1}^{k-1}\int_{\prod\limits_{j=1}^{k-1}\CV_j}{\bf 1}_{t_{l+1}\leq 0<t_l}
(w_{q}f_{0})(Y^{l}_{\mathbf{cl}}(0),V^{l}_{\mathbf{cl}}(0))\,d\bar{\Sigma}^{(1)}_{l}(0),\notag
\end{align}
\begin{align}
\CI^{(1)}_3=&\CW^{(1)}_0\sum\limits_{l=1}^{k-1}\int_{\prod\limits_{j=1}^{k-1}\CV_j}\left\{\mathbf{1}_{\{t_{l+1}\leq 0<t_{l}\}}
\int_{0}^{t_l}+\mathbf{1}_{\{t_{l+1}>0\}}\int_{t_{l+1}}^{t_l}\right\}\left\{\chi_{M}w_q\CK
\left(\frac{\RG^{n-l}_1}{w_{q}}\right)\right\}(Y^{l}_{\mathbf{cl}},V^{l}_{\mathbf{cl}})(s)\,
d\bar{\Sigma}^{(1)}_{l}(s)ds,\notag
\end{align}
\begin{align}
\CI^{(1)}_4=&-\al\CW^{(1)}_0\sum\limits_{l=1}^{k-1}\int_{\prod\limits_{j=1}^{k-1}\CV_j}
\left\{\mathbf{1}_{\{t_{l+1}\leq 0<t_{l}\}}
\int_{0}^{t_l}+\mathbf{1}_{\{t_{l+1}>0\}}\int_{t_{l+1}}^{t_l}\right\}
\left\{\frac{v_xv_y}{2}\sqrt{\mu}\RG^{n-l}_2\right\}
(Y^{l}_{\mathbf{cl}},V^{l}_{\mathbf{cl}})(s)\,d\bar{\Sigma}^{(1)}_{l}(s)ds,\notag
\end{align}
\begin{align}
\CI^{(1)}_5=&\CW^{(1)}_0\sum\limits_{l=1}^{k-1}\int_{\prod\limits_{j=1}^{k-1}\CV_j}
\left\{\mathbf{1}_{\{t_{l+1}\leq 0<t_{l}\}}
\int_{0}^{t_l}+\mathbf{1}_{\{t_{l+1}>0\}}\int_{t_{l+1}}^{t_l}\right\}
\left(w_{q}H(f_1^{n-l},f_2^{n-l})\right)
(Y^{l}_{\mathbf{cl}},V^{l}_{\mathbf{cl}})(s)\,d\bar{\Sigma}^{(1)}_{l}(s)ds,\notag
\end{align}

\begin{align}
\CI^{(2)}_1=&\underbrace{{\bf 1}_{t_1>0}e^{-\int_{t_1}^t\CA(\tau,V(\tau))d\tau}
\tilde{w}_2^{-1}(V(t_1))}_{\CW^{(2)}_0}\int_{\prod\limits_{j=1}^{k-1}\CV_j}{\bf 1}_{t_k>0}
\RG_2^{n+1-k}(t_k,y_k,V_{\mathbf{cl}}^{k-1}(t_k))\,d\bar{\Sigma}^{(2)}_{k-1}(t_k),\notag
\end{align}
\begin{align}
\CI^{(2)}_2=&\CW^{(2)}_0\sum\limits_{l=1}^{k-1}\int_{\prod\limits_{j=1}^{k-1}\CV_j}\left\{\mathbf{1}_{\{t_{l+1}\leq 0<t_{l}\}}
\int_{0}^{t_l}+\mathbf{1}_{\{t_{l+1}>0\}}\int_{t_{l+1}}^{t_l}\right\}
\left\{w_qK\left(\frac{\RG^{n-l}_{2}}{w_{q}}\right)\right\}(Y^{l}_{\mathbf{cl}},V^{l}_{\mathbf{cl}})(s)\,
d\bar{\Sigma}^{(2)}_{l}(s)ds,\notag
\end{align}
and
\begin{align}
\CI^{(2)}_3=&\CW^{(2)}_0\sum\limits_{l=1}^{k-1}\int_{\prod\limits_{j=1}^{k-1}\CV_j}
\left\{\mathbf{1}_{\{t_{l+1}\leq 0<t_{l}\}}
\int_{0}^{t_l}+\mathbf{1}_{\{t_{l+1}>0\}}\int_{t_{l+1}}^{t_l}\right\}\left\{(1-\chi_{M})w_{q}\mu^{-\frac{1}{2}}\CK
\left(\frac{\RG_1^{n-l}}{w_{q}}\right)\right\}(Y^{l}_{\mathbf{cl}},V^{l}_{\mathbf{cl}})(s)\,d\bar{\Sigma}^{(2)}_{l}(s)ds,\notag
\end{align}
where $k\geq2$. Here, similar to \eqref{Sigma}, $\bar{\Sigma}_{l}^{(i)}(s)$ $(i=1,2)$ is given as
\begin{align}
\bar{\Sigma}^{(i)}_{l}(s)=\prod\limits_{j=l+1}^{k-1}d\si_j e^{-\int_s^{t_l}
\CA(\tau,V_{\mathbf{cl}}^l(\tau))d\tau}\tilde{w}_i(v_l)d\si_l \prod\limits_{j=1}^{l-1}\frac{\tilde{w}_i(v_j)}{\tilde{w}_i(V^{j}_{\mathbf{cl}}(t_{j+1}))}
\prod\limits_{j=1}^{l-1}e^{-\int_{t_{j+1}}^{t_j}
\CA(\tau,V_{\mathbf{cl}}^l(\tau))d\tau}d\si_j.\notag
\end{align}

To obtain \eqref{fn1.bd}, one can first prove  that for fixed finite $k>0$ and any $t\geq0$,
\begin{align}
\sup\limits_{0\leq l\leq k}\sup\limits_{0\leq s\leq t}\|[\RG_1^{l},\RG_2^{l}](s)\|_{L^\infty}\leq C(k)\|w_qf_0\|_{L^\infty}\leq \frac{1}{2}C_0\vps_0,
\label{fs-ktm}
\end{align}
by choosing $C_0>0$ suitably large. Note that \eqref{fs-ktm} can be easily obtained by using \eqref{gn1.st} and \eqref{gn2.st} recursively
because $k$ is finite.

In the following, we prove \eqref{fn1.bd} for $m=n+1$ under the assumption that it holds for $m\leq n.$
By letting $t\leq T_\ast$ with $T_\ast>0$ being suitably small  and applying Lemma \ref{k.cyc},  we have 
\begin{align}\label{g1n.itp1}
\sup\limits_{0\leq t\leq T_\ast}\|\RG_1^{n+1}\|_{L^\infty}
\leq& \left(\frac{C}{q}+C\bar{\vps}\right)\sup\limits_{1\leq l\leq k}\sup\limits_{0\leq t\leq T_\ast}\|\RG_1^{n+1-l}\|_{L^\infty}
+C\al\sup\limits_{1\leq l\leq k}\sup\limits_{0\leq t\leq T_\ast}\|\RG_2^{n+1-l}\|_{L^\infty}
\notag\\&+C\sup\limits_{1\leq l\leq k}\sup\limits_{0\leq t\leq T_\ast}(\|\RG_1^{n+1-l}\|^2_{L^\infty}+\|\RG_2^{n+1-l}\|^2_{L^\infty})
+C\|w_qf_0\|_{L^\infty}\notag\\
\leq& C\|w_qf_0\|_{L^\infty}+\left(\frac{C}{q}+C\al+C\bar{\vps}\right)C_0\vps_0+CC^2_0\vps^2_0,
\end{align}
and
\begin{align}\label{g2n.itp1}
\sup\limits_{0\leq t\leq T_\ast}\|\RG_2^{n+1}\|_{L^\infty}
\leq& CT_\ast\sup\limits_{1\leq l\leq k}\sup\limits_{0\leq t\leq T_\ast}\|\RG_1^{n+1-l}\|_{L^\infty}\notag\\&+C(T_\ast+\bar{\vps})
\sup\limits_{1\leq l\leq k}\sup\limits_{0\leq t\leq T_\ast}\|\RG_2^{n+1-l}\|_{L^\infty},
\end{align}
where Lemma \ref{CK} has been used to have the factor $\frac{1}{q}$ in \eqref{g1n.itp1}, and the coefficient $T_\ast$ on the right hand side of \eqref{g2n.itp1} comes from  the last two terms in \eqref{gn2.st} as well as $\CI^{(2)}_2$ and $\CI^{(2)}_3$.
Choosing $T_\ast$ and $\bar{\vps}$ suitably small so that $C(T_\ast+\bar{\vps})\leq\frac{1}{8}$, and using the induction argument, we have from \eqref{g2n.itp1} and \eqref{g1n.itp1} that
\begin{align}
\|\RG_2^{n+1}\|_{L^\infty}\leq& \frac{1}{8^{[\frac{n}{k}]}}
\sup\limits_{0\leq l\leq k}\|\RG_2^{l}\|_{L^\infty}
\notag\\&+\frac{8kCT_\ast}{7}\left\{C\|w_qf_0\|_{L^\infty}+\left(\frac{C}{q}+C\al+C\bar{\vps}\right)C_0\vps_0+CC^2_0\vps^2_0\right\},\ \ n\geq k,\notag
\end{align}
where $[\frac{n}{k}]$ stands for the largest integer no more than $\frac{n}{k}$. Therefore,
\begin{multline}
\|\RG_1^{n+1}\|_{L^\infty}+\|\RG_2^{n+1}\|_{L^\infty}\leq \frac{1}{8^{[\frac{n}{k}]}}
\sup\limits_{0\leq l\leq k}\|\RG_2^{l}\|_{L^\infty}
\notag\\+\frac{8kCT_\ast+7}{7}\left\{C\|w_qf_0\|_{L^\infty}+\left(\frac{C}{q}+C\al+C\bar{\vps}\right)C_0\vps_0+CC^2_0\vps^2_0\right\},\ \ n\geq k.\notag
\end{multline}
This together with \eqref{fs-ktm} implies that \eqref{fn1.bd} holds for $m=n+1$  because
 $q>0$ can be sufficiently large and $\vps_0>0$ as well as $\al>0$ can be suitably small.

Let us now show that $\{[f_1^n,f_2^n]\}_{n=1}^\infty$ converges strongly in the space $\FY_{\al,T_\ast}$.
We denote $[\tilde{\RG}_1^n,\tilde{\RG}_2^n]=[\RG_1^n-\RG_1^{n-1},\RG_2^n-\RG_2^{n-1}]$ with $n\geq1$. Then
$[\tilde{\RG}_1^n,\tilde{\RG}_2^n]$ satisfies
\begin{align}
\pa_t\tilde{\RG}^{n+1}_1&+v_y\pa_y\tilde{\RG}^{n+1}_1-\al v_y\pa_{v_x}\tilde{\RG}^{n+1}_1+2q\al\frac{v_xv_y}{1+|v|^2}\tilde{\RG}^{n+1}_1+\nu_0 \tilde{\RG}^{n+1}_1
\notag\\=&\chi_{M}w_q\CK \left(\frac{\tilde{\RG}^{n}_1}{w_q}\right)-\frac{\al}{2}\sqrt{\mu}v_xv_y\tilde{\RG}^{n}_2
+w_q[H(f_1^{n},f_2^{n})-H(f_1^{n-1},f_2^{n-1})],\notag
\end{align}
\begin{align}
\tilde{\RG}^{n+1}_1(0,y,v)=0,\ \ \tilde{\RG}^{n+1}_1(\pm1,v)|_{v_y\lessgtr0}=\tilde{w}^{-1}_1\dis{\int_{v_y\gtrless0}}\tilde{w}_1\sqrt{2\pi}\mu\tilde{\RG}_1^{n}(\pm1,v)|v_y|dv,\notag
\end{align}
and
\begin{align}
\pa_t\tilde{\RG}^{n+1}_2&+v_y\pa_y\tilde{\RG}^{n+1}_2-\al v_y\pa_{v_x}\tilde{\RG}^{n+1}_2+2q\al\frac{v_xv_y}{1+|v|^2}\tilde{\RG}^{n+1}_2+\nu_0\tilde{\RG}^{n+1}_2
\notag\\&=w_qK\left(\frac{\tilde{\RG}_2^{n}}{w_q}\right)
+(1-\chi_{M})w_q\mu^{-\frac{1}{2}}\CK \tilde{f}^{n}_1,\notag
\end{align}
\begin{align}
\tilde{\RG}^{n+1}_2(0,y,v)=0,\ \ \tilde{\RG}^{n+1}_2(\pm1,v)|_{v_y\lessgtr0}
=\tilde{w}^{-1}_2\dis{\int_{v_y\gtrless0}}\tilde{w}_2\sqrt{2\pi}\mu\RG_2^{n}(\pm1,v)|v_y|dv,\notag
\end{align}
where $\tilde{f}^{n}_1=f^n_1-f_1^{n-1},$ and $\sqrt{\mu}\tilde{\RG}^{n}=\tilde{\RG}_1^{n}+\sqrt{\mu}\tilde{\RG}_2^{n}$.
Then similar to \eqref{g1n.itp1}
and \eqref{g2n.itp1}, one has
\begin{align}\label{tg1n.itp1}
\sup\limits_{0\leq t\leq T_\ast}\|\tilde{\RG}_1^{n+1}\|_{L^\infty}\leq& \left(\frac{C}{q}+C\bar{\vps}\right)\sup\limits_{1\leq l\leq k}\sup\limits_{0\leq t\leq T_\ast}\|\tilde{\RG}_1^{n-l}\|_{L^\infty}+C\al\sup\limits_{1\leq l\leq k}\sup\limits_{0\leq t\leq T_\ast}\|\tilde{\RG}_2^{n-l}\|_{L^\infty}\notag\\&+C\vps_0\sup\limits_{1\leq l\leq k}\sup\limits_{0\leq t\leq T_\ast}(\|\tilde{\RG}_1^{n-l}\|_{L^\infty}+\|\tilde{\RG}_2^{n-l}\|_{L^\infty}),
\end{align}
and
\begin{equation}\label{tg2n.itp1}
\sup\limits_{0\leq t\leq T_\ast}\|\tilde{\RG}_2^{n+1}\|_{L^\infty}\leq CT_\ast\sup\limits_{1\leq l\leq k}\sup\limits_{0\leq t\leq T_\ast}\|\tilde{\RG}_1^{n+1-l}\|_{L^\infty}+C(T_\ast+\bar{\vps})
\sup\limits_{1\leq l\leq k}\sup\limits_{0\leq t\leq T_\ast}\|\tilde{\RG}_2^{n+1-l}\|_{L^\infty}.
\end{equation}
Plugging \eqref{tg1n.itp1} into \eqref{tg2n.itp1} gives
\begin{align}\label{tg2n.itp2}
\sup\limits_{0\leq t\leq T_\ast}\|\tilde{\RG}_2^{n+1}\|_{L^\infty}\leq C(\frac{1}{q}+\al+\vps_0+T_\ast+\bar{\vps})
\sup\limits_{1\leq l\leq k}\sup\limits_{0\leq t\leq T_\ast}\|[\tilde{\RG}_1^{n+1-l},\tilde{\RG}_2^{n+1-l}]\|_{L^\infty}.
\end{align}
By taking $q>0$ sufficiently large and $\al>0$, $\vps_0>$ as well $T_\ast>0$ suitably small, we have from \eqref{tg2n.itp2} and \eqref{tg1n.itp1} that
\begin{align}
\sup\limits_{0\leq t\leq T_\ast}\|\tilde{\RG}_1^{n+1}\|_{L^\infty}+\|\tilde{\RG}_2^{n+1}\|_{L^\infty}\leq& \frac{1}{8^{[\frac{n}{k}]}}
\sup\limits_{0\leq l\leq k}\sup\limits_{0\leq t\leq T_\ast}\|[\tilde{\RG}_1^{l},\tilde{\RG}_2^{l}]\|_{L^\infty},\ n\geq k.\notag
\end{align}
On the other hand, $\sup_{0\leq l\leq k}\sup_{0\leq t\leq T_\ast}\|[\tilde{\RG}_1^{l},\tilde{\RG}_2^{l}]\|_{L^\infty}$ is bounded due to \eqref{fn1.bd}.
Hence it follows that $\{[f_1^n,f_2^n]\}_{n=1}^\infty$ is a Cauchy sequence in the space $\FY_{\al,T_\ast}$, and there is  a unique $[f_1,f_2]\in \FY_{\al,T_\ast}$ such that  $[f_1^{n},f_2^{n}]$ converges  strongly to  $[f_1,f_2]$ as $n\rightarrow+\infty$ and  $[f_1,f_2]$ is the desired local in time solution to the coupled system \eqref{f1}-\eqref{f1bd} and \eqref{f2}-\eqref{f2bd}.
This completes the proof of Theorem \ref{loc.th}.
\end{proof}

\section{Unsteady problem: asymptotic stability and positivity}\label{ust-pro}

This section is about the global existence and large time behavior of solution to the initial boundary value problem \eqref{f}. Recall the decomposition \eqref{def.tpf12} with $f_1$ and $f_2$ satisfying the coupled system \eqref{f1}-\eqref{f1bd} and \eqref{f2}-\eqref{f2bd}. Firstly, we focus on the uniform $L^\infty\cap L^2$ estimates  under the {\it a priori} assumption
\begin{align}\label{aps}
\sup\limits_{s\geq 0 
}\{e^{\la_0s}\|w_qf_1(s,y,v)\|_{L^\infty}+e^{\la_0s}\|w_qf_2(s,y,v)\|_{L^\infty}\}\leq\tilde{\vps},
\end{align}
for a constant $\tilde{\vps}>0$ suitably small, where $\lambda_0>0$ independent of $\alpha$ is to be determined later. And
then we will give the proof of Theorem \ref{ust.mth}.

\subsection{$L^\infty$ estimates}
As in the proof of Theorem \ref{loc.th}, the
$L^\infty$ estimates of $f$ follows from
the uniform $L^\infty$ estimates on $f_1$ and $f_2$.

\begin{lemma}\label{lem.tp1}
Let $0<\lambda_0\leq \frac{\nu_0}{4}$, then under the assumption \eqref{aps}, it holds that
\begin{align}\label{g1.es}
\sup\limits_{0\leq s\leq t}e^{\lambda_0s}\|w_qf_1(s)\|_{L^\infty}
\leq C_{q}\|w_qf_{0}\|_{L^\infty}+C(\al+\tilde{\vps})\sup\limits_{0\leq s\leq t}e^{\lambda_0s}\|w_qf_2(s)\|_{L^\infty},
\end{align}
and
\begin{align}\label{g2sum3}
\sup\limits_{0\leq s\leq t}e^{\lambda_0s}\|w_qf_2(s)\|_{L^\infty}
\leq C\|w_qf_{0}\|_{L^\infty}
+C\sup\limits_{0\leq s\leq t}\|e^{\la_0s}f_2(s)\|,
\end{align}
for any $t\geq 0$.
\end{lemma}

\begin{proof}
For brevity, set
\begin{equation}
\label{def.g12w}
[g_1,g_2](t,y,v)=e^{\la_0t}w_q(v)[f_1,f_2](t,y,v)
\end{equation}
with $\la_0>0$ to be chosen.
Then, the  IBVP for $[g_1,g_2]$ is given as follows:
\begin{align}
\pa_tg_1&+v_y\pa_yg_1-\al v_y\pa_{v_x}g_1+2q\al\frac{v_xv_y}{1+|v|^2}g_1+(\nu_0-\la_0)g_1
\notag\\=&\chi_{M}w_q\CK \left(\frac{g_1}{w_q}\right)-\frac{\al}{2}\sqrt{\mu}v_xv_yg_2
+e^{\la_0t}w_qH(f_1,f_2),\notag
\end{align}
\begin{align}
g_1(0,y,v)=w_qf_0(x,v),\ \ g_1(\pm1,v)|_{v_y\lessgtr0}=\tilde{w}^{-1}_1\dis{\int_{v_y\gtrless0}}\tilde{w}_1\sqrt{2\pi}\mu g_1(\pm1,v)|v_y|dv,\notag
\end{align}
and
\begin{align}
\pa_tg_2&+v_y\pa_yg_2-\al v_y\pa_{v_x}g_2+2q\al\frac{v_xv_y}{1+|v|^2}g_2+(\nu_0-\la_0)g_2\notag\\&=w_qK\left(\frac{g_2}{w_q}\right)
+(1-\chi_{M})w_q\mu^{-\frac{1}{2}}\CK \left(\frac{g_1}{w_q}\right),\notag
\end{align}
\begin{align}
g_2(0,y,v)=0,\ \ g_2(\pm1,v)|_{v_y\lessgtr0}
=\tilde{w}^{-1}_2\dis{\int_{v_y\gtrless0}}\tilde{w}_2\sqrt{2\pi}\mu g_2(\pm1,v)|v_y|dv.\notag
\end{align}
Along the characteristic line \eqref{chl}
the solution to the above problem can be written in the mild form:
\begin{align}\label{g1mild}
g_1(t,y,v)
=&{\bf 1}_{t_1\leq0}e^{-\int_{0}^t\CA_1(\tau,V(\tau))\,d\tau}(w_{q}f_{0})(Y(0),V(0))\notag\\
&+\int_{\max\{0,t_1\}}^{t}e^{-\int_{s}^t\CA_1(\tau,V(\tau))\,d\tau}\left\{\chi_{M}w_q\CK
\left(\frac{g_1}{w_{q}}\right)\right\}(V(s))\,ds\notag\\
&-\al\int_{\max\{0,t_1\}}^{t}e^{-\int_{s}^t\CA_1(\tau,V(\tau))\,d\tau}
\frac{V_x(s)V_y(s)}{2}\sqrt{\mu}(V(s))g_2(V(s))\,ds\notag\\
&+\int_{\max\{0,t_1\}}^{t}e^{-\int_{s}^t\CA_1(\tau,V(\tau))\,d\tau}e^{\frac{\nu_0s}{4}}\left(w_{q}H(f_1,f_2)\right)(V(s))\,ds
+\sum\limits_{n=1}^5\CJ^{(1)}_n,
\end{align}

and
\begin{align}
g_2(t,y,v)
=&\int_{\max\{0,t_1\}}^{t}e^{-\int_{s}^t\CA_1(\tau,V(\tau))\,d\tau}\left\{(1-\chi_{M})\mu^{-\frac{1}{2}}w_{q}\CK
\left(\frac{g_1}{w_{q}}\right)\right\}(V(s))\,ds\notag\\
&+\int_{\max\{0,t_1\}}^{t}e^{-\int_{s}^t\CA_1(\tau,V(\tau))\,d\tau}\left[w_{q}K\left(\frac{g_2}{w_{q}}\right)\right](V(s))\,ds
+\sum\limits_{n=1}^3\CJ^{(2)}_n,\label{g2mild}
\end{align}
where 
\begin{align}
\CA_1(\tau,V(\tau))=\nu_0-\la_0+2q \al \frac{V_y(\tau)V_x(\tau)}{{1+|V(\tau)|^2}}.
\notag
\end{align}
We will take $0<\lambda_0\leq\frac{\nu_0}{4}$ and let $2q\alpha\leq \frac{\nu_0}{4}$ such that $\CA_1(\tau,V(\tau))\geq \frac{\nu_0}{2}$. Moreover, for an integer $k\geq2$, the terms $\CJ^{(1)}_n$ $(1\leq n \leq5)$ in \eqref{g1mild} are given by
\begin{align}
\CJ^{(1)}_1=&\underbrace{{\bf 1}_{t_1>0}e^{-\int_{t_1}^t\CA_1(\tau,V(\tau))d\tau}
\tilde{w}^{-1}_1(V(t_1))}_{\CW^{(1)}_1}\int_{\prod\limits_{j=1}^{k-1}\CV_j}{\bf 1}_{t_k>0}
g_1^{n+1-k}(t_k,y_k,V_{\mathbf{cl}}^{k-1}(t_k))\,d\tilde{\Sigma}^{(1)}_{k-1}(t_k),\notag
\end{align}
\begin{align}
\CJ^{(1)}_2=&\CW^{(1)}_1\sum\limits_{l=1}^{k-1}\int_{\prod\limits_{j=1}^{k-1}\CV_j}{\bf 1}_{t_{l+1}\leq 0<t_l}
(w_{q}f_{0})(Y^{l}_{\mathbf{cl}}(0),V^{l}_{\mathbf{cl}}(0))\,d\tilde{\Sigma}^{(1)}_{l}(0),\notag
\end{align}
\begin{align}
\CJ^{(1)}_3=&\CW^{(1)}_1\sum\limits_{l=1}^{k-1}\int_{\prod\limits_{j=1}^{k-1}\CV_j}\left\{\mathbf{1}_{\{t_{l+1}\leq 0<t_{l}\}}
\int_{0}^{t_l}+\mathbf{1}_{\{t_{l+1}>0\}}\int_{t_{l+1}}^{t_l}\right\}\left\{\chi_{M}w_q\CK
\left(\frac{g_1}{w_{q}}\right)\right\}(Y^{l}_{\mathbf{cl}},V^{l}_{\mathbf{cl}})(s)\,
d\tilde{\Sigma}^{(1)}_{l}(s)ds,\notag
\end{align}
\begin{align}
\CJ^{(1)}_4=&-\al\CW^{(1)}_1\sum\limits_{l=1}^{k-1}\int_{\prod\limits_{j=1}^{k-1}\CV_j}
\left\{\mathbf{1}_{\{t_{l+1}\leq 0<t_{l}\}}
\int_{0}^{t_l}+\mathbf{1}_{\{t_{l+1}>0\}}\int_{t_{l+1}}^{t_l}\right\}
\left\{\frac{v_xv_y}{2}\sqrt{\mu}g_2\right\}
(Y^{l}_{\mathbf{cl}},V^{l}_{\mathbf{cl}})(s)\,d\tilde{\Sigma}^{(1)}_{l}(s)ds,\notag
\end{align}
\begin{align}
\CJ^{(1)}_5=&\CW^{(1)}_1\sum\limits_{l=1}^{k-1}\int_{\prod\limits_{j=1}^{k-1}\CV_j}
\left\{\mathbf{1}_{\{t_{l+1}\leq 0<t_{l}\}}
\int_{0}^{t_l}+\mathbf{1}_{\{t_{l+1}>0\}}\int_{t_{l+1}}^{t_l}\right\}
\left(w_{q}H(f_1,f_2)\right)
(Y^{l}_{\mathbf{cl}},V^{l}_{\mathbf{cl}})(s)\,d\tilde{\Sigma}^{(1)}_{l}(s)ds.\notag
\end{align}
And the terms $\CJ^{(2)}_n$ $(1\leq n \leq3)$ in \eqref{g2mild} are
\begin{align}
\CJ^{(2)}_1=&\underbrace{{\bf 1}_{t_1>0}e^{-\int_{t_1}^t\CA_1(\tau,V(\tau))d\tau}
\tilde{w}_2^{-1}(V(t_1))}_{\CW^{(2)}_1}\int_{\prod\limits_{j=1}^{k-1}\CV_j}{\bf 1}_{t_k>0}
g_2(t_k,y_k,V_{\mathbf{cl}}^{k-1}(t_k))\,d\tilde{\Sigma}^{(2)}_{k-1}(t_k),\notag
\end{align}
\begin{align}
\CJ^{(2)}_2=&\CW^{(2)}_1\sum\limits_{l=1}^{k-1}\int_{\prod\limits_{j=1}^{k-1}\CV_j}\left\{\mathbf{1}_{\{t_{l+1}\leq 0<t_{l}\}}
\int_{0}^{t_l}+\mathbf{1}_{\{t_{l+1}>0\}}\int_{t_{l+1}}^{t_l}\right\}
\left\{w_qK\left(\frac{g_{2}}{w_{q}}\right)\right\}(Y^{l}_{\mathbf{cl}},V^{l}_{\mathbf{cl}})(s)\,
d\tilde{\Sigma}^{(2)}_{l}(s)ds,\notag
\end{align}
\begin{align}
\CJ^{(2)}_3=&\CW^{(2)}_1\sum\limits_{l=1}^{k-1}\int_{\prod\limits_{j=1}^{k-1}\CV_j}
\left\{\mathbf{1}_{\{t_{l+1}\leq 0<t_{l}\}}
\int_{0}^{t_l}+\mathbf{1}_{\{t_{l+1}>0\}}\int_{t_{l+1}}^{t_l}\right\}\left\{(1-\chi_{M})w_{q}\mu^{-\frac{1}{2}}\CK
\left(\frac{g_1}{w_{q}}\right)\right\}(Y^{l}_{\mathbf{cl}},V^{l}_{\mathbf{cl}})(s)\,d\tilde{\Sigma}^{(2)}_{l}(s)ds,\notag
\end{align}
where
\begin{align}
\tilde{\Sigma}^{(i)}_{l}(s)=\prod\limits_{j=l+1}^{k-1}d\si_j e^{-\int_s^{t_l}
\CA_1(\tau,V_{\mathbf{cl}}^l(\tau))d\tau}\tilde{w}_i(v_l)d\si_l \prod\limits_{j=1}^{l-1}\frac{\tilde{w}_i(v_j)}{\tilde{w}_i(V^{j}_{\mathbf{cl}}(t_{j+1}))}
\prod\limits_{j=1}^{l-1}e^{-\int_{t_{j+1}}^{t_j}
\CA_1(\tau,V_{\mathbf{cl}}^l(\tau))d\tau}d\si_j,\ i=1,2.\notag
\end{align}
Consequently, for any $t\geq0$, by applying Lemmas \ref{Ga}, \ref{CK} and \ref{k.cyc} as well as the {\it a priori} assumption \eqref{aps}, we get from \eqref{g1mild} that
\begin{align}
\sup\limits_{0\leq s\leq t}\|g_1(s,y,v)\|_{L^\infty}
\leq&C_{q}\|w_qf_{0}\|_{L^\infty}+\left(\frac{C}{q}+C\bar{\vps}\right)\sup\limits_{0\leq s\leq t}\|g_1(s,y,v)\|_{L^\infty}
\notag\\&+C(\al+\tilde{\vps})\sup\limits_{0\leq s\leq t}\{\|g_1(s,y,v)\|_{L^\infty}+\|g_2(s,y,v)\|_{L^\infty}\},\notag
\end{align}
that gives \eqref{g1.es}.

For $g_2$, 
similar to \eqref{H2sum1}, one has
\begin{align}\label{g2sum1}
|g_{2}(t,y,v)|
\leq& C_{q}e^{-\frac{\nu_0}{2}(t-t_1)}\int_{\max\{t_1,0\}}^te^{-\frac{\nu_0}{2}(t-s)}\int_{\R^3}
{\bf k}_{w}(V(s),v')|g_{2}(s,Y(s;t,y,v),v')|dv'ds\notag\\
&+C_{q}e^{-\frac{\nu_0}{2}(t-t_1)}\sum\limits_{l=1}^{k-1}\int_{\prod\limits_{j=1}^{k-1}\CV_j}
\int_{\max\{t_{l+1},0\}}^{t_l}\int_{\R^3}{\bf k}_{w}(V^{l}_{\mathbf{cl}}(s),v')\notag\\ &\qquad\qquad\qquad\qquad\qquad\times|g_{2}(s,Y^{l}_{\mathbf{cl}}(s;t,y,v),v')|dv'\,d\tilde{\Sigma}^{(2)}_{l}(s)ds+\CP(t),
\end{align}
where
\begin{align}
\CP(t)=&C_{q}\sup\limits_{0\leq s\leq t}\|g_1(s)\|_{L^\infty}
+\tilde{\vps} C_{q}\sup\limits_{0\leq s\leq t}\|g_{2}(s)\|_{L^\infty}.\label{CP}
\end{align}
We now have by iterating \eqref{g2sum1} that
\begin{align}\label{g2sum2}
|g_{2}(t,y,v)|\leq& C_{q}\int_{\max\{t_1,0\}}^te^{-\frac{\nu_0}{2}(t-s)}\int_{\R^3}{\bf k}_{w}(V(s),v')
\int_{\max\{t'_1,0\}}^{s}e^{-\frac{\nu_0}{2}(s-s')}\int_{\R^3}{\bf k}_{w}(\bar{V}(s';Y(s),v'),v'')
\notag\\&\quad\times|g_{2}(s',\bar{Y}(s';Y(s),v'),v'')|~dv''ds'dv'ds\notag\\
&+C_{q}\int_{\max\{t_1,0\}}^te^{-\frac{\nu_0}{2}(t-s)}\int_{\R^3}{\bf k}_{w}(V(s),v')e^{-\frac{\nu_0}{2}(s-t_1')}
\sum\limits_{\ell=1}^{\imath-1}\int_{\prod\limits_{\jmath=1}^{\imath-1}\CV_\jmath}
\int_{\max\{t'_{\ell+1},0\}}^{t'_\ell}\int_{\R^3}
\notag
\\&\qquad\times {\bf k}_{w}(\bar{V}^{\ell}_{\mathbf{cl}}(s';Y(s),v'),v'')|g_{2}(s',\bar{Y}^{\ell}_{\mathbf{cl}}(s';Y(s),v'),v'')|
dv''\,d\tilde{\Sigma}^{(2)}_{\ell}(s')ds'dv'ds\notag\\
&+C_{q}\sum\limits_{l=1}^{k-1}\int_{\prod\limits_{j=1}^{k-1}\CV_j}
\int_{\max\{t_{l+1},0\}}^{t_l}\int_{\R^3}{\bf k}_{w}(V^{l}_{\mathbf{cl}}(s),v')
\int_{\max\{t'_1,0\}}^{s}e^{-\frac{\nu_0}{2}(s-s')}\int_{\R^3}
\notag\\&\qquad\quad\times {\bf k}_{w}(\bar{V}(s';Y^{l}_{\mathbf{cl}}(s),v'),v'')|g_{2}(s',\bar{Y}(s';Y^{l}_{\mathbf{cl}}(s),v'),v'')|~dv''ds'dv'\,d\tilde{\Sigma}^{(2)}_{l}(s)ds
\notag\\&+C_{q}\sum\limits_{l=1}^{k-1}\int_{\prod\limits_{j=1}^{k-1}\CV_j}
\int_{\max\{t_{l+1},0\}}^{t_l}\int_{\R^3}{\bf k}_{w}(V^{l}_{\mathbf{cl}}(s;v),v')e^{-\frac{\nu_0}{2}(s-t_1')}
\sum\limits_{\ell=1}^{\imath-1}\int_{\prod\limits_{\jmath=1}^{\imath-1}\CV_\jmath}
\int_{\max\{t'_{\ell+1},0\}}^{t'_\ell}\int_{\R^3} \notag\\
&\qquad\quad\times {\bf k}_{w}(\bar{V}^{\ell}_{\mathbf{cl}}(s';Y^{l}_{\mathbf{cl}}(s),v'),v'')
|g_{2}(s',\bar{Y}^{\ell}_{\mathbf{cl}}(s';Y^{l}_{\mathbf{cl}}(s),v'),v'')|dv''\,d\tilde{\Sigma}^{(2)}_{\ell}(s')ds'dv'
\,d\tilde{\Sigma}^{(2)}_{l}(s)ds
\notag\\
&+C_{q}\int_{\max\{t_1,0\}}^te^{-\frac{\nu_0}{2}(t-s)}\int_{\R^3}{\bf k}_{w}(V(s),v')\CP(s)dv'ds
\notag\\&+C_{q}\sum\limits_{l=1}^{k-1}\int_{\prod\limits_{j=1}^{k-1}\CV_j}
\int_{\max\{t_{l+1},0\}}^{t_l}\int_{\R^3}{\bf k}_{w}(V^{l}_{\mathbf{cl}}(s),v')\CP(s)dv'\,d\tilde{\Sigma}^{(2)}_{l}(s)ds.
\end{align}
With \eqref{g2sum2}, similar to \eqref{H2sum3}, for sufficiently large $T_0>0$, we have
\begin{align*}
\sup\limits_{0\leq s\leq T_0}\|g_{2}(s)\|_{L^\infty}\leq C\bar{\vps}\sup\limits_{0\leq s\leq T_0}\|g_{2}(s)\|_{L^\infty}
+C(T_0)\sup\limits_{0\leq s\leq T_0}\|f_2(s)\|+C\sup\limits_{0\leq s\leq T_0}\CP(s),
\end{align*}
which together with \eqref{CP} gives
\begin{align}\label{g1g2sum4}
\sup\limits_{0\leq s\leq T_0}\|g_{2}(s)\|_{L^\infty}\leq C\sup\limits_{0\leq s\leq T_0}\|g_{1}(s)\|_{L^\infty}
+C(T_0)\sup\limits_{0\leq s\leq T_0}\|f_2(s)\|.
\end{align}
Next, combining \eqref{g1.es} at $t=T_0$ and \eqref{g1g2sum4}, one has
\begin{align}
\sup\limits_{0\leq s\leq T_0}\|[g_1,g_{2}](s)\|_{L^\infty}\leq& C\|w_q[f_{1}(0,y,v),f_2(0,y,v)]\|_{L^\infty}
+C(T_0)\sup\limits_{0\leq s\leq T_0}\|f_2(s)\|\notag\\
\leq& C\|w_qf_0\|_{L^\infty}
+C(T_0)\sup\limits_{0\leq s\leq T_0}\|f_2(s)\|.\notag
\end{align}
Then it follows that for any $t\in[0,T_0]$,
\begin{align}\label{g1g2decay.p0}
\|w_q[f_1,f_{2}](t)\|_{L^\infty}\leq&
Ce^{-\la_0t}\|w_qf_{0}\|_{L^\infty}
+C(T_0)e^{-\la_0t}\sup\limits_{0\leq s\leq T_0}\|f_2(s)\|.
\end{align}
In particular, we have
\begin{align}\label{g1g2decay}
\|w_q[f_1,f_{2}](T_0)\|_{L^\infty}\leq& Ce^{-\la_0T_0}\|w_q[f_{1}(0,y,v),f_2(0,y,v)]\|_{L^\infty}
+C(T_0)e^{-\la_0T_0}\sup\limits_{0\leq s\leq T_0}\|f_2(s)\|\notag\\
\leq& Ce^{-\la_0T_0}\|w_qf_{0}\|_{L^\infty}
+C(T_0)e^{-\la_0T_0}\sup\limits_{0\leq s\leq T_0}\|f_2(s)\|.
\end{align}
Moreover, \eqref{g1g2decay.p0} can be extended to
\begin{align}\label{g1g2decay.p1}
\|w_q[f_1,f_{2}](t)\|_{L^\infty}\leq& Ce^{-\la_0(t-s)}\|w_q[f_1,f_2](s)\|_{L^\infty}
+C(T_0)e^{-\la_0(t-s)}\sup\limits_{s\leq \tau\leq t}\|f_2(\tau)\|,
\end{align}
for any $t\in[s,s+T_0]$ with $s\geq0.$

Next, for any integer $m\geq 1,$ we can repeat the estimate \eqref{g1g2decay}  in finite times so that
the functions $[f_1,f_2](lT_{0}+s)$ for $l=m-1,m-2,...,0$  satisfy
\begin{align}\label{lif-lift}
\left\|w_q[f_1,f_2](mT_{0})\right\|_{L^\infty}
\leq&Ce^{-\la_0 T_{0}}\left\|w_q[f_1,f_2](\{m-1\}T_0)\right\|_{L^{\infty}}
+C(T_0)e^{-\la_0T_0}\sup\limits_{\{m-1\} T_0\leq s\leq mT_{0}}\|f_2(s)\|\notag \\[2mm]
\leq&Ce^{-\la_0 T_{0}}\left\|w_q[f_1,f_2](\{m-1\}T_0)\right\|_{L^{\infty}}
\notag\\[2mm]&+C(T_0)e^{-\la_0T_0}e^{-\la_0(m-1)T_0}\sup\limits_{\{m-1\} T_0\leq s\leq mT_{0}}\|e^{\la_0s}f_2(s)\|\notag \\[2mm]
\leq &Ce^{-\la_0mT_0}\|w_q[f_1,f_2](0)\|_{L^\infty}
\notag\\&+C(T_0)\sum_{l=0}^{m-1}e^{-m\la_0  T_0}\sup\limits_{\{m-l-1\} T_0\leq s\leq (m-l)T_{0}}\|e^{\la_0s}f_2(s)\|\notag
\\
\leq &Ce^{-\la_0mT_0}\|w_qf_{0}\|_{L^\infty}
+C(T_0)e^{-m\la_0  T_0}\sup\limits_{0\leq s\leq m T_0}\|e^{\la_0s}f_2(s)\|.
\end{align}%
Furthermore, for any $t\geq T_0 $, we can find an integer $m\geq0$ such that $t=m T_0+s$ with $0\leq s\leq T_0$. Then we have, on  one hand,
by \eqref{lif-lift}, that
\begin{equation}\label{g1g2decay.p2}
\begin{split}
\left\|[g_1,g_2](m T_0)\right\|_{\infty }
\leq C\|w_qf_{0}\|_{L^\infty}
+C(T_0)\sup\limits_{0\leq s\leq m T_0}\|e^{\la_0s}f_2(s)\|.
\end{split}
\end{equation}
On the other hand,  \eqref{g1g2decay.p1} implies that
\begin{equation}
\begin{split}
\left\|w_q[f_1,f_2](t)\right\|_{\infty }=&\left\|w_q[f_1,f_2](m T_0+s)\right\|_{L^{\infty}}\\
\leq& Ce^{-\la_0 s}\left\|w_q[f_1,f_2](m T_0)\right\|_{\infty }
+C(T_0)e^{-\la_0 s}\sup\limits_{m T_0\leq \tau\leq m T_0+s}\|e^{\la_0\tau}f_2(\tau)\|,\notag
\end{split}
\end{equation}
which is equivalent to
\begin{equation}\label{g1g2decay.p3}
\begin{split}
\left\|w_q[g_1,g_2](t)\right\|_{\infty }=&\left\|w_q[g_1,g_2](m T_0+s)\right\|_{L^{\infty}}\\
\leq& C\left\|w_q[g_1,g_2](m T_0)\right\|_{\infty }
+C(T_0)\sup\limits_{m T_0\leq \tau\leq m T_0+s}\|e^{\la_0\tau}f_2(\tau)\|.
\end{split}
\end{equation}
Consequently, applying \eqref{g1g2decay.p2} to \eqref{g1g2decay.p3} gives the second estimate \eqref{g2sum3}.
This together with \eqref{g1.es} concludes the $L^\infty$ estimate on $f_1$ and $f_2$, and then it completes the proof of Lemma \ref{lem.tp1}.
\end{proof}

\subsection{$L^2$ estimates}

In order to close the $L^\infty$ estimate in terms of \eqref{g1.es} and \eqref{g2sum3}, we need to deduce the
$L^2$ estimate on $e^{\la_0t}f_2(t,y,v)$. As pointed out in Section \ref{sp.sec}, the key is to obtain the dissipation estimate of the macroscopic component of $f_{2}$ as well as $f_1$ through  the conservation of mass. Therefore, we need resort to the original perturbation  $\sqrt{\mu}g_{\la}:=g_{1}+\sqrt{\mu}g_{2}$ with some abuse of notations
\begin{equation}
\label{def.g12to}
[g_1,g_2](t,y,v):=e^{\la_0t}[f_1,f_2](t,y,v)
\end{equation}
compared to \eqref{def.g12w} in the previous subsection. Note that the velocity weight is no longer needed for the $L^2$ estimates. Indeed, the only time-weighted function $g_{\la}$ satisfies the IBVP
\begin{align}\label{ff}
\left\{\begin{array}{rll}
&\pa_tg_{\la}+v_y\pa_{y}g_{\la}-\al v_y\pa_{v_x}g_{\la}+\frac{\al}{2}v_xv_yg_{\la}+Lg_{\la}-\la_0 g_{\la}
\\[2mm]&\quad=\underbrace{e^{\la_0 t}\Ga(f,f)+\al e^{\la_0 t}\{\Ga(G_1+\al G_R,f)+\Ga(f,G_1+\al G_R)\}}_{\mathscr{H}},\\
&\qquad\qquad t>0,\ y\in(-1,1),\ v=(v_x,v_y,v_z)\in\R^3,\\[2mm]
&\sqrt{\mu}g_{\la}(0,y,v)=f_0(y,v)=F(0,y,v)-F_{st}(y,v),\ y\in(-1,1),\ v\in\R^3,\\[2mm]
&g_{\la} (t,\pm1,v)|_{v_y\lessgtr0}=\sqrt{2\pi \mu}\dis{\int_{v_y\gtrless0}}\sqrt{\mu}g_{\la}(t,\pm1,v)|v_y|dv,\ t\geq0,\ v\in\R^3.
\end{array}\right.
\end{align}
Note that since $\int_{-1}^1\int_{\R^3}f(t,y,v)\sqrt{\mu}dvdy=0$ holds  according to \eqref{id.cons} and \eqref{f}, it is direct to see that 
\begin{align}
\int_{-1}^1\int_{\R^3}g_{\la}(t,y,v)\sqrt{\mu}dvdy=0,\quad \forall\,t\geq 0.
\notag
\end{align}
Next, as in \eqref{abc.def} and \eqref{abc1.def}, we define
$$
\FP_0g_{\la}=(a_\la+\Fb_\la\cdot v+c_\la(|v|^2-3))\sqrt{\mu},\quad \FP_0g_{2}=(a_{\la,2}+\Fb_{\la,2}\cdot v+c_{\la,2}(|v|^2-3))\sqrt{\mu},
$$
and
$$
\bar{\FP}_0g_{1}=(a_{\la,1}+\Fb_{\la,1}\cdot v+c_{\la,1}(|v|^2-3))\mu.
$$
We also use the notation $\Fb_{\la}=(b_{\la}^1,b_{\la}^2,b_{\la}^3).$
Obviously,
\begin{align}\label{m-m-2}
a_{\la}=a_{\la,1}+a_{\la,2},\ \Fb_{\la}=\Fb_{\la,1}+\Fb_{\la,2},\ c_{\la}=c_{\la,1}+c_{\la,2};\quad \int_{-1}^1a_{\la}(t,y)dy=0,\ \forall\,t\geq 0.
\end{align}

As in Section \ref{sp.sec}, we are able to prove the following result in order to capture the macroscopic dissipation of $g_\lambda$.

\begin{lemma}\label{abc.lem}
Under the assumption \eqref{aps}, there exists an instant functional $\CE_{int}(t)$ satisfying
\begin{equation}
\label{abc.lem1}
|\CE_{int}(t)|\leq \|g_{2}\|^2+\|w_qg_{1}\|_{L^\infty}^2
\end{equation}
such that for any $t\geq 0$,
\begin{align}\label{abcla.es}
\frac{d}{dt}\CE_{int}(t)+\la\|[a_\la,\Fb_\la,c_\la]\|^2 \leq& C\|\FP_1g_{2}\|^2+C\|w_qg_{1}\|^2_{L^\infty}\notag\\&+C(\al+\tilde{\vps})\|g_{2}\|^2
+C|\{I-P_\ga\}g_{2}|^2_{2,+}.
\end{align}
\end{lemma}

\begin{proof}
The proof of \eqref{abcla.es} is similar to that of \eqref{abc.es} in Section \ref{sp.sec}. For brevity,  we only show how to derive the $L^2$ estimate on $a_{\la}$.
By letting $\Psi=\Psi(t,y,v)\in C^\infty([0,\infty)\times[-1,1]\times\R^3)$ be a test function and taking the inner product of  \eqref{ff} and $\Psi$, one has
\begin{align}
\frac{d}{dt}(g_{\la},\Psi)&-(g_{\la},\pa_t\Psi)-(v_yg_{\la},\pa_y\Psi)+\langle v_y,(g_{\la}\Psi)(1)\rangle-\langle v_y,(g_{\la}\Psi)(-1) \rangle
\notag\\&+\al (v_yg_{\la},\pa_{v_x}\Psi)+\frac{\al}{2}(v_xv_yg_{\la},\Psi)+((-\la_0+L)g_{\la},\Phi)
=\left(\mathscr{H},\Phi\right).\label{fla.tt}
\end{align}
Choose
$$\Psi=\Psi_{a_{\la}}=v_y\pa_y\phi_{a_{\la}}(t,y) (|v|^2-10)\sqrt{\mu},$$
where
\begin{align}
\pa^2_y\phi_{a_{\la}}=a_{\la},\ \pa_y\phi_{a_{\la}}(\pm1)=0,\ \int_{-1}^1a_\la(y)dy=0.\label{ala}
\end{align}
It follows
\begin{align}\label{epa.la}
\|\phi_{a_{\la}}\|_{H^2}\leq C\|a_{\la}\|.
\end{align}
We now compute the terms in \eqref{fla.tt} one by one. The Cauchy-Schwarz inequality and \eqref{epa.la} directly give
\begin{align}
|(g_{\la},\Psi_{a_\la})|\leq C\|g_{2}\|^2+C\|w_qg_{1}\|^2_{L^\infty},\notag
\end{align}
\begin{align}
\al |(v_yg_{\la},\pa_{v_x}\Psi_{a_\la})|\leq C\al\|g_{2}\|^2+C\al\|w_qg_{1}\|^2_{L^\infty},\notag
\end{align}
\begin{align}
\frac{\al}{2}|(v_xv_yg_{\la},\Psi_{a_{\la}})|\leq C\al\|g_{2}\|^2+C\al\|w_qg_{1}\|^2_{L^\infty},\notag
\end{align}
\begin{align}
|(Lg_{\la},\Psi_{a_{\la}})|\leq \eta\|a_{\la,2}\|^2+C_\eta\|g_{2}\|^2+C_\eta\|w_qg_{1}\|^2_{L^\infty}.\notag
\end{align}
And we have from Lemma \ref{Ga} and the {\it a priori} assumption \eqref{aps} that
\begin{align}
|(\SH,\Psi_{a_{\la}})|\leq C(\al+\tilde{\vps}+\eta)\|a_{\la}\|^2+C_\eta(\al+\tilde{\vps})\{\|w_qg_{2}\|_{L^\infty}^2+\|w_qg_{1}\|^2_{L^\infty}\},\notag
\end{align}
where we have used 
\begin{align}
|(e^{\la_0 t}\Ga(f,f),\Psi_{a_{\la}})|\leq&\eta\|a_{\la}\|^2+C_\eta\int_{-1}^1\int_{\R^3}\left[e^{\la_0 t}\Ga(f,f)|v_y|(|v|^2-10)^2\sqrt{\mu}\right]^2dvdy\notag\\
\leq&\eta\|a_{\la}\|^2+
C_\eta\|w_{q}Q(\mu^{\frac{1}{2}}g_{\la},\mu^{\frac{1}{2}}g_{\la})\|^2_{L^\infty}
\int_{-1}^1\left(\int_{\R^3}w_{-q}\left[|v_y|(|v|^2-10)^2\right]^2\,dv\right)^2dy\notag\\
\leq& \eta\|a_{\la}\|^2+C_\eta\|w_q\mu^{\frac{1}{2}}g_{\la}\|_{L^\infty}^4\notag\\
\leq&  \eta\|a_{\la}\|^2+ C_\eta\tilde{\vps}^2\{\|w_qg_{2}\|_{L^\infty}^2+\|w_qg_{1}\|^2_{L^\infty}\}.\notag
\end{align}
For the second term on the left hand side of \eqref{fla.tt}, from the inner product $\langle\eqref{laf2},\sqrt{\mu}\rangle$,
we have in the weak sense that
\begin{align}
\pa_ta_{\la}+\pa_yb^2_{\la}=0,\notag
\end{align}
which yields
\begin{align}
|(g_{\la},\pa_t\Psi_{a_{\la}})|\leq C\|b^2_{\la}\|^2+C\|\FP_1g_{2}\|^2+C\|w_qg_{1}\|_{L^{\infty}}^2.\notag
\end{align}
In particular, the third term on the left hand side of \eqref{fla.tt} gives the following main contribution
\begin{align}
-(v_yg_{\la},\pa_y\Psi_{a_{\la}})=&-(v_y\FP_0g_{\la},\pa_y\Psi_{\la,2})-(v_y\FP_1g_{\la},\pa_y\Psi_{\la})\notag\\
\geq&5\|a_{\la}\|^2-\eta\|a_{\la}\|^2-C_\eta\|\FP_1g_{\la}\|^2.\notag
\end{align}
The boundary term $\langle v_y,(g_{\la}\Psi_{a_{\la}})(1)\rangle-\langle v_y,(g_{\la}\Psi_{a_{\la}})(-1)\rag$ vanishes due to the boundary condition in \eqref{ala}. Putting all the above estimates for $a_\la$ together, we have
\begin{align}\label{fa.es}
\frac{d}{dt}(g_{\la},\Psi_{a_\la})+\ka\|a_\la\|^2\leq C\|b^y_{\la}\|^2+C\|\FP_1g_{2}\|^2+C\|w_qg_{1}\|_{L^{\infty}}^2+C(\al+\tilde{\vps})\|w_qg_{2}\|_{L^{\infty}}^2.
\end{align}
Next, let
\begin{align*}
\Psi=\Psi_{b_{\la}^i}=\left\{\begin{array}{rll}
&v_yv_{x}\frac{d}{dy}\phi_{b_{\la,1}}(y)\sqrt{\mu},\ i=1,\\[2mm]
&v_yv_{z}\frac{d}{dy}\phi_{b_{\la,3}}(y)\sqrt{\mu},\ i=3,\\[2mm]
&v_y^2(|v|^2-5)\frac{d}{dy}\phi_{b_{\la,2}}(y)\sqrt{\mu},\ i=2,
\end{array}
\right.
\end{align*}
where
\begin{align*}
-\phi''_{b_{\la}^i}=b_{\la}^i,\ \phi_{b_{\la}^i}(\pm1)=0,
\end{align*}
and
\begin{align*}
\Psi=\Psi_{c_\la}=v_y(|v|^2-5)\frac{d}{dy}\phi_{c_\la}(y)\sqrt{\mu},
\end{align*}
where
$$
-\phi_{c_{\la}}''=c_\la,\ \phi_{c_\la}(\pm1)=0.
$$
Similar to \eqref{b.es} and \eqref{c.es}, one can show that
\begin{align}\label{fb.es}
\frac{d}{dt}(g_{\la},\Psi_{\Fb_\la})+\ka\|\Fb_\la\|^2 \leq& C\|c_\la\|^2+C\|\FP_1g_{2}\|^2+C\|w_qg_{1}\|_{L^{\infty}}^2+C(\al+\tilde{\vps})\|w_qg_{2}\|_{L^{\infty}}^2
\notag\\&+C|\{I-P_\ga\}g_{2}|^2_{2,+},
\end{align}
and
\begin{align}\label{fc.es}
\frac{d}{dt}(g_{\la},\Psi_{c_\la})+\ka\|c_\la\|^2 \leq& C\|\FP_1g_{2}\|^2+C\|w_qg_{1}\|_{L^{\infty}}^2+C(\al+\tilde{\vps})\|w_qg_{2}\|_{L^{\infty}}^2
\notag\\&+C|\{I-P_\ga\}g_{2}|^2_{2,+},
\end{align}
respectively. Note that the decomposition $\sqrt{\mu}g_{\la}=g_{1}+\sqrt{\mu}g_{2}$ has been also used to handle the terms involving $\langle v_y,(\{I-P_\ga\}g_{\la}\Psi)(1)\rangle-\langle v_y,(\{I-P_\ga\}g_{\la}\Psi)(-1)\rangle.$

Consequently, by choosing $0<\ka_1\ll \ka_2\ll1$, we have from  $\ka_1\times\eqref{fa.es}+\ka_2\times\eqref{fb.es}+\eqref{fc.es}$ that
\begin{align}\label{fabc}
\frac{d}{dt}\{\ka_1(g_{\la},\Psi_{\Fb_\la})&+\ka_2(g_{\la},\Psi_{\Fb_\la})+(g_{\la},\Psi_{c_\la})\}
+\ka\|[a_\la,\Fb_\la,c_\la\|^2\notag\\
\leq& C\|\FP_1g_{2}\|^2+C\|w_qg_{1}\|_{L^{\infty}}^2+C(\al+\tilde{\vps})\|w_qg_{2}\|_{L^{\infty}}^2+C|\{I-P_\ga\}g_{2}|^2_{2,+}.
\end{align}
Finally, \eqref{abcla.es} follows from \eqref{fabc} by defining
\begin{align}\label{Edef}
\CE_{int}(t)=\ka_1(g_{\la},\Psi_{\Fb_\la})+\ka_2(g_{\la},\Psi_{\Fb_\la})+(g_{\la},\Psi_{c_\la}).
\end{align}
Note that \eqref{abc.lem1} is satisfied. Thus the proof of
Lemma \ref{abc.lem} is completed.
\end{proof}

Now, with Lemma \ref{abc.lem} and Lemma \ref{lem.tp1}, we are ready to complete the
proof of Theorem \ref{ust.mth}.

\begin{proof}[Proof of Theorem \ref{ust.mth}]
The global existence of solution to the problem \eqref{f} follows from the local existence constructed in Section \ref{loc.ex} and the {\it a priori} estimates in the weighted $L^\infty$ space by the continuity argument. Therefore, to prove Theorem \ref{ust.mth}, it remains  to show the uniform estimate \eqref{lif.decay} under the {\it a priori} assumption \eqref{aps}. Indeed, by \eqref{m-m-2}, we can rewrite \eqref{abcla.es} as
\begin{equation}\label{mf2.es}
\frac{d}{dt}\CE_{int}(t)+\la\|\FP_0g_{2}\|^2 \leq C\|\FP_1g_{2}\|^2+C\|w_qg_{1}\|^2_{L^\infty}
+C(\al+\tilde{\vps})\|w_qg_{2}\|_{L^{\infty}}^2+C|\{I-P_\ga\}g_{2}|^2_{2,+},
\end{equation}
where $[g_1,g_2]$ is defined in  \eqref{def.g12to}. For  the $L^2$ estimate on $\FP_1g_{2}$,
note that $g_{2}$ satisfies
\begin{align}\label{laf2}
\pa_tg_{2}&+v_y\pa_yg_{2}-\al v_y\pa_{v_x}g_{2}+(-\la_0+L)g_{2}=(1-\chi_{M})\mu^{-\frac{1}{2}}\CK g_{1},
\end{align}
and
\begin{align}
g_{2}(0,y,v)=0,\ \ g_{2}(\pm1,v)|_{v_y\lessgtr0}
=\sqrt{2\pi \mu}\dis{\int_{v_y\gtrless0}}\sqrt{\mu}g_{2}(\pm1,v)|v_y|dv.\notag
\end{align}
By taking the inner product of \eqref{laf2} and $g_{2}$ with respect to $y$ and $v$ over $(-1,1)\times\R^3,$ one has
\begin{align}\label{laf2.l2}
\frac{d}{dt}\|g_{2}\|^2&+|\{I-P_\ga\}g_{2}|^2_{2,+}+\de_0\|\FP_1g_{2}\|^2
\leq C_\eta\|w_qg_{1}\|^2_{L^\infty}+C(\eta+\la_0)\|w_qg_{2}\|^2.
\end{align}
Let $\tilde{C}>0$ be a constant sufficiently large. By taking the summation of $\tilde{C}\times\eqref{laf2.l2}$ and \eqref{mf2.es} 
we have
\begin{equation}\label{ttf2.es}
\frac{d}{dt}\{\tilde{C}\|g_{2}(t)\|^2+\CE_{int}(t)\}
+\la\|g_{2}\|^2+\la|\{I-P_\ga\}g_{2}|^2_{2,+}\leq C\|w_qg_{1}\|^2_{L^\infty}+C(\al+\tilde{\vps})\|w_qg_{2}\|_{L^{\infty}}^2.
\end{equation}
Denote
$$
\CE(t)=\tilde{C}\|g_{2}(t)\|^2+\CE_{int}(t).
$$
For $\tilde{C}>0$ being large enough, from \eqref{abc.lem1}  there exist constants $C_1>0$ and $C_2>0$ such that for any $t\geq 0$,
\begin{align}\label{lbdE}
2\tilde{C}\|g_{2}(t)\|^2+C_2\|w_qg_{1}(t)\|^2_{L^\infty}\geq\CE(t)\geq \frac{\tilde{C}}{2}\|g_{2}(t)\|^2-C_1\|w_qg_{1}(t)\|^2_{L^\infty}.
\end{align}
Then, from \eqref{ttf2.es} and \eqref{lbdE}, it follows
\begin{align}
\frac{d}{dt}\CE(t)
+\frac{2\la}{\tilde{C}}\CE(t)+\la|\{I-P_\ga\}g_{2}|^2_{2,+}
\leq C\|w_qg_{1}\|^2_{L^\infty}+C(\al+\tilde{\vps})\|w_qg_{2}\|_{L^{\infty}}^2.\notag
\end{align}
Hence
\begin{align}
&\CE(t)+\la\int_0^te^{-\frac{2\la}{\tilde{C}}(t-s)} |\{I-P_\ga\}g_{2}(s)|^2_{2,+}ds\notag\\
&\leq \CE(0)e^{-\frac{2\la}{\tilde{C}} t}
+C\int_0^te^{-\frac{2\la}{\tilde{C}}(t-s)} \|w_qg_{1}(s)\|^2_{L^\infty}ds+C(\al+\tilde{\vps})\int_0^te^{-\frac{2\la}{\tilde{C}}(t-s)} \|w_qg_{2}(s)\|_{L^{\infty}}^2ds\notag\\
&\leq \CE(0)+C\sup\limits_{0\leq s\leq t}\|w_qg_{1}(s)\|^2_{L^\infty}+C(\al+\tilde{\vps})\sup\limits_{0\leq s\leq t}\|w_qg_{2}(s)\|_{L^{\infty}}^2,
\label{lbdE-intap1}
\end{align}
for any $t\geq 0$.
Therefore, by using \eqref{lbdE}  and  \eqref{def.g12to}, it follows from \eqref{lbdE-intap1} that
\begin{align*}
\sup\limits_{0\leq s\leq t}e^{\la_0s}\|f_2(s)\|\leq C\sup\limits_{0\leq s\leq t}e^{\la_0s}\|w_qf_1(s)\|_{L^\infty}+C(\al+\tilde{\vps})\sup\limits_{0\leq s\leq t}e^{\la_0s}\|w_qf_{2}(s)\|_{L^{\infty}}.
\end{align*}
By putting the above estimate back to \eqref{g2sum3} and using the smallness of $\alpha$ and $\tilde{\vps}$, one has
\begin{align}\label{f2.l2fl}
\sup\limits_{0\leq s\leq t}e^{\lambda_0s}\|w_qf_2(s)\|_{L^\infty}
\leq C\|w_qf_{0}\|_{L^\infty}
+C\sup\limits_{0\leq s\leq t}e^{\la_0s}\|w_qf_1(s)\|_{L^\infty}.
\end{align}
Moreover, by plugging \eqref{f2.l2fl} to \eqref{g1.es} and  using the smallness of $\alpha$ and $\tilde{\vps}$ as well as  \eqref{f2.l2fl}, it holds that
\begin{equation*}
\sup\limits_{0\leq s\leq t}e^{\la_0s}\|w_q[f_1,f_2](s)\|_{L^\infty}\leq  C\|w_qf_{0}\|_{L^\infty},
\end{equation*}
which gives \eqref{lif.decay}. Since $\|w_qf_{0}\|_{L^\infty}$ is sufficiently small, the {\it a priori} assumption \eqref{aps} is closed.

Finally, the non-negativity of the global solution constructed above can be proved  similar to \cite{DL-2020}
so that the proof of Theorem \ref{ust.mth} is  completed.
\end{proof}

\section{Appendix}\label{app-sec}

Recall the backward time cycle starting at $(t_0,y_0,v_0)=(t,y,v)$ in \eqref{cyc}, the boundary probability measure $d\si_l$ on $\CV_l$ in \eqref{def.sil} and the product measure $d\Sigma_l(s)$ over $\prod_{j=1}^{k-1}\CV_j$ in \eqref{Sigma}. The following lemma gives an estimate on the measure of the phase space $\Pi _{j=1}^{k-1}\mathcal{V}_{j}$ when there are $k$ times bounce.

\begin{lemma}\label{k.cyc}
For any $\bar{\vps}>0$ and any $T_0>0$, there exists an integer $k_{0}=k_0(\bar{\vps}
,T_{0})$ such that for any integer $k\geq k_{0}$ and for all $(t,y,v)\in[0, T_{0}]\times [-1,1]\times\R^{3}$, it holds
\begin{equation}\label{cy-1}
\int_{\Pi_{l=1}^{k-1}\mathcal{V}_{l}}\mathbf{1}_{\mathcal{\{}%
t_{k}(t,y,v,v_{1},v_{2}...,v_{k-1})>0\}}\Pi _{l=1}^{k-1}d\sigma _{l}\leq
\bar{\vps}.
\end{equation}%
In particular, let $T_{0}>0$ large enough, there exist constants $
C_{1}$ and $C_{2}>0$ independent of $T_{0}$ such that for $k=C_{1}T_{0}^{5/4}$ with a suitable choice of $C_1$ such that $k$ is an integer and for all $(t,y,v)\in \lbrack 0,\infty)\times [-1,1]\times \R^{3}$, it holds
\begin{equation}\label{cy-2}
\int_{\Pi _{j=1}^{k-1}\mathcal{V}_{j}}
\mathbf{1}_{\{t_{k}(t,y,v,v_{1},v_{2},\cdots ,v_{k-1})>0\}}\Pi _{l=1}^{k-1}d\sigma _{l}\leq
\left\{ \frac{1}{2}\right\} ^{C_{2}T_{0}^{5/4}}.
\end{equation}%
Furthermore, for any $q>0$ in the weight function $w_q(v)$, there exist constants $C_3$ and $C_4>0$ independent of $k$ and $T_0$ such that
\begin{equation}\label{cy-3}
\begin{split}
\int_{\Pi _{j=1}^{k-1}\mathcal{V}_{j}}\sum_{l=1}^{k-1}\mathbf{1}_{\{t_{l+1}\leq 0<t_{l}\}}
\int_0^{t_l} d\Sigma_l(s)ds\leq C_3,
\end{split}
\end{equation}
and
\begin{equation}\label{cy-4}
\begin{split}
\int_{\Pi _{j=1}^{k-1}\mathcal{V}_{j}}\sum_{l=1}^{k-1}\mathbf{1}_{\{t_{l+1}>0\}}\int_{t_{l+1}}^{t_l} d\Sigma_l(s)ds\leq
C_4.
\end{split}
\end{equation}
\end{lemma}
\begin{proof}
We only give the proof for \eqref{cy-3}, since \eqref{cy-1}, \eqref{cy-2}, and \eqref{cy-4} can be proved
similarly by using Lemma 23 in \cite[pp. 781]{Guo-2010}.
Recall the definition \eqref{Sigma}. We have
\begin{multline*}
\int_{\Pi _{j=1}^{k-1}\mathcal{V}_{j}}\sum_{l=1}^{k-1}\mathbf{1}_{\{t_{l+1}\leq 0<t_{l}\}}
\int_0^{t_l} d\Sigma_l(s)ds\\
=\int_{\Pi_{j=1}^{k-1}\mathcal{V}_j} \sum_{l=1}^{k-1}\mathbf{1}_{\{t_{l+1}\leq 0<t_{l}\}}
\int_0^{t_l}\prod\limits_{j=l+1}^{k-1}d\si_j e^{-\int_s^{t_l}
\CA^\eps(\tau,V_{\mathbf{cl}}^l(\tau))d\tau}\tilde{w}(v_l)d\si_l \prod\limits_{j=1}^{l-1}\frac{\tilde{w}(v_j)}{\tilde{w}(V^{j}_{\mathbf{cl}}(t_{j+1}))}\\
\times\prod\limits_{j=1}^{l-1}e^{-\int_{t_{j+1}}^{t_j}
\CA^\eps(\tau,V_{\mathbf{cl}}^l(\tau))d\tau}d\si_jds.
\end{multline*}
which using direct calculations, can be bounded by
\begin{multline}
\int_{\Pi_{j=1}^{k-1}\mathcal{V}_j} \mathbf{1}_{t_{k}\leq 0}
\int_0^{t_l}\prod\limits_{j=l+1}^{k-1}d\si_j e^{-\frac{\nu_0}{2}(t_l-s)}\tilde{w}(v_l)d\si_l \prod\limits_{j=1}^{l-1}\frac{\tilde{w}(v_j)}{\tilde{w}(V^{j}_{\mathbf{cl}}(t_{j+1}))}\prod\limits_{j=1}^{l-1}e^{-\frac{\nu_0}{2}(t_j-t_{j+1})}d\si_jds
\\
\leq C\int_{\Pi_{j=1}^{l}\mathcal{V}_j}
\int_0^{t_l} e^{-\frac{\nu_0}{2}(t_1-s)}\tilde{w}(v_l)d\si_l
\prod\limits_{j=1}^{l-1}d\si_jds\leq C.\notag
\end{multline}
Here we have used
$$
\int_{\mathcal{V}_l}\tilde{w}_2(v_l)d\si_l<+\infty,
$$
and
\begin{align}
\frac{\tilde{w}(v_j)}{\tilde{w}(V^{j}_{\mathbf{cl}}(t_{j+1}))}
=&\frac{w_{q}(V^{j}_{\mathbf{cl}}(t_{j+1}))\mu^{\frac{1}{2}}(V^{j}_{\mathbf{cl}}(t_{j+1}))}{w_q(v_j)\mu^{\frac{1}{2}}(v_j)}
=\frac{(1+|V^{j}_{\mathbf{cl}}(t_{j+1})|^2)^q}{(1+|v_j|^2)^q}\cdot e^{\frac{|v_j|^2-|V^{j}_{\mathbf{cl}}(t_{j+1})|^2}{4}}\notag\\
\leq&(1+|V^{j}_{\mathbf{cl}}(t_{j+1})-v_j|^2)^q e^{\frac{\al^2(t_b(v_j)v_{jy})^2}{4}}\leq (1+4\al^2)^qe^{\al^2},\notag
\end{align}
by the Peetre's inequality and the fact that $|V^{j}_{\mathbf{cl}}(t_{j+1})-v_j|=\al|t_b(v_j)v_{jy}|\leq 2\al.$
Then the proof of lemma  is completed.
\end{proof}

\begin{remark}\label{ed-cyc}
The time interval $[0,T_0]$ in Lemma \ref{k.cyc} can be replaced by any interval $[s,t]$ with the length $t-s=T_0$.
In addition, since
$$
\int_{\mathcal{V}_l}\tilde{w}_1(v_l)d\si_l<+\infty,\  q>3/2,
$$
and $\CA^\eps$, $\CA$ and $\CA_1$ have the same lower bound $\nu_0/2$, the statement in Lemma \ref{k.cyc} is
also valid if $\Sigma_l(s)$ is replaced by either $\bar{\Sigma}_l(s)$ or $\tilde{\Sigma}^{(i)}_l(s)$ $(i=1,2)$.
\end{remark}

\noindent {\bf Acknowledgements:}
Renjun Duan's research was partially supported by the General Research Fund (Project No.~14302817) from RGC of Hong Kong and a Direct Grant from CUHK. Shuangqian Liu's research was supported by grants from the National Natural Science Foundation of China (contracts: 11971201 and 11731008), and Hong Kong Institute for Advanced Study No.9360157.
Tong Yang's research was supported by a fellowship award from the Research Grants Council of the
 Hong Kong Special Administrative Region, China (Project no. SRF2021-1S01).


\end{document}